\documentclass{article}
\usepackage[utf8]{inputenc}
\usepackage[utf8]{inputenc}
\usepackage[utf8]{inputenc}
\usepackage{dirtytalk}
\usepackage{tikz}
\usepackage{geometry}
\usepackage{graphicx}
\usepackage{enumitem}
\usepackage{parskip}
\usepackage{amsfonts}
\usepackage{amssymb}
\usepackage{amsmath}
\usepackage{amsthm}
\usepackage{xcolor}
\usepackage{hyperref}
\newlist{myitemize}{itemize}{1}
\setlist[myitemize,1]{leftmargin = 0.5in}
\hypersetup{colorlinks = true, linkcolor = red,  linktocpage}

\theoremstyle{plain}
\newtheorem{thm}{Theorem}[section]
\newtheorem*{thm*}{Theorem}

\newtheorem{cor}[thm]{Corollary}

\theoremstyle{definition}

\newtheorem{conj}[thm]{Conjecture}

\newtheorem{rem}[thm]{Remark}

\title{\textbf{\small{COUNTING THE NUMBER OF $1_{m}$-PREPERIODIC $\mathcal{O}_{K}$-POINTS OF A DISCRETE DYNAMICAL SYSTEM WITH APPLICATIONS FROM ARITHMETIC STATISTICS, VII}}}
\author{\footnotesize{BRIAN KINTU}}

\date{\small{\say{\textit{...What I would say to younger people, the more time you spend on the \textbf{foundations}, the better equipped you will be to make a contribution...Spend a lot of time trying to think about the very basic concepts....}}- Prof. Neil Turok}}

\begin{document}
\maketitle
\begin{abstract}
\footnotesize{In this follow-up article of a multi-part series on (strictly) preperiodic point-counting, we inspect an astonishing relationship between the set of (strictly) $1_{m}$-preperiodic points of a polynomial map $\varphi_{d, c}$ defined by $\varphi_{d, c}(z) = z^d + c$ for all $c, z \in \mathcal{O}_{K}$ and the coefficient $c$, where $K$ is any number field of degree $n\geq 1$, $d>2$ is an integer and $m\in \mathbb{Z}_{\geq 1}$ is any fixed (eventual period). As in \cite{BK22} we again wish to study counting problems that are inspired by torsion point-counting in arithmetic statistics and (strictly) preperiodic point-counting in arithmetic dynamics. In doing so, we then first prove that for any prime $p\geq 3$ and for any fixed $\ell \in \mathbb{Z}_{\geq 1}$ and fixed (eventual period) $m\in \mathbb{Z}_{\geq 1}$, the average number of distinct $1_{m}$-preperiodic integral points of any odd degree map $\varphi_{p^{\ell}, c}$ modulo prime $p\mathcal{O}_{K}$ is unbounded or zero as $c\to \infty$; and so the average behavior here coincide with the average behavior of the number of distinct $m$-periodic integral points of any $\varphi_{p^{\ell}, c}$ modulo $p\mathcal{O}_{K}$ in \cite{BK22}. Inspired further by work of Doyle-Poonen along with conjectural work of Hutz and \textit{abc}(\textit{d})-conditional work of Panraksa on $K$-preperiodic points of even degree map $\varphi_{(p-1)^{\ell}, c}$ for any prime $p\geq 5$ in arithmetic dynamics, we then also prove that for any fixed (eventual period) $m \in \mathbb{Z}_{ \geq 1}$, the average number of distinct $1_{m}$-preperiodic integral points of any $\varphi_{(p-1)^{\ell}, c}$ modulo prime $p\mathcal{O}_{K}$ is unbounded or zero as $c\to \infty$; and so the average behavior here differ from the average behavior of the number of distinct $m$-periodic integral points of any $\varphi_{(p-1)^{\ell}, c}$ modulo $p\mathcal{O}_{K}$ in \cite{BK22}. Finally, we then apply density, polynomial-and number field-counting, and Sato-Tate equidistribution results from arithmetic statistics, and thereby obtaining counting and statistical results on arithmetic objects arising naturally in our polynomial discrete dynamical settings. Together with the previous articles in our series, it might then be fair to view this article as an attempt to contribute to \say{arithmetic statistical dynamics}, which aspires to study counting, distribution, and application of arithmetic objects from discrete dynamical systems.}
\end{abstract}

\begin{center}
\tableofcontents
\end{center}

\newpage 
\begin{center}
    \section{Introduction}\label{sec1}
\end{center}
\noindent
Given any morphism $\varphi: {\mathbb{P}^N(K)} \rightarrow {\mathbb{P}^N(K)} $ of degree $d \geq 2$ defined on a projective space ${\mathbb{P}^N(K)}$ of dimension $N$, where $K$ is a number field. Then for any $n\in\mathbb{Z}$ and $\alpha\in\mathbb{P}^N(K)$, we then call $\varphi^n = \underbrace{\varphi \circ \varphi \circ \cdots \circ \varphi}_\text{$n$ times}$ the $n^{th}$ \textit{iterate of $\varphi$} and call $\varphi^n(\alpha)$ the \textit{$n^{th}$ iteration of $\varphi$ on $\alpha$}. By convention, $\varphi^{0}$ acts as the identity map, i.e., $\varphi^{0}(\alpha) = \alpha$ for every point $\alpha\in {\mathbb{P}^N(K)}$. As before, the everyday philosopher may want to know (quoting here Devaney \cite{Dev}): \say{\textit{Where do points $\alpha, \varphi(\alpha), \varphi^2(\alpha), \ \cdots\ ,\varphi^n(\alpha)$ go as time $n$ becomes large, and what do they do when they get there?}} Now for any given integer $n\geq 0$ and any given point $\alpha\in {\mathbb{P}^N(K)}$, we then call the set consisting of all the iterates $\varphi^n(\alpha)$ the \textit{(forward) orbit of $\alpha$}; which in dynamical systems it's denoted by $\mathcal{O}^{+}(\alpha)$.

As mentioned in \cite{BK22} that one of the main 
goals in arithmetic dynamics is to classify all the points $\alpha\in\mathbb{P}^N(K)$ according to the behavior of their forward orbits $\mathcal{O}^{+}(\alpha)$. In this direction, we recall that any point $\alpha\in {\mathbb{P}^N(K)}$ is called a \textit{preperiodic point of $\varphi$}, whenever $\varphi^{m+n} (\alpha) = \varphi^m(\alpha)$ for some integers $m\in \mathbb{Z}_{\geq 0}$ and $n\in \mathbb{Z}_{\geq 1}$. In this case, we recall that the smallest integers $m\geq 0$ and $n\geq 1$ such that $\varphi^{m+n} (\alpha) = \varphi^m(\alpha)$ are called the \textit{preperiod} and \textit{eventual period} of $\alpha$, resp,. We recall PrePer$(\varphi, {\mathbb{P}^N(K)})$ to be the set of all preperiodic points of $\varphi$; and also note that for any given point $\alpha\in$PrePer$(\varphi, {\mathbb{P}^N(K)})$ the set of all iterates of $\varphi$ on $\alpha$ is called \textit{preperiodic orbit of $\alpha$}. Note that when $m=0$, so that $\varphi^{n}(\alpha) = \alpha $ and so $\alpha$ is a periodic point of period $n$, it then follows Per$(\varphi, {\mathbb{P}^N(K)}) \subseteq$ PrePer$(\varphi, {\mathbb{P}^N(K)})$. In this case when $m=0$, we note that all classical examples \cite{Russo, Poonen} of rational periodic points of any $\varphi_{2,c}$ alluded to in \cite{BK111} are essentially also examples of rational preperiodic points of $\varphi_{2,c}$; and so the interested reader may wish to revisit \cite{Russo, Poonen}. However, we note that it need not be that PrePer$(\varphi, {\mathbb{P}^N(K)})\subseteq$ Per$(\varphi, {\mathbb{P}^N(K)})$. In their 2014 paper \cite{Doyle}, Doyle-Faber-Krumm give nice examples of preperiodic points of any quadratic map $\varphi$ defined over quadratic fields; and hence the reader may see \cite{Poonen, Doyle}. 

Previously in \cite{BK22} we (inspired by work \cite{AM} of Artin-Mazur on periodic orbits and of (BST) on torsion point-counting in arithmetic statistics, along with conjectural work \ref{silver-morton} of Morton-Silverman on $K$-rational periodic point-counting in arithmetic dynamics) proved (a quantitative form of \say{all or nothing}  principle) that the number of distinct $m$-periodic integral points of any $\varphi_{p^{\ell}, c}$ modulo prime $p\mathcal{O}_{K}$ is equal to $p$ or zero; and from which it then followed that the average number of distinct $m$-periodic integral points of any $\varphi_{p^{\ell}, c}$ modulo $p\mathcal{O}_{K}$ is unbounded or zero as $c\to \infty$. Moreover, we then also observed in \cite{BK22} that the expected total number (namely, $p + 0 =p$ for every fixed period $m\in \mathbb{Z}_{\geq 1}$) of distinct $m$-periodic integral points in the whole family of maps $\varphi_{p^{\ell}, c}$ modulo $p\mathcal{O}_{K}$ may not only depend on $p$ (and so depend on deg$(\varphi_{p^{\ell}, c})$) and however be independent of degree $n=[K:\mathbb{Q}]$ and time $m$, but may also grow to infinity whenever degree $p^{\ell}\to \infty$. So now, inspired by work of (BST) on torsion point-counting in arithmetic statistics, along with conjectural work \ref{preper} of Morton-Silverman and work \ref{dopo} of Doyle-Poonen on $K$-rational preperiodic point-counting in arithmetic dynamics, we then revisit the setting in \cite{BK22}. In doing so, we then prove the following (quantitative form of \say{all or nothing}  principle) on every map $\varphi_{p,c}$, which we state later more precisely as Theorem \ref{2.2} and then generalize further as Theorem \ref{2.3}; and moreover restricting on the subring $\mathbb{Z} \subset \mathcal{O}_{K}$ of integers, we then also prove Corollary \ref{cor2.4}:

\begin{thm}\label{BB} 
Let $K\slash \mathbb{Q}$ be any number field of degree $ n \geq 1$ with the ring of integers $\mathcal{O}_{K}$, and in which any fixed prime integer $p\geq 3$ is inert. Let $m\geq 1$ be any fixed integer, and $\varphi_{p, c}$ be a map defined by $\varphi_{p, c}(z) = z^p + c$ for all $c, z\in\mathcal{O}_{K}$. Then the number of distinct $1_{m}$-preperiodic integral points of  $\varphi_{p,c}$ modulo $p\mathcal{O}_{K}$ is $p$ or zero. 
\end{thm}

As noted earlier that a preperiodic point $\alpha$ of any map need not be a periodic point; and when that is the case, we then call such $\alpha$ a \textit{strictly preperiodic point}. So now, as a consequence of the expected total count obtained in Theorem \ref{BB} and [\cite{BK22}, Theorem 2.3] on any $\varphi_{p^{\ell}, c}$ modulo $p\mathcal{O}_{K}$, we then also prove the following corollary on the expected total number of distinct strictly $1_{m}$-preperiodic integral points of $\varphi_{p,c}$ modulo $p\mathcal{O}_{K}$:

\begin{cor}
Let $K\slash \mathbb{Q}$ be any number field of degree $ n \geq 1$ with the ring of integers $\mathcal{O}_{K}$, and in which any fixed prime $p\geq 3$ is inert. Let $m\geq 1$ be any fixed integer, and  $\varphi_{p, c}$ be a map defined by $\varphi_{p, c}(z)$ for all $c, z\in\mathcal{O}_{K}$. Then the expected total number of distinct strictly $1_{m}$-preperiodic integral points of $\varphi_{p,c}$ modulo $p\mathcal{O}_{K}$ is zero.
\end{cor}

Recall further in \cite{BK22} we (again inspired by work \cite{AM} of Artin-Mazur on periodic orbits and of (BST) on torsion point-counting in arithmetic statistics, along with conjectural work \ref{conjecture 3.2.1} of Hutz and \cite{par1} of Panraksa on $K$-rational periodic point-counting in arithmetic dynamics) proved that the number of distinct $m$-periodic integral points of any $\varphi_{(p-1)^{\ell}, c}$ modulo prime $p\mathcal{O}_{K}$ is equal to $1$ or $2$ or $0$; from which it then followed that the average number of distinct $m$-periodic integral points of any $\varphi_{(p-1)^{\ell}, c}$ modulo $p\mathcal{O}_{K}$ is also $1$ or $2$ or $0$ as $c\to \infty$. Moreover, we then also observed in \cite{BK22} that the expected total number (namely, $1 + 2 + 0 =3$ for every fixed odd period $m\in \mathbb{Z}_{\geq 1}$ or namely, $1 + 1 + 2 + 0 =4$ for every fixed even period $m\in \mathbb{Z}_{\geq 2}$) of distinct $m$-periodic integral points in the whole family of maps $\varphi_{(p-1)^{\ell},c}$ modulo $p\mathcal{O}_{K}$ is not only independent of $p$ (and so independent of deg$(\varphi_{(p-1)^{\ell},c})$), degree $n$ and period $m$, but is also a constant $3$ or $4$ even when degree $(p-1)^{\ell}\to \infty$ or $n\to \infty$ or time $m\to \infty$. So now, inspired by work of (BST) on torsion point-counting in arithmetic statistics, and also by work \ref{dopo} of Doyle-Poonen along with conjectural work \ref{conjecture 3.2.1} and \ref{conjecture 3.2.2} of Hutz and [\cite{par2}, Theorem 1.7 and 1.8] of Panraksa on $K$-rational preperiodic point-counting in arithmetic dynamics, we then revisit the setting in Section \ref{sec2}. To that end, we then also prove the following (quantitative form of \say{all or nothing}  principle) on every polynomial map $\varphi_{p-1,c}$, which we state later more precisely as Theorem \ref{3.2} and then generalize further as Theorem \ref{3.3}; and moreover restricting  again on a subring $\mathbb{Z}$, we then also immediately prove Corollary \ref{cor3.4}: 
 
\begin{thm}\label{BB1}
Let $K\slash \mathbb{Q}$ be any number field of degree $n\geq 1$ with the ring of integers $\mathcal{O}_{K}$, and in which any fixed prime $p\geq 5$ is inert. Let $m\geq 1$ be any fixed integer, and $\varphi_{p-1, c}$ be defined by $\varphi_{p-1, c}(z) = z^{p-1} + c$ for all $c, z\in\mathcal{O}_{K}$. The number of distinct $1_{m}$-preperiodic integral points of $\varphi_{p-1,c}$ modulo $p\mathcal{O}_{K}$ is $p$ or $p-1$ or zero.
\end{thm}

\noindent Notice that the count obtained in Theorem \ref{BB1} on the number of distinct $1_{m}$-preperiodic integral points of any $\varphi_{p-1,c}$ modulo $p\mathcal{O}_{K}$ may depend on $p$ (and hence depend on deg$(\varphi_{p-1,c})$) and however be independent of degree $n$ and eventual period $m$ in each of the possibilities. Moreover, we may also observe that the expected total count (namely, $p + 0 = (p-1)+1+0 = p$ for every fixed eventual period $m\in \mathbb{Z}_{\geq 1}$) in Theorem \ref{BB1} on the number of distinct $1_{m}$-preperiodic integral points in the whole family of maps $\varphi_{p-1,c}$ modulo $p\mathcal{O}_{K}$ may not only also depend on $p$ and however be independent of $n$ and time $m$, but may also grow to infinity whenever $p-1\to \infty$. Furthermore, we may also notice that this same dependency and the resulting asymptotic behavior of the expected total number of distinct $1_{m}$-preperiodic integral points of $\varphi_{p-1,c}$ modulo $p\mathcal{O}_{K}$ on $p$, also happens in Theorem \ref{BB} on the expected total number of distinct $1_{m}$-preperiodic integral points of any $\varphi_{p,c}$ modulo $p\mathcal{O}_{K}$.

As a consequence of Thm.\ref{BB1} and [\cite{BK22}, Thm.3.3] on any $\varphi_{(p-1)^{\ell}, c}$ modulo $p\mathcal{O}_{K}$, we also prove the following corollary on the expected total number of distinct strictly $1_{m}$-preperiodic integral points of $\varphi_{p-1,c}$ modulo $p\mathcal{O}_{K}$: 

\begin{cor}
Let $K\slash \mathbb{Q}$ be any number field of degree $n\geq 1$ with the ring of integers $\mathcal{O}_{K}$, and in which any fixed prime $p\geq 5$ is inert. Then the expected total number of distinct strictly $1_{m}$-preperiodic integral points of $\varphi_{p-1,c}$ modulo $p\mathcal{O}_{K}$ is $p-3$ for every odd $m\in \mathbb{Z}_{\geq 1}$ or $p-4$ for every even $m\in \mathbb{Z}_{\geq 2}$. Moreover, the expected total number of distinct strictly $1_{m}$-preperiodic integral points of $\varphi_{p-1,c}$ modulo $p\mathcal{O}_{K}$ is $p-4$ for every $m\in \mathbb{Z}_{\geq 1}$.
\end{cor}

Motivated by \say{counting-application} philosophy in arithmetic statistics along with $\mathbb{Q}_{p}$-preperiodic point-counting work in arithmetic dynamics, we (in \cite{BK333}) revisit the settings in Section \ref{sec2} and \ref{sec3} where we consider a polynomial map iterating on $\mathbb{Z}_{p}\slash p\mathbb{Z}_{p}$; and again with the sole purpose of investigating further the aforementioned relationship. Interestingly, we (in \cite{BK333}) prove counting and asymptotic results that are analogous to results in this article. Motivated further by that same \say{counting-application} philosophy in arithmetic statistics and function field number theory along with work \ref{main} of Benedetto and \cite{DoyPo} of Doyle-Poonen on function field-valued preperiodic point-counting in arithmetic dynamics, we (in that same \cite{BK333}) revisit the setting in Section \ref{sec2} and \ref{sec3} where we consider a polynomial map iterating on $\mathbb{F}_{p}[t]\slash (\pi)$ for any fixed irreducible monic polynomial $\pi\in \mathbb{F}_{p}[t]$ of degree $m\geq 1$. In doing so, we then also prove counting and asymptotic results that are analogous to results proved in Section \ref{sec2} and \ref{sec3}. Moreover, in \cite{BK333, BK3rm} and follow-ups \cite{BK3rm2}, we (thanks to Prof. Ilia Binder!) historically meditate on why human-created systems (e.g., artificial intelligence) with unbounded capabilities that defy our universal sense of order often incite \say{sustainability-related questions} concerned with having such systems in real world. Together with the previous articles in our multi-part series, it might then be fair to view this article and \cite{BK333, BK3rm, BK3rm2} as an attempt to contribute to \say{arithmetic statistical dynamics}, which aspires to study counting, distribution, and application of families of arithmetic objects from discrete dynamical systems. 

In the year 1950, Northcott \cite{North} showed that not only is PrePer$(\varphi, {\mathbb{P}^N(K)})$ finite, but also for a given $\varphi$, the set PrePer$(\varphi, {\mathbb{P}^N(K)})$ can be computed effectively. Forty-five years later, in the year 1995, Morton-Silverman conjectured PrePer$(\varphi, \mathbb{P}^N(K))$ can be bounded in terms of degree $d$ of $\varphi$, degree $D$ of $K$, and dimension $N$ of ${\mathbb{P}^N(K)}$. This conjecture is called the \textit{Uniform Boundedness Conjecture}; which we restate here as the following:

\begin{conj} \label{silver-morton}[\cite{Morton}]
Fix integers $D \geq 1$, $N \geq 1$, and $d \geq 2$. There exists a constant $C'= C'(D, N, d)$ such that for all number fields $K/{\mathbb{Q}}$ of degree at most $D$, and all morphisms $\varphi: {\mathbb{P}^N}(K) \rightarrow {\mathbb{P}^N}(K)$ of degree $d$ defined over $K$, the total number of preperiodic points of a morphism $\varphi$ is at most $C'$, i.e., \#PrePer$(\varphi, \mathbb{P}^N(K)) \leq C'$.
\end{conj}
\noindent A special case of Conjecture \ref{silver-morton} is when  $D = 1$, $N = 1$, and $d = 2$. In this case, if $\varphi$ is a polynomial morphism, then $\varphi$ is a quadratic map defined over $\mathbb{Q}$. Moreover, in the year 1995, Flynn-Poonen-Schaefer \cite{Flynn} conjectured that a quadratic map has no points $z\in\mathbb{Q}$ with exact period more than 3; and which restate here as following:
 
\begin{conj} \label{conj:2.4.1}[\cite{Flynn}, Conjecture 2]
If $n \geq 4$, then there is no quadratic polynomial $\varphi_{2,c }(z) = z^2 + c\in \mathbb{Q}[z]$ with a rational point of exact period $n$.
\end{conj}
Now by assuming Conjecture \ref{conj:2.4.1} and also establishing interesting results on rational preperiodic points, in the year 1998, Poonen \cite{Poonen} then concluded that the total number of rational preperiodic points of any quadratic polynomial $\varphi_{2, c}(z)=z^2 + c$ is at most nine. We restate here formally Poonen's result as the following corollary:
\begin{cor}\label{cor2}[\cite{Poonen}, Corollary 1]
If Conjecture \ref{conj:2.4.1} holds, then $\#$PrePer$(\varphi_{2,c}, \mathbb{Q}) \leq 9$,  for all quadratic maps $\varphi_{2, c}$ defined by $\varphi_{2, c}(z) = z^2 + c$ for all points $c, z\in\mathbb{Q}$.
\end{cor}

On still the same note of preperiodic points, the next natural question that one could ask is whether the aforementioned phenomenon on preperiodic points has been investigated in some other cases, namely, when $D\geq 2$, $N\geq 1$ and $d\geq 2$. In the case $D = d = 2$ and $N = 1$, if $\varphi$ is a polynomial map, then $\varphi$ is a quadratic map defined over $K = \mathbb{Q}(\sqrt{D'})$. In this case, we note that in the year in 2013, Hutz-Ingram \cite{Ingram} showed that the smallest upper bound on the size of PrePer$(\varphi_{2,c}, K)$ is 15. A year later, in 2014, Doyle-Faber-Krumm \cite{Doyle} gave computational evidence on 250000 pairs $(K, \varphi_{2,c})$ which not only established [\cite{Doyle}, Thm. 1.2] as that of \cite{Ingram} on the upper bound of the size of PrePer$(\varphi_{2,c}, K)$, but also covered Poonen's claims in \cite{Poonen} on $\varphi_{2,c}$ over $\mathbb{Q}$. Three years later, in 2018, Doyle \cite{Doy} adjusted the computations in his aforementioned cited work with Faber and Krumm; and after from which he made the following conjecture on any quadratic map over any $K = \mathbb{Q}(\sqrt{D'})$:

\begin{conj}\label{do}[\cite{Doy}, Conjecture 1.4]
Let $K\slash \mathbb{Q}$ be a quadratic field and let $f\in K[z]$ be a quadratic polynomial. \newline Then, $\#$PrePer$(f, K)\leq 15$.
\end{conj}

Recall in \cite{BK1, BK111, BK22} we attempted to understand (on the level of the ring of integers $\mathcal{O}_{K}$ of any number field $K$ of degree $n\geq 1$) the possibility and validity of periodic version of Conjecture \ref{preper}. In this article, we wish to continue with this adventure of hoping to understand (again on the level of $\mathcal{O}_{K}$) the possibility and validity of preperiodic version \ref{preper} of Conject.\ref{silver-morton}. That is, in Sect.\ref{sec2} and \ref{sec3} we consider a polynomial map of any odd prime power degree $d\geq 3$, and also consider polynomial map of even degree $d\geq 4$ defined over $\mathcal{O}_{K}$, where $K$ is any number field of degree $n\geq 1$, resp.; in the attempt of understanding the possibility and validity of \ref{preper}:

\begin{conj} \label{silver-morton 1}($(D,1)$-version of Conjecture \ref{silver-morton})\label{preper}
Fix integers $D \geq 1$ and $d \geq 2$. There exists a constant $C'= C'(D, d)$ such that for all number fields $K/{\mathbb{Q}}$ of degree at most $D$, and all morphisms $\varphi: {\mathbb{P}}^1(K) \rightarrow {\mathbb{P}}^1(K)$ of degree $d$ over $K$, the total number of preperiodic points of $\varphi$ is at most $C'$, that is, \#PrePer$(\varphi, \mathbb{P}^1(K)) \leq C'$.
\end{conj}

\subsection*{History on the Connection Between the Size of PrePer$(\varphi_{d, c}, K)$ and Coefficient $c$}

In the year 1997, Call-Goldstine \cite{Call} proved that the size of PrePer$(\varphi_{2,c},\mathbb{Q})$ can be bounded above in terms of the number of distinct odd primes dividing the denominator of $c\in \mathbb{Q}$. We restate formally this result of Call-Goldstine as the following theorem, in which $GCD(a, e)$ refers to the greatest common divisor of $a$, $e \in \mathbb{Z}$:

\begin{thm}\label{2.3.1}[\cite{Call}, Theorem 6.9]
Let $e>0$ be an integer and let $s$ be the number of distinct odd prime factors of e. Define $\varepsilon  = 0$, $1$, $2$, if $4\nmid e$, if $4\mid e$ and $8 \nmid e$, if $8 \mid e$, respectively. Let $c = a/e^2$, where $a\in \mathbb{Z}$ and $GCD(a, e) = 1$. If $c \neq -2$, then the total number of $\mathbb{Q}$-preperiodic points of $\varphi_{2, c}$ is at most $2^{s + 2 + \varepsilon} + 1$. Moreover, a quadratic map $\varphi_{2, -2}$ has exactly six rational preperiodic points.
\end{thm}

Eight years later after the work of Call-Goldstine, in the year 2005, Benedetto \cite{detto} studied polynomial maps $\varphi$ of arbitrary degree $d\geq 2$ defined over an arbitrary global field $K$, and then established the following theorem on the relationship between the size of the set PrePre$(\varphi, K)$ and the number of bad primes of $\varphi$ in $K$:

\begin{thm}\label{main} [\cite{detto}, Main Theorem]
Let $K$ be a global field, $\varphi\in K[z]$ be a polynomial of degree $d\geq 2$ and $s$ be the number of bad primes of $\varphi$ in $K$. The number of preperiodic points of $\varphi$ in $\mathbb{P}^N(K)$ is at most $O(\text{s log s})$. 
\end{thm}
 
\noindent Since Benedetto's Theorem \ref{main} applies to any polynomial $\varphi$ of arbitrary degree $d\geq 2$ defined over any number field $K$, it then follows that one can immediately apply Benedetto's Theorem \ref{main} to any polynomial $\varphi$ of arbitrary odd or even degree $d> 2$ defined over such $K$ and as a result obtain the upper bound in Thm. \ref{main}. 

Five years after the work \cite{detto}, in the year 2010, Narkiewicz \cite{Narkie1} studied polynomial maps $\varphi_{d,c}$ of odd prime-power degree $d = p^{\ell}$ defined over a totally complex extension $K\slash \mathbb{Q}$ of degree $n$ where $K$ does not contain $p$-\text{th} roots of unity, and then established the following theorem on the length of $K$-cycles of any such map $\varphi_{d,c}$:

\begin{thm} \label{theorem 3.2.1}[\cite{Narkie1}, Theorem]
Let $K$ be a totally complex extension of $\mathbb{Q}$ of degree $n>1$, denote by $R$ its ring of integers and $D$ be the maximal order of a primitive root of unity contained in $K$. Let $p$ be a prime not dividing $D$, and put $F(X) = X^n + c\in K[X]$ with $n = p^k$ with $k\geq 1$ and $c\neq 0$. Then the lengths of cycles of $F$ in $K$ are bounded by a constant $B = B(K, p)$. If $p>2^n$, then this constant can be taken to be $n2^{n+1}(2^n-1)$.
\end{thm}\noindent Now recall in arithmetic dynamics, and more generally in classical dynamical systems that we can always identify any $K$-orbit of any map, say $\varphi_{p,c}$, with any $K$-cycle of the same map. So then, if we were working under the assumptions in Theorem \ref{theorem 3.2.1}, then by Theorem \ref{theorem 3.2.1} the total number of distinct points in any $K$-orbit is bounded by a constant $B$ depending on only $K$ and $p$; and moreover $B = n2^{n+1}(2^n -1)$ whenever $p>2^n$.

Three years after \cite{Narkie}, in 2015, Hutz \cite{Hutz} developed an algorithm determining effectively all $\mathbb{Q}$-preperiodic points of a morphism defined over a given number field $K$; from which he then made the following conjecture: 

\begin{conj} \label{conjecture 3.2.1}[\cite{Hutz}, Conjecture 1a]
For any integer $n > 2$, there is no even degree $d > 2$ and no point $c \in \mathbb{Q}$ such that the polynomial map $\varphi_{d, c}$ has rational points of exact period $n$.
Moreover, \#PrePer$(\varphi_{d, c}, \mathbb{Q}) \leq 4$. 
\end{conj}

\begin{rem}\label{rem1.14}
If Conjecture \ref{conjecture 3.2.1} held, it would then also follow that the total number of rational integral preperiodic points of any $\varphi_{d, c}(x)$ of even degree $d\geq 4$ is bounded above by $4$. Moreover, since the monic polynomial $\varphi_{d,c}(x)\in \mathbb{Z}[x]$ has good reduction modulo $p$ and since we also know from [\cite{Sil}, Cor. 2.20] that good reduction modulo $p$ preserves preperiodicity, it would then also follow that the total number of rational integral preperiodic points of any $\varphi_{d,c}$ modulo $p$ of even degree $d\geq 4$ is also bounded above by a constant equal to $4$. But now (as in \cite{BK111}) the issue is that we don't know (as to the author's knowledge) whether Conjecture \ref{conjecture 3.2.1} holds or not, let alone whether $4$ is the correct upper bound on the total number of $\mathbb{Q}$-preperiodic of any $\varphi_{d, c}(x)$ of even degree $d> 2$. Note that in the year 2022, Panraksa not only conditioned on the \textit{abc}-conjecture and then proved in [\cite{par2}, Theorem 1.7] that \#PrePer$(\varphi_{d, c}, {\mathbb{P}^1(\mathbb{Q})}) \leq 4$ for sufficiently large even degree $d$, but also assumed the \textit{explicit abc}-conjecture and then proved that \#PrePer$(\varphi_{d, c}, {\mathbb{P}^1(\mathbb{Q})}) \leq 4$ for even degree $d\geq 7$.
\end{rem}

Using the same algorithm in his work \cite{Hutz}, Hutz also made the following conjecture (which he also referred to as \say{Generalized Poonen}) on (the total number) rational preperiodic points of any map $\varphi_{d, c}$ of degree $d\geq 2$:  

\begin{conj} \label{conjecture 3.2.2}[\cite{Hutz}, Conjecture 1]
For any integer $n > 3$, there is no degree $d \geq 2$ and no point $c \in \mathbb{Q}$ such that the map $\varphi_{d, c}$ has rational points of exact period $n$. For maps of the form $\varphi_{d,c}$, we have \#PrePer$(\varphi_{d, c}, \mathbb{Q}) \leq 9$. 
\end{conj}

\begin{rem}\label{rem1.16}
As in Rem.\ref{rem1.14} we also note that if Conjecture \ref{conjecture 3.2.2} held, it would then also follow that the total number of rational integral preperiodic points of any $\varphi_{d, c}(x)$ of any degree $d\geq 2$ is bounded above by $9$. Moreover, since $\varphi_{d,c}(x)\in \mathbb{Z}[x]$ has good reduction modulo $p$, it would then also follow that the total number of rational integral preperiodic points of any $\varphi_{d,c}$ modulo $p$ of degree $d\geq 2$ is also bounded above by a constant equal to $9$. But now since Conjecture \ref{conj:2.4.1} is still open (and so the upper bound $9$ in Corollary \ref{cor2} need not be obtained) and moreover since the upper bound in Conjecture \ref{conjecture 3.2.1} is also not yet known to be the correct one or not, we then also note that the issue here is that we don't know whether Conjecture \ref{conjecture 3.2.2} holds or not, let alone whether $9$ is the correct upper bound on the total number of rational preperiodic of any $\varphi_{d, c}(x)$ of any $d> 2$. It's worth mentioning that in establishing the computational evidence that would lead him to formulate his Conject.\ref{conjecture 3.2.2} and \ref{conjecture 3.2.1}, Hutz considered rational points $c$ of bounded height. But now one may wonder why he never incorporated that bounded height information in stating his Conjecture \ref{conjecture 3.2.2} and \ref{conjecture 3.2.1}. Perhaps Hutz might've thought that regardless of that arithmetic information, the upper bounds in \ref{conjecture 3.2.2} and \ref{conjecture 3.2.1} must hold. 
\end{rem}

In the year 2020, Doyle-Poonen \cite{DoyPo} proved that the number of $K$-rational preperiodic points of $z^d + c$ having a bounded eventual period, is uniformly bounded. We restate this result of Doyle-Poonen as the following:

\begin{thm}\label{dopo}[\cite{DoyPo}, Theorem 1.10]
Fix integers $d\geq 2$, $D\geq 1$, and $N\geq 1$. Then there exists $B=B(d, D, N)>0$ such that, for every number field $K$ satisfying $[K:\mathbb{Q}]\leq D$ and every $c\in K$, the number of preperiodic points of $z^d + c$ in $K$ with eventual period at most $N$ is at most $B$.
\end{thm}

\begin{rem}
Observe in Doyle-Poonen's Theorem \ref{dopo} that for fixed $D\geq 1$, then the number of preperiodic points of $z^d + c$ in $K$ with eventual period at most $N$ is at most a positive constant $B=B(d,D, N)$. But then it also follows that the number of preperiodic points of $z^d + c$ in $\mathcal{O}_{K}\subset K$ with eventual period at most $N$ is bounded above by a positive constant $B=B(d,D, N)$. Moreover, since the monic polynomial $\varphi_{d,c}(x)=x^d + c\in \mathcal{O}_{K}[x]$ has good reduction modulo any prime ideal $\mathfrak{p}$ and also since we know from [\cite{Silverman}, Corollary 2.20] that good reduction modulo $\mathfrak{p}$ preserves preperiodicity, it then also follows that the number of preperiodic points of the reduced monic polynomial $\varphi_{d,c}(x)$ modulo $\mathfrak{p}$ with eventual period at most $N$ is bounded above by $B(d, D, N)$.  
\end{rem}

A year later after the work \cite{DoyPo}, Looper \cite{Loo} proved that conditioning on the \textit{abc}-conjecture for number fields $K$ yields that $\varphi_{d,c}(x)\in K[x]$ of degree $d\geq 5$ has at most $B=B(d, K)$ preperiodic points in $K$. Moreover, for $2\leq d\leq 4$, Looper also proved [\cite{Loo}, Theorem 1.2] that conditioning on the \textit{abcd}-conjecture yields that $\varphi_{d,c}$ has at most $B=B(d, K)$ in $K$. We here restate this number field version of [\cite{Loo}, Theorem 1.2] as the following:

\begin{thm} \label{Loop}[\cite{Loo}, Theorem 1.2 (Uniform Boundedness for polynomials $z^d + c$)]
Let $K$ be a number field, and let $d\geq 2$. Let $\varphi_{d,c}(z)=z^d + c\in K[z]$. If $d\geq 5$, assume the \textit{abc}-conjecture for $K$. If $2\leq d \leq 4$, assume the $\textit{abcd}$-conjecture (Conject. 2.1). There is a $B=B(d, K)$ such that $\varphi_{d,c}$ has at most $B$ preperiodic points in $K$.
\end{thm}

A year later after Thm.\ref{Loop}, Panraksa conditioned on \textit{abcd}-conjecture and proved in [\cite{par2}, Thm. 4] the existence of an absolute constant $B_{0}$ such that \#PrePer$(\varphi_{d, c}, \mathbb{P}^1(\mathbb{Q})) \leq B_{0}$ for all $\varphi_{d,c}(x)\in \mathbb{Q}[x]$ of degree $d\geq 2$.

In the year 2021, Sadek [\cite{Sad1}, Theorem 5] proved that infinitely many parametric families of polynomials of the form $\varphi_{d, c(t)}(x)=x^d + c(t)\in K(t)[x]$  where $K$ is a number field and $d\geq 2$ is any integer, have no rational preperiodic points for any $t\in K$. Earlier in the year 2018, Sadek had also proved [\cite{Sad2}, Corollary 4.5] (using a non-elementary divisibility criteria) that infinitely many polynomials $\varphi_{d, c}(x)\in K[x]$ with $c=c_{1}\slash c_{2}$ where $c_{1}, c_{2}$ are relatively prime in the ring of integers $\mathcal{O}_{K}$ and $c_{2}$ is not a $d^{th}$-power, have no periodic points of any period. 

Fourteen years later after the work \cite{detto}, in the year 2021, Le Boudec-Mavraki proved [\cite{Bou}, Theorem 1.1] (using arithmetic dynamics insights on canonical heights and the Weil height, combined with some classical geometry-of-numbers methods) that the average number of rational preperiodic points of random polynomials of fixed degree $d\geq 2$ defined over $\mathbb{Q}$, with affine conjugacy classes ordered by height, is equal to zero. Inspired by advances in arithmetic statistics, in the year 2023, Olechnowicz [\cite{Ole}, Page 5, Eqn 1.6] not only formulated a \say{finiteness averaged uniform boundedness} version of Conject.\ref{silver-morton} and proved some special cases of it, but he also conditioned on a weak form of \ref{silver-morton} and calculated the average number of $K$-rational preperiodic points of a one-parameter family of rational maps defined over $K$, and additionally gave some examples supporting his calculations unconditionally and also gave new estimates on the average number of $\mathbb{Q}$-preperiodic points of $\varphi_{2,c}$.

In the case of polynomial dynamical systems defined over finite fields $\mathbb{F}_{p^r}$, more recently, Andersen-Garton \cite{Gart} have computed statistics of strictly preperiodic points of dynamical systems induced by polynomials $\varphi_{d,c}(x)=x^d + c\in \mathbb{F}_{p^r}[x]$ of degree $d\geq 2$; in which the authors not only establish effective upper bounds for the proportion of $\mathbb{F}_{q}$ lying in a given $W_{p,r,n}(f)=f^n(\mathbb{F}_{p^r})\backslash f^{n+1}(\mathbb{F}_{p^r})$ where $f$ is a polynomial defined over $\mathbb{F}_{p^r}$, but also in generalizing their definition of $W_{p,r,n}(f)$, they prove upper and lower bounds for the resulting averages. Later in the year 2025, Garton proved [\cite{Gart2}, Theorem 1.1] that the expected value of the proportion of points in $\mathbb{P}^1(\mathbb{F}_{p^r})$ that is periodic with respect to a function $\phi$ ranging over polynomials or rational functions with coefficients in a finite field $\mathbb{F}_{p^r}$ of degree $d\geq 2$, tends to zero as $p^r$ grows in a set of prime powers coprime to $d!$.

\noindent Once again, the purpose of this follow-up article of a multi-part series is to investigate further (using elementary arguments) the above connection and also try to understand Conjecture \ref{preper}, in the case of odd degree-$p^{\ell}$ polynomial maps $\varphi_{p^{\ell}, c}$, and in the case of even degree-$(p-1)^{\ell}$ polynomial maps $\varphi_{(p-1)^{\ell}, c}$ defined over $\mathcal{O}_{K}$, where $K$ is any number field of degree $n\geq 1$, $p>2$ is any prime and $\ell \in \mathbb{Z}_{\geq 1}$ is any fixed integer; and doing so from a spirit that's truly inspired and guided by some of the many striking developments in arithmetic statistics.

\section{On Number of $1_{m}$-PrePeriodic Integral Points of any Family of Polynomial Maps $\varphi_{p^{\ell},c}$}\label{sec2}

In this section, we wish to count the number of distinct $1_{m}$-preperiodic integral points of any $\varphi_{p^{\ell},c}$ modulo $p\mathcal{O}_{K}$, where $p\geq 3$ is any fixed prime, and $\ell, m \in \mathbb{Z}_{\geq 1}$ are any fixed integers. To do so, we let $p\geq 3$ be any prime, $c\in \mathcal{O}_{K}$ be any point, $\ell, m \in \mathbb{Z}_{\geq 1}$ be any fixed integers, and then define $1_{m}$-preperiodic point-counting function 
\begin{equation}\label{N_{c}}
N_{c}^{(1_{m})}(p) := \# \Biggl\{ z\in \mathcal{O}_{K}\slash p\mathcal{O}_{K} : \begin{aligned} \varphi_{p^{\ell},c}^{m}(z) -z \not \equiv 0 \ \text{(mod $p\mathcal{O}_{K}$)} \\ \ \varphi_{p^{\ell},c}^{1+m}(z) - \varphi_{p^{\ell},c}(z) \equiv 0 \ \text{(mod $p\mathcal{O}_{K}$)} \end{aligned} \Biggr\}.
\end{equation}\noindent Setting $\ell =1$, and so the map $\varphi_{p^{\ell}, c} = \varphi_{p,c}$, we then first prove the following theorem and its generalization \ref{2.2}:

\begin{thm} \label{2.1}
Let $K\slash \mathbb{Q}$ be any number field of degree $n \geq 1$ with the ring of integers $\mathcal{O}_{K}$, and in which $3$ is inert. Let  $\varphi_{3, c}$ be a cubic map defined by $\varphi_{3, c}(z) = z^3 + c$ for all $c, z\in\mathcal{O}_{K}$, and  $N_{c}^{(1_{m})}(3)$ be the number defined as in \textnormal{(\ref{N_{c}})}. Then $N_{c}^{(1_{m})}(3)=3$ for any coefficient $c\in 3\mathcal{O}_{K}$; otherwise $N_{c}^{(1_{m})}(3) = 0$ for any coefficient $c \not \in 3\mathcal{O}_{K}$.
\end{thm}

\begin{proof}
Let $f(z)= \varphi_{3,c}^{1+m}(z)-\varphi_{3,c}(z) = \varphi_{3,c}^{1+m}(z) - z^3-c$, and so $f(z)= \varphi_{3,c}^{1+m}(z) - z^3 - c$. Now applying the multinomial theorem repeatedly on the term $\varphi_{3,c}^{1+m}(z)$, we then obtain that $\varphi_{3,c}^{1+m}(z)$ is a monic polynomial in $z$ of degree $3^{1+m}$ with integral coefficients in multiples of $c$. Thus, we may then write $\varphi_{3,c}^{1+m}(z) = z^{3^{m+1}} + h(z) + c$, where $h(z)$ is a non-constant polynomial in $z$ of degree deg$(h)<3^{m+1}$ with integral coefficients in multiples of $c$. But now $f(z)= z^{3^{m+1}} + h(z) +c - z^3 - c$, and so $f(z) = z^{3^{m+1}}-z^3 + h(z)$. Now for every coefficient $c\in 3\mathcal{O}_{K}$, reducing $f(z)$ modulo prime ideal $3\mathcal{O}_{K}$, it then follows $f(z)\equiv z^{3^{m+1}} - z^3$ (mod $3\mathcal{O}_{K}$), since also $h(z)\in c\mathcal{O}_{K}[z]$ and thus $h(z)\equiv 0$ (mod $3\mathcal{O}_{K}$); and so now $f(z)$ modulo $3\mathcal{O}_{K}$ is a polynomial defined over a finite field $\mathcal{O}_{K}\slash 3\mathcal{O}_{K}$ of order $3^{[K:\mathbb{Q}]} = 3^n$. Now recall from a well-known fact that every subfield of a finite field $\mathcal{O}_{K}\slash 3\mathcal{O}_{K}$ is of order $3^r$ for some positive integer $r\mid n$, we then obtain the inclusion $\mathbb{F}_{3}\hookrightarrow \mathcal{O}_{K}\slash 3\mathcal{O}_{K}$ of fields, where $\mathbb{F}_{3}$ is a field of order $3$. Moreover, it is well-known that $z^3 = z$ for every element $z\in \mathbb{F}_{3}\subset \mathcal{O}_{K}\slash 3\mathcal{O}_{K}$. But now we note $z^{3^{m+1}}= (z^3)^{3^{m}} = (z^3)^{3^{m-1}} = z^{3^{m-1}}$ for every element $z\in \mathbb{F}_{3}$. Now since $m\geq 1$ and so $m-1\geq 0$, then if $m-1 = 0$ and so $z^{3^{m-1}} = z$, it then follows $z^{3^{m+1}} = z$  for every point $z\in \mathbb{F}_{3}$; and so the reduced polynomial $f(z)\equiv 0$ (mod $3\mathcal{O}_{K}$) for every point $z\in \mathbb{F}_{3}\subset \mathcal{O}_{K}\slash 3\mathcal{O}_{K}$. Otherwise, if $m-1 > 0$, then since $m$ is a fixed integer, we may then continue performing the above procedure of decreasing the exponent $m-1$ of $ z^{3^{m-1}} = z^{3^{m+1}}$ for every element $z\in \mathbb{F}_{3}$, until $m-1$ is eventually equal to zero; and from which we then again obtain that $f(z)\equiv 0$ (mod $3\mathcal{O}_{K}$) for every point $z\in \mathbb{F}_{3}\subset \mathcal{O}_{K}\slash 3\mathcal{O}_{K}$. But now we then conclude $N_{c}^{(1_{m})}(3) = 3$.

Finally, we now show $N_{c}^{(1_{m})}(3) = 0$ for every coefficient $c\not \equiv 0$ (mod $3\mathcal{O}_{K}$) and for every fixed integer $m\in \mathbb{Z}_{\geq 1}$. To do so, let's for the sake of a contradiction, suppose $f(z)\equiv 0$ (mod $3\mathcal{O}_{K}$) for some point $z\in \mathcal{O}_{K}\slash 3\mathcal{O}_{K}$ and for every coefficient $c\not \equiv 0$ (mod $3\mathcal{O}_{K}$) and for every fixed $m\in \mathbb{Z}_{\geq 1}$. Now recall from earlier $f(z)= z^{3^{m+1}} -z^3 + h(z)$ where $h(z)\in c\mathcal{O}_{K}[z]$, it then also follows by the above supposition that $z^{3^{m+1}} - z^3 + h(z) \equiv 0$ (mod $3\mathcal{O}_{K}$) for some $z\in \mathcal{O}_{K}\slash 3\mathcal{O}_{K}$ and for every $c\not \equiv 0$ (mod $3\mathcal{O}_{K}$) and every fixed $m\in \mathbb{Z}_{\geq 1}$. So now, since we know from earlier $z^{3^{m+1}} = z^3$ for every $z\in \mathbb{F}_{3}\subset \mathcal{O}_{K}\slash 3\mathcal{O}_{K}$ and every fixed $m\in \mathbb{Z}_{\geq 1}$, we may rewrite $z^{3^{m+1}} - z^3 + h(z) \equiv 0$ (mod $3\mathcal{O}_{K}$) for some $z\in \mathcal{O}_{K}\slash3\mathcal{O}_{K}$ and every $c\not \equiv 0$ (mod $3\mathcal{O}_{K}$) to then obtain $h(z)\equiv 0$ (mod $3\mathcal{O}_{K}$) for some $z\in \mathbb{F}_{3}\subset \mathcal{O}_{K}\slash3\mathcal{O}_{K}$ and every $c\not \equiv 0$ (mod $3\mathcal{O}_{K}$) and every fixed $m\in \mathbb{Z}_{\geq 1}$. Moreover, looking at the multinomial expansion of $(\varphi_{3,c}^{m}(z))^3$, we then obtain $h(z)\equiv \sum_{i=1}^{m}c^{3^{m+1-i}}$ (mod $3\mathcal{O}_{K}$); and so $\sum_{i=1}^{m}c^{3^{m+1-i}} \equiv 0$ (mod $3\mathcal{O}_{K}$). But now we note that $\sum_{i=1}^{m}c^{3^{m+1-i}} \equiv 0$ (mod $3\mathcal{O}_{K}$) can also happen if $c \equiv 0$ (mod $3\mathcal{O}_{K}$); and so contradicting $c\not \equiv 0$ (mod $3\mathcal{O}_{K}$). Otherwise, suppose $f(\alpha)\equiv 0$ (mod $3\mathcal{O}_{K}$) and so $\alpha^{3^{m+1}} - \alpha^3 + h(\alpha)\equiv 0$ (mod $3\mathcal{O}_{K}$) for some point $\alpha \in \mathcal{O}_{K}\slash 3\mathcal{O}_{K} \setminus \mathbb{F}_{3}$ and for every $c\not \equiv 0$ (mod $3\mathcal{O}_{K}$) and every fixed $m\in \mathbb{Z}_{\geq 1}$. Now since (from a well-known fact) $z^{3^{m+1}}=z$ for every $z\in \mathbb{F}_{3^{m+1}}\subset \mathcal{O}_{K}\slash 3\mathcal{O}_{K}$ and for every fixed $(m+1) \mid n$, then if a root $\alpha \in \mathbb{F}_{3^{m+1}}\subset \mathcal{O}_{K}\slash 3\mathcal{O}_{K}\setminus \mathbb{F}_{3}$ and so $\alpha^{3^{m+1}}=\alpha$, we then obtain $\alpha - \alpha^3 + h(\alpha)\equiv 0$ (mod $3\mathcal{O}_{K}$) for every $c\not \equiv 0$ (mod $3\mathcal{O}_{K}$) and every fixed $m\in \mathbb{Z}_{\geq 1}$; and since $h(\alpha)\equiv \sum_{i=1}^{m}c^{3^{m+1-i}}$ (mod $3\mathcal{O}_{K}$), we then also have $\alpha- \alpha^3 +\sum_{i=1}^{m}c^{3^{m+1-i}} \equiv 0$ (mod $3\mathcal{O}_{K}$). But now we note that the congruence $\alpha- \alpha^3 +\sum_{i=1}^{m}c^{3^{m+1-i}}\equiv 0$ (mod $3\mathcal{O}_{K}$) can also happen if $\alpha - \alpha^3 \equiv 0$ (mod $3\mathcal{O}_{K}$) and also $\sum_{i=1}^{m}c^{3^{m+1-i}} \equiv 0$ (mod $3\mathcal{O}_{K}$). Moreover, we also note $\alpha - \alpha^3 \equiv 0$ (mod $3\mathcal{O}_{K}$) for every point $\alpha\equiv \pm1, 0$ (mod $3\mathcal{O}_{K}$), which also happened earlier when $c\equiv 0$ (mod $3\mathcal{O}_{K}$); and which then also contradicts the earlier condition $c\not \equiv 0$ (mod $3\mathcal{O}_{K}$). Otherwise, if a root $\alpha \not \in \mathbb{F}_{3^{m+1}}$ and since $h(\alpha)\equiv \sum_{i=1}^{m}c^{3^{m+1-i}}$ (mod $3\mathcal{O}_{K}$), then $\alpha^{3^{m+1}} - \alpha^3 +\sum_{i=1}^{m}c^{3^{m+1-i}} \equiv 0$ (mod $3\mathcal{O}_{K}$). But then as before $\alpha^{3^{m+1}} - \alpha^3 +\sum_{i=1}^{m}c^{3^{m+1-i}}  \equiv 0$ (mod $3\mathcal{O}_{K}$) can also occur if $\alpha^{3^{m+1}} - \alpha^3 \equiv 0$ (mod $3\mathcal{O}_{K}$) and also $\sum_{i=1}^{m}c^{3^{m+1-i}} \equiv 0$ (mod $3\mathcal{O}_{K}$); and thus also follows a contradiction. It then overall follows $f(x)= \varphi_{3,c}^{1+m}(x)-\varphi_{3,c}(x)$ has no roots in $ \mathcal{O}_{K}\slash 3\mathcal{O}_{K}$ for every $c\not \in 3\mathcal{O}_{K}$ and every fixed $m\in \mathbb{Z}_{\geq 1}$; and so we then conclude $N_{c}^{(1_{m})}(3) = 0$. This completes the whole proof, as required.
\end{proof} 
We now wish to generalize Theorem \ref{2.1} to any map $\varphi_{p,c}$ for any prime $p\geq 3$. More precisely, we prove that the number of distinct $1_{m}$-preperiodic integral points of any polynomial map $\varphi_{p,c}$ modulo $p\mathcal{O}_{K}$ is $p$ or zero:

\begin{thm} \label{2.2}
Let $K\slash \mathbb{Q}$ be any number field of degree $n\geq 1$ with the ring of integers $\mathcal{O}_{K}$, and in which any fixed prime $p\geq 3$ is inert. Let $\varphi_{p, c}$ be defined by $\varphi_{p, c}(z) = z^p + c$ for all $c, z\in\mathcal{O}_{K}$, and $N_{c}^{(1_{m})}(p)$ be defined as in \textnormal{(\ref{N_{c}})}. Then $N_{c}^{(1_{m})}(p)=p$ for every coefficient $c\in p\mathcal{O}_{K}$; otherwise $N_{c}^{(1_{m})}(p) = 0$ for every coefficient $c \not \in p\mathcal{O}_{K}$. 
\end{thm}
\begin{proof}
By applying a similar argument as in the Proof of Theorem \ref{2.1}, we then obtain the count as desired. That is, let $f(z)= \varphi_{p,c}^{1+m}(z)-\varphi_{p,c}(z) = \varphi_{p,c}^{1+m}(z) - z^p-c$, and so $f(z)= \varphi_{p,c}^{1+m}(z) - z^p - c$. Now applying the multinomial theorem repeatedly on the term $\varphi_{p,c}^{1+m}(z)$, it then follows $\varphi_{p,c}^{1+m}(z)$ is a monic polynomial in $z$ of degree $p^{m+1}$ with integral coefficients in multiples of $c$. Hence, we may then write $\varphi_{p,c}^{1+m}(z) = z^{p^{m+1}} + h(z)+c$, where $h(z)$ is a non-constant polynomial in $z$ of deg$(h)<p^{m+1}$ with integral coefficients in multiples of $c$. But then $f(z)= z^{p^{m+1}} + h(z) +c - z^p - c$, and so $f(z) = z^{p^{m+1}}-z^p + h(z)$. Now for every coefficient $c\in p\mathcal{O}_{K}$, reducing $f(z)$ modulo prime ideal $p\mathcal{O}_{K}$, it then follows $f(z)\equiv z^{p^{m+1}} - z^p$ (mod $p\mathcal{O}_{K}$), since also $h(z)\in c\mathcal{O}_{K}[z]$ and so $h(z)\equiv 0$ (mod $p\mathcal{O}_{K}$); and thus now $f(z)$ modulo $p\mathcal{O}_{K}$ is a polynomial defined over a finite field $\mathcal{O}_{K}\slash p\mathcal{O}_{K}$ of order $p^{[K:\mathbb{Q}]} = p^n$. So now, since every subfield of a finite field $\mathcal{O}_{K}\slash p\mathcal{O}_{K}$ is of order $p^r$ for some positive integer $r\mid n$, we then obtain the inclusion $\mathbb{F}_{p}\hookrightarrow \mathcal{O}_{K}\slash p\mathcal{O}_{K}$ of fields, where $\mathbb{F}_{p}$ is a field of order $p$; and moreover (as a well-known fact) $z^p = z$ for every $z\in \mathbb{F}_{p}\subset \mathcal{O}_{K}\slash p\mathcal{O}_{K}$. But then observe $z^{p^{m+1}}= (z^p)^{p^{m}} = (z^p)^{p^{m-1}} = z^{p^{m-1}}$ for every element $z\in \mathbb{F}_{p}$. Now since $m\geq 1$ and so $m-1\geq 0$, then if $m-1 = 0$ and so $z^{p^{m-1}} = z$, it then follows $z^{p^{m+1}} = z$  for every $z\in \mathbb{F}_{p}$; and so $f(z)\equiv 0$ (mod $p\mathcal{O}_{K}$) for every point $z\in \mathbb{F}_{p}\subset \mathcal{O}_{K}\slash p\mathcal{O}_{K}$. Otherwise, if $m-1 > 0$, then since $m$ is a fixed integer, we may then continue performing the above procedure of decreasing the exponent $m-1$ of $z^{p^{m-1}} = z^{p^{m+1}}$ for every $z\in \mathbb{F}_{p}$, until $m-1$ is equal to zero; and from which we then again obtain $f(z)\equiv 0$ (mod $p\mathcal{O}_{K}$) for every point $z\in \mathbb{F}_{p}\subset \mathcal{O}_{K}\slash p\mathcal{O}_{K}$. But now we then conclude $N_{c}^{(1_{m})}(p) = p$.

Finally, we now show $N_{c}^{(1_{m})}(p) = 0$ for every coefficient $c\not \equiv 0$ (mod $p\mathcal{O}_{K}$) and for every fixed integer $m\in \mathbb{Z}_{\geq 1}$. For the sake of a contradiction, let's suppose $f(z)\equiv 0$ (mod $p\mathcal{O}_{K}$) for some point $z\in \mathcal{O}_{K}\slash p\mathcal{O}_{K}$ and for every coefficient $c\not \equiv 0$ (mod $p\mathcal{O}_{K}$) and for every fixed $m\in \mathbb{Z}_{\geq 1}$. 
Now recall from earlier that the polynomial $f(z)= z^{p^{m+1}} -z^p + h(z)$ where $h(z)\in c\mathcal{O}_{K}[z]$, it then also follows by the above supposition that $z^{p^{m+1}} - z^p + h(z) \equiv 0$ (mod $p\mathcal{O}_{K}$) for some $z\in \mathcal{O}_{K}\slash p\mathcal{O}_{K}$ and for every $c\not \equiv 0$ (mod $p\mathcal{O}_{K}$) and every fixed $m\in \mathbb{Z}_{\geq 1}$. Now since we know from earlier $z^{p^{m+1}} = z^p$ for every $z\in \mathbb{F}_{p}\subset \mathcal{O}_{K}\slash p\mathcal{O}_{K}$ and every fixed $m\in \mathbb{Z}_{\geq 1}$, we may then rewrite $z^{p^{m+1}} - z^p + h(z) \equiv 0$ (mod $p\mathcal{O}_{K}$) for some $z\in \mathcal{O}_{K}\slash p\mathcal{O}_{K}$ and every $c\not \equiv 0$ (mod $p\mathcal{O}_{K}$) to obtain $h(z)\equiv 0$ (mod $p\mathcal{O}_{K}$) for some $z\in \mathbb{F}_{p}\subset \mathcal{O}_{K}\slash p\mathcal{O}_{K}$ and every $c\not \equiv 0$ (mod $p\mathcal{O}_{K}$) and every fixed $m\in \mathbb{Z}_{\geq 1}$. Moreover, looking at the multinomial expansion of $(\varphi_{p,c}^{m}(z))^p$, we then obtain $h(z)\equiv \sum_{i=1}^{m}c^{p^{m+1-i}}$ (mod $p\mathcal{O}_{K}$); and so $\sum_{i=1}^{m}c^{p^{m+1-i}} \equiv 0$ (mod $p\mathcal{O}_{K}$). But now we note that $\sum_{i=1}^{m}c^{p^{m+1-i}} \equiv 0$ (mod $p\mathcal{O}_{K}$) can also happen if $c \equiv 0$ (mod $p\mathcal{O}_{K}$); and so contradicting $c\not \equiv 0$ (mod $p\mathcal{O}_{K}$). Otherwise, suppose $f(\alpha)\equiv 0$ (mod $p\mathcal{O}_{K}$) and so $\alpha^{p^{m+1}} - \alpha^p + h(\alpha)\equiv 0$ (mod $p\mathcal{O}_{K}$) for some point $\alpha \in \mathcal{O}_{K}\slash p\mathcal{O}_{K} \setminus \mathbb{F}_{p}$ and for every $c\not \equiv 0$ (mod $p\mathcal{O}_{K}$) and every fixed $m\in \mathbb{Z}_{\geq 1}$. Now since (from a well-known fact) $z^{p^{m+1}}=z$ for every $z\in \mathbb{F}_{p^{m+1}}\subset \mathcal{O}_{K}\slash p\mathcal{O}_{K}$ and for every fixed $(m+1) \mid n$, then if a root $\alpha \in \mathbb{F}_{p^{m+1}}\subset \mathcal{O}_{K}\slash p\mathcal{O}_{K}\setminus \mathbb{F}_{p}$ and so $\alpha^{p^{m+1}}=\alpha$, we then obtain $\alpha - \alpha^p + h(\alpha)\equiv 0$ (mod $p\mathcal{O}_{K}$) for every $c\not \equiv 0$ (mod $p\mathcal{O}_{K}$) and every fixed $m\in \mathbb{Z}_{\geq 1}$; and moreover since $h(\alpha)\equiv \sum_{i=1}^{m}c^{p^{m+1-i}}$ (mod $p\mathcal{O}_{K}$), we then also have $\alpha- \alpha^p +\sum_{i=1}^{m}c^{p^{m+1-i}} \equiv 0$ (mod $p\mathcal{O}_{K}$). But now we note that the congruence $\alpha- \alpha^p +\sum_{i=1}^{m}c^{p^{m+1-i}}\equiv 0$ (mod $p\mathcal{O}_{K}$) can also happen if $\alpha - \alpha^p \equiv 0$ (mod $p\mathcal{O}_{K}$) and also $\sum_{i=1}^{m}c^{p^{m+1-i}} \equiv 0$ (mod $p\mathcal{O}_{K}$). Moreover, we also note $\alpha - \alpha^p \equiv 0$ (mod $p\mathcal{O}_{K}$) for every point $\alpha\equiv \pm1, 0$ (mod $p\mathcal{O}_{K}$), which also happened earlier when $c\equiv 0$ (mod $p\mathcal{O}_{K}$); and which then also contradicts the earlier condition $c\not \equiv 0$ (mod $p\mathcal{O}_{K}$). Otherwise, if a root $\alpha \not \in \mathbb{F}_{p^{m+1}}$ and since $h(\alpha)\equiv \sum_{i=1}^{m}c^{p^{m+1-i}}$ (mod $p\mathcal{O}_{K}$), then $\alpha^{p^{m+1}} - \alpha^p +\sum_{i=1}^{m}c^{p^{m+1-i}} \equiv 0$ (mod $p\mathcal{O}_{K}$). But now as before we note $\alpha^{p^{m+1}} - \alpha^p +\sum_{i=1}^{m}c^{p^{m+1-i}}  \equiv 0$ (mod $p\mathcal{O}_{K}$) can also occur if $\alpha^{p^{m+1}} - \alpha^p \equiv 0$ (mod $p\mathcal{O}_{K}$) and also $\sum_{i=1}^{m}c^{p^{m+1-i}} \equiv 0$ (mod $p\mathcal{O}_{K}$); and so also follows a contradiction. It then overall follows $f(x)= \varphi_{p,c}^{1+m}(x)-\varphi_{p,c}(x)$ has no roots in $ \mathcal{O}_{K}\slash p\mathcal{O}_{K}$ for every $c\not \in p\mathcal{O}_{K}$ and every fixed $m\in \mathbb{Z}_{\geq 1}$; and thus we then conclude $N_{c}^{(1_{m})}(p) = 0$. This completes the whole proof, as desired.
\end{proof}

Finally, we wish to generalize Theorem \ref{2.2} further to any $\varphi_{p^{\ell}, c}$ for any prime $p\geq 3$ and any $\ell \in \mathbb{Z}_{\geq 1}$. That is, we prove the number of distinct $1_{m}$-preperiodic integral points of any $\varphi_{p^{\ell}, c}$ modulo $p\mathcal{O}_{K}$ is $p$ or zero:

\begin{thm} \label{2.3}
Let $K\slash \mathbb{Q}$ be any number field of degree $ n \geq 1$ with ring $\mathcal{O}_{K}$, and in which any fixed prime $p\geq 3$ is inert. Let $\ell \geq 1$ be any fixed integer, and  $\varphi_{p^{\ell}, c}$ be defined by $\varphi_{p^{\ell}, c}(z) = z^{p^{\ell}} + c$ for all $c, z\in\mathcal{O}_{K}$. Let $N_{c}^{(1_{m})}(p)$ be defined as in \textnormal{(\ref{N_{c}})}. Then $N_{c}^{(1_{m})}(p) = p$  for every $c\in p\mathcal{O}_{K}$; otherwise $N_{c}^{(1_{m})}(p) = 0$ for every point $c \not \in p\mathcal{O}_{K}$. 
\end{thm}

\begin{proof}
By applying a similar argument as in the Proof of Theorem \ref{2.2}, we then obtain the count as desired. That is, let $f(z)= \varphi_{p^{\ell},c}^{1+m}(z)-\varphi_{p^{\ell},c}(z) = \varphi_{p^{\ell},c}^{1+m}(z) - z^{p^{\ell}}-c$, and so $f(z)= \varphi_{p^{\ell},c}^{1+m}(z) - z^{p^{\ell}} - c$. So now, as before applying the multinomial theorem repeatedly on $\varphi_{p^{\ell},c}^{1+m}(z)$, we then obtain that $\varphi_{p^{\ell},c}^{1+m}(z)$ is a monic polynomial in $z$ of degree $p^{(m+1)\ell}$ with integral coefficients in multiples of $c$. Hence, we may then write $\varphi_{p^{\ell},c}^{1+m}(z) = z^{p^{(m+1)\ell}} + h(z)+c$, where $h(z)$ is a non-constant polynomial in $z$ of deg$(h)<p^{(m+1)\ell}$ with integral coefficients in multiples of $c$. But then we note $f(z)= z^{p^{(m+1)\ell}} + h(z) +c - z^{p^{\ell}} - c$, and so $f(z) = z^{p^{(m+1)\ell}}-z^{p^{\ell}} + h(z)$. Now for every coefficient $c\in p\mathcal{O}_{K}$, reducing $f(z)$ modulo prime  $p\mathcal{O}_{K}$, it then follows $f(z)\equiv z^{p^{(m+1)\ell}} - z^{p^{\ell}}$ (mod $p\mathcal{O}_{K}$), since also $h(z)\in c\mathcal{O}_{K}[z]$ and thus $h(z)\equiv 0$ (mod $p\mathcal{O}_{K}$); and so now $f(z)$ modulo $p\mathcal{O}_{K}$ is a polynomial defined over a finite field $\mathcal{O}_{K}\slash p\mathcal{O}_{K}$. Now recall the inclusion  $\mathbb{F}_{p}\hookrightarrow \mathcal{O}_{K}\slash p\mathcal{O}_{K}$ of fields and also recall that $z^p = z$ for every element $z\in \mathbb{F}_{p}\subset \mathcal{O}_{K}\slash p\mathcal{O}_{K}$, it then follows $z^{p^{(m+1)\ell}}= (z^p)^{p^{(m+1)\ell-1}} = (z^p)^{p^{(m+1)\ell-2}} = z^{p^{(m+1)\ell-2}}$ for every element $z\in \mathbb{F}_{p}$; and also follows $z^{p^{\ell}}=(z^p)^{p^{\ell-1}} = z^{p^{\ell-1}}$ for every element $z\in \mathbb{F}_{p}$. So now, since $(m+1)\ell\geq 2$ and so $(m+1)\ell-2\geq 0$, then if $(m+1)\ell-2 = 0$ and so $z^{p^{(m+1)\ell-2}} = z$, we then obtain $z^{p^{(m+1)\ell}} = z$ for every $z\in \mathbb{F}_{p}$; and similarly since $\ell \geq 1$ and so $\ell - 1\geq 0$, then if $\ell - 1 = 0$ and so $z^{p^{\ell-1}} = z$, we the also obtain $z^{p^{\ell}} = z$ for every $z\in \mathbb{F}_{p}$; and so $f(z)\equiv 0$ (mod $p\mathcal{O}_{K}$) for every point $z\in \mathbb{F}_{p}\subset \mathcal{O}_{K}\slash p\mathcal{O}_{K}$. Otherwise, if $(m+1)\ell-2 > 0$ or $\ell-1 > 0$, then since $m$ and $\ell$ (and so $(m+1)\ell$) are fixed integers, we may then continue performing the above procedure of decreasing the exponent $(m+1)\ell-2$ of $z^{p^{(m+1)\ell-2}} = z^{p^{(m+1)\ell}}$ for every $z\in \mathbb{F}_{p}$ or $\ell-1$ of $z^{p^{\ell-1}} = z^{p^{\ell}}$ for every $z\in \mathbb{F}_{p}$, until $(m+1)\ell-2$ or $\ell-1$ is equal to zero; and from which we then again obtain $f(z)\equiv 0$ (mod $p\mathcal{O}_{K}$) for every point $z\in \mathbb{F}_{p}\subset \mathcal{O}_{K}\slash p\mathcal{O}_{K}$. But now as before we hereby conclude $N_{c}^{(1_{m})}(p) = p$.

Finally, we now show $N_{c}^{(1_{m})}(p) = 0$ for every coefficient $c\not \equiv 0$ (mod $p\mathcal{O}_{K}$) and every fixed integers $\ell, m\in \mathbb{Z}_{\geq 1}$. To do so, let's for the sake of a contradiction, suppose $f(z)\equiv 0$ (mod $p\mathcal{O}_{K}$) for some point $z\in \mathcal{O}_{K}\slash p\mathcal{O}_{K}$ and for every coefficient $c\not \equiv 0$ (mod $p\mathcal{O}_{K}$) and for every fixed $\ell, m\in \mathbb{Z}_{\geq 1}$. So now, recall from earlier that $f(z) = z^{p^{(m+1)\ell}}-z^{p^{\ell}} + h(z)$ where $h(z)\in c\mathcal{O}_{K}[z]$, it then also follows $z^{p^{(m+1)\ell}}-z^{p^{\ell}} + h(z) \equiv 0$ (mod $p\mathcal{O}_{K}$) for some $z\in \mathcal{O}_{K}\slash p\mathcal{O}_{K}$ and for every $c\not \equiv 0$ (mod $p\mathcal{O}_{K}$) and every fixed $\ell, m\in \mathbb{Z}_{\geq 1}$. Now since we know from earlier $z^{p^{(m+1)\ell}} = z^{p^{\ell}}$ for every $z\in \mathbb{F}_{p}$ and every fixed $\ell, m\in \mathbb{Z}_{\geq 1}$, we may then rewrite $z^{p^{(m+1)\ell}}-z^{p^{\ell}} + h(z) \equiv 0$ (mod $p\mathcal{O}_{K}$) for some $z\in \mathcal{O}_{K}\slash p\mathcal{O}_{K}$ and every $c\not \equiv 0$ (mod $p\mathcal{O}_{K}$) to obtain $h(z)\equiv 0$ (mod $p\mathcal{O}_{K}$) for some point $z\in \mathbb{F}_{p}\subset \mathcal{O}_{K}\slash p\mathcal{O}_{K}$ and every $c\not \equiv 0$ (mod $p\mathcal{O}_{K}$). More to this, looking at the multinomial expansion of $(\varphi_{p^{\ell},c}^{m}(z))^{p^{\ell}}$, we then obtain $h(z)\equiv \sum_{i=1}^{m}c^{p^{(m+1)\ell-i}}$ (mod $p\mathcal{O}_{K}$); and so $\sum_{i=1}^{m}c^{p^{(m+1)\ell-i}} \equiv 0$ (mod $p\mathcal{O}_{K}$). But now we note  $\sum_{i=1}^{m}c^{p^{(m+1)\ell-i}} \equiv 0$ (mod $p\mathcal{O}_{K}$) can also happen if $c \equiv 0$ (mod $p\mathcal{O}_{K}$); and so yielding a contradiction. Otherwise, suppose $f(\alpha)\equiv 0$ (mod $p\mathcal{O}_{K}$) and so $\alpha^{p^{(m+1)\ell}} - \alpha^{p^{\ell}} + h(\alpha)\equiv 0$ (mod $p\mathcal{O}_{K}$) for some $\alpha \in \mathcal{O}_{K}\slash p\mathcal{O}_{K} \setminus \mathbb{F}_{p}$ and for every $c\not \equiv 0$ (mod $p\mathcal{O}_{K}$) and every fixed $\ell, m\in \mathbb{Z}_{\geq 1}$. So now, since (from a well-known fact) $z^{p^{(m+1)\ell}}=z$ for every $z\in \mathbb{F}_{p^{(m+1)\ell}}\subset \mathcal{O}_{K}\slash p\mathcal{O}_{K}$ and every fixed $(m+1)\ell \mid n$, then if a root $\alpha \in \mathbb{F}_{p^{(m+1)\ell}}\subset \mathcal{O}_{K}\slash p\mathcal{O}_{K}\setminus \mathbb{F}_{p}$ and so $\alpha^{p^{(m+1)\ell}}=\alpha$, we then obtain $\alpha - \alpha^{p^{\ell}} + h(\alpha)\equiv 0$ (mod $p\mathcal{O}_{K}$); and since $h(\alpha)\equiv \sum_{i=1}^{m}c^{p^{(m+1)\ell-i}}$ (mod $p\mathcal{O}_{K}$), we then also have $\alpha- \alpha^{p^{\ell}} +\sum_{i=1}^{m}c^{p^{(m+1)\ell-i}} \equiv 0$ (mod $p\mathcal{O}_{K}$). But now we note $\alpha- \alpha^{p^{\ell}} +\sum_{i=1}^{m}c^{p^{(m+1)\ell-i}}\equiv 0$ (mod $p\mathcal{O}_{K}$) can also happen if $\alpha - \alpha^{p^{\ell}} \equiv 0$ (mod $p\mathcal{O}_{K}$) and also $\sum_{i=1}^{m}c^{p^{(m+1)\ell-i}} \equiv 0$ (mod $p\mathcal{O}_{K}$). Moreover, we also note $\alpha - \alpha^{p^{\ell}} \equiv 0$ (mod $p\mathcal{O}_{K}$) for every $\alpha\equiv \pm1, 0$ (mod $p\mathcal{O}_{K}$) and fixed $\ell \in \mathbb{Z}_{\geq 1}$, which also happened when $c\equiv 0$ (mod $p\mathcal{O}_{K}$); and which also contradicts the condition that $c\not \equiv 0$ (mod $p\mathcal{O}_{K}$). Otherwise, if a root $\alpha \not \in \mathbb{F}_{p^{(m+1)\ell}}$ and since $h(\alpha)\equiv \sum_{i=1}^{m}c^{p^{(m+1)\ell-i}}$ (mod $p\mathcal{O}_{K}$), then $\alpha^{p^{(m+1)\ell}} - \alpha^{p^{\ell}} +\sum_{i=1}^{m}c^{p^{(m+1)\ell-i}} \equiv 0$ (mod $p\mathcal{O}_{K}$). But then as before we note that the congruence $\alpha^{p^{(m+1)\ell}} - \alpha^{p^{\ell}} +\sum_{i=1}^{m}c^{p^{(m+1)\ell-i}}  \equiv 0$ (mod $p\mathcal{O}_{K}$) can also occur if $\alpha^{p^{(m+1)\ell}} - \alpha^{p^{\ell}} \equiv 0$ (mod $p\mathcal{O}_{K}$) and also $\sum_{i=1}^{m}c^{p^{(m+1)\ell-i}} \equiv 0$ (mod $p\mathcal{O}_{K}$); and thus also follows a contradiction. It then overall follows that $f(x)= \varphi_{p^{\ell},c}^{1+m}(x)-\varphi_{p^{\ell},c}(x)$ has no roots in $ \mathcal{O}_{K}\slash p\mathcal{O}_{K}$ for every coefficient $c\not \in p\mathcal{O}_{K}$ and every fixed integers $\ell, m\in \mathbb{Z}_{\geq 1}$; and therefore we then conclude $N_{c}^{(1_{m})}(p) = 0$. This then completes the whole proof, as desired. 
\end{proof}

Restricting on a subring $\mathbb{Z}\subset \mathcal{O}_{K}$ of ordinary integers, we obtain the following consequence of Theorem \ref{2.3} on the number of $1_{m}$-preperiodic integral points of any $\varphi_{p^{\ell},c}$ modulo $p$ for any prime $p\geq 3$ and any $\ell \in \mathbb{Z}_{\geq 1}$: 

\begin{cor} \label{cor2.4}
Let $p\geq 3$ be any fixed prime integer, and $\ell \geq 1$ be any fixed integer. Let $\varphi_{p^{\ell}, c}$ be a polynomial map defined by $\varphi_{p^{\ell}, c}(z) = z^{p^{\ell}} + c$ for all $c, z\in\mathbb{Z}$, and $N_{c}^{(1_{m})}(p)$ be the number defined as in \textnormal{(\ref{N_{c}})} with $\mathcal{O}_{K} / p\mathcal{O}_{K}$ replaced with $\mathbb{Z}\slash p\mathbb{Z}$. Then $N_{c}^{(1_{m})}(p) = p$  for every coefficient $c=pt$; otherwise $N_{c}^{(1_{m})}(p) = 0$ for every $c\neq pt$.
\end{cor}

\begin{proof}
By applying a similar argument as in the Proof of Theorem \ref{2.3}, we then obtain the count as desired. That is, let $f(z)= \varphi_{p^{\ell},c}^{1+m}(z)-\varphi_{p^{\ell},c}(z) = \varphi_{p^{\ell},c}^{1+m}(z) - z^{p^{\ell}}-c$, and so $f(z)= \varphi_{p^{\ell},c}^{1+m}(z) - z^{p^{\ell}} - c$. So now, applying the multinomial theorem repeatedly on $\varphi_{p^{\ell},c}^{1+m}(z)$, it then follows $\varphi_{p^{\ell},c}^{1+m}(z)$ is a monic polynomial in $z$ of degree $p^{(m+1)\ell}$ with integer coefficients in multiples of $c$. Thus, we may write $\varphi_{p^{\ell},c}^{1+m}(z) = z^{p^{(m+1)\ell}} + h(z)+c$, where $h(z)$ is a non-constant polynomial in $z$ of deg$(h)<p^{(m+1)\ell}$ with integer coefficients in multiples of $c$. But then $f(z)= z^{p^{(m+1)\ell}} + h(z) +c - z^{p^{\ell}} - c$ and so $f(z) = z^{p^{(m+1)\ell}}-z^{p^{\ell}} + h(z)$. Now for every coefficient $c=pt$, reducing $f(z)$ modulo $p$, it then follows $f(z)\equiv z^{p^{(m+1)\ell}} - z^{p^{\ell}}$ (mod $p$), since also $h(z)\in c\mathbb{Z}[z]$ and so $h(z)\equiv 0$ (mod $p$); and so now $f(z)$ modulo $p$ is a polynomial defined over a finite field $\mathbb{Z}\slash p\mathbb{Z}$ of order $p$. Now recall from Fermat's Little Theorem (FLT) that $z^p = z$ for every $z\in \mathbb{Z}\slash p\mathbb{Z}$, it then also follows $z^{p^{(m+1)\ell}}= (z^p)^{p^{(m+1)\ell-1}} = (z^p)^{p^{(m+1)\ell-2}} = z^{p^{(m+1)\ell-2}}$ for every $z\in \mathbb{Z}\slash p\mathbb{Z}$; and also follows $z^{p^{\ell}}=(z^p)^{p^{\ell-1}} = z^{p^{\ell-1}}$ for every $z\in \mathbb{Z}\slash p\mathbb{Z}$. Now since $(m+1)\ell\geq 2$ and so $(m+1)\ell-2\geq 0$, then if $(m+1)\ell-2 = 0$ and so $z^{p^{(m+1)\ell-2}} = z$, it then follows $z^{p^{(m+1)\ell}} = z$ for every $z\in \mathbb{Z}\slash p\mathbb{Z}$; and similarly since $\ell \geq 1$ and so $\ell - 1\geq 0$, then if $\ell - 1 = 0$ and so $z^{p^{\ell-1}} = z$, it then also follows $z^{p^{\ell}} = z$ for every $z\in \mathbb{Z}\slash p\mathbb{Z}$; and so $f(z)\equiv 0$ (mod $p$) for every $z\in \mathbb{Z}\slash p\mathbb{Z}$. Otherwise, if $(m+1)\ell-2 > 0$ or $\ell-1 > 0$, then since $m$ and $\ell$ (and so $(m+1)\ell$) are fixed integers, we may continue performing the above procedure of decreasing the exponent $(m+1)\ell-2$ of $z^{p^{(m+1)\ell-2}} = z^{p^{(m+1)\ell}}$ for every $z\in \mathbb{Z}\slash p\mathbb{Z}$ or $\ell-1$ of $z^{p^{\ell-1}} = z^{p^{\ell}}$ for every $z\in \mathbb{Z}\slash p\mathbb{Z}$, until $(m+1)\ell-2$ or $\ell-1$ is equal to zero; from which we then again obtain $f(z)\equiv 0$ (mod $p$) for every $z\in \mathbb{Z}\slash p\mathbb{Z}$. But now we hereby conclude $N_{c}^{(1_{m})}(p) = p$. 

Finally, we now show $N_{c}^{(1_{m})}(p) = 0$ for every coefficient $c\not \equiv 0$ (mod $p$) and every fixed integers $\ell, m\in \mathbb{Z}_{\geq 1}$. As before, let's for the sake of a contradiction, suppose $f(z)\equiv 0$ (mod $p$) for some $z\in \mathbb{Z}\slash p\mathbb{Z}$ and for every $c\not \equiv 0$ (mod $p$) and for every fixed $\ell, m\in \mathbb{Z}_{\geq 1}$. So now, recall from earlier that the polynomial $f(z) = z^{p^{(m+1)\ell}}-z^{p^{\ell}} + h(z)$ where $h(z)\in c\mathbb{Z}[z]$, it then also follows by the above supposition that $z^{p^{(m+1)\ell}} - z^{p^{\ell}} + h(z) \equiv 0$ (mod $p$) for some $z\in \mathbb{Z}\slash p\mathbb{Z}$ and for every $c\not \equiv 0$ (mod $p$) and for every fixed $\ell, m\in \mathbb{Z}_{\geq 1}$. So now, since we know from earlier that $z^{p^{(m+1)\ell}} = z^{p^{\ell}}$ for every $z\in \mathbb{Z}\slash p\mathbb{Z}$ and for every $\ell, m \in \mathbb{Z}_{\geq 1}$, we may then rewrite $z^{p^{(m+1)\ell}} - z^{p^{\ell}} + h(z) \equiv 0$ (mod $p$) for some $z\in \mathbb{Z}\slash p \mathbb{Z}$ and for every $c\not \equiv 0$ (mod $p$) and every fixed $\ell, m \in \mathbb{Z}_{\geq 1}$ to then obtain $h(z) \equiv 0$ (mod $p$) for some $z\in \mathbb{Z}\slash p \mathbb{Z}$ and for every $c\not \equiv 0$ (mod $p$) and every fixed $\ell, m \in \mathbb{Z}_{\geq 1}$. Moreover, looking at the multinomial expansion of $(\varphi_{p^{\ell},c}^{m}(z))^{p^{\ell}}$, we then obtain $h(z)\equiv \sum_{i=1}^{m}c^{p^{(m+1)\ell-i}}$ (mod $p$); and so $\sum_{i=1}^{m}c^{p^{(m+1)\ell-i}} \equiv 0$ (mod $p$). But now we note that the congruence $\sum_{i=1}^{m}c^{p^{(m+1)\ell-i}} \equiv 0$ (mod $p$) can also happen whenever $c \equiv 0$ (mod $p$); and  which then contradicts the earlier condition $c\not \equiv 0$ (mod $p$). This then means that $f(x)= \varphi_{p^{\ell},c}^{1+m}(x)-\varphi_{p^{\ell},c}(x)$ has no roots in $ \mathbb{Z}\slash p\mathbb{Z}$ for every coefficient $c$ indivisible by $p$ and for every fixed $\ell, m\in\mathbb{Z}_{\geq 1}$; and so we then conclude $N_{c}^{(1_{m})}(p) = 0$. This then completes the whole proof, as required.
\end{proof}

\begin{rem}\label{re2.3}
With now Theorem \ref{2.3}, we may then associate to each distinct $1_{m}$-preperiodic point of $\varphi_{p^{\ell},c}$ an $1_{m}$-preperiodic orbit. In doing so, we then obtain a dynamical translation of (a quantitative form of \say{all or nothing} principle in) Theorem \ref{2.3}, namely, that the number of distinct $1_{m}$-preperiodic integral orbits that any $\varphi_{p^{\ell},c}$ has when iterated on $\mathcal{O}_{K} / p\mathcal{O}_{K}$ is equal to $p$ or zero. As mentioned in Intro.\ref{sec1} that the count in Theorem \ref{2.3} on the number of distinct $1_{m}$-preperiodic integral points of any $\varphi_{p^{\ell},c}$ modulo $p\mathcal{O}_{K}$ may depend on $p$ (and so depend on deg$(\varphi_{p^{\ell},c})$ for every fixed $\ell \in \mathbb{Z}_{\geq 1}$) and however be independent of $n$ in one of the possibilities; or the count in Theorem \ref{2.3} may neither depend on $p$ nor on $n$ in the other possibility. Consequently, the expected total number (namely, $p+0 =p$ for every fixed eventual period $m\in \mathbb{Z}_{\geq 1}$) of distinct $1_{m}$-preperiodic integral points in the whole family of maps $\varphi_{p^{\ell},c}$ modulo $p\mathcal{O}_{K}$ may not only also depend on $p$ and however also be independent of $n$, but may also grow to infinity whenever degree $p^{\ell}\to \infty$. Worth also observing is that the $1_{m}$-preperiodic point-counting function $N_{c}^{(1_{m})}(p)\to \infty$ or $N_{c}^{(1_{m})}(p)\to 0$ whenever $p^{\ell}\to \infty$; a somewhat interesting phenomenon that not only coincide with a phenomenon in [\cite{BK22}, Remark 2.5] and also here in Remark \ref{re3.5}, but also somewhat suggesting here that \say{arithmetic is governing the expected dynamics}. Now recalling a well-known important philosophy in arithmetic geometry, namely, that \say{geometry governs arithmetic}, it might then also follow in some precise sense that \say{geometry would also govern the expected dynamics} in our setting; as loosely shown in Sect.\ref{sec4b} by observing that the expected total number of distinct $1_{m}$-preperiodic integral points of any $\varphi_{p^{\ell},c}$ modulo $p\mathcal{O}_{K}$ may not only be independent of $m$, but may also be a constant even as time $m\to \infty$. 
\end{rem}

As noted earlier in Intro.\ref{sec1} that as an immediate consequence of the expected total count obtained in Theorem \ref{2.3} and in [\cite{BK22}, Theorem 2.3] on any $\varphi_{p^{\ell},c}$ modulo $p\mathcal{O}_{K}$, we then also prove the following corollary on the expected total number of distinct strictly $1_{m}$-preperiodic integral points of any map $\varphi_{p^{\ell},c}$ modulo $p\mathcal{O}_{K}$:

\begin{cor}\label{co2.6}
Let $K\slash \mathbb{Q}$ be any number field of degree $ n \geq 1$ with the ring of integers $\mathcal{O}_{K}$, and in which any fixed prime $p\geq 3$ is inert. Let $m\in \mathbb{Z}_{\geq 1}$ be any fixed integer, and  $\varphi_{p^{\ell}, c}$ be defined by $\varphi_{p^{\ell}, c}(z)$ for all $c, z\in\mathcal{O}_{K}$. Then the expected total number of distinct strictly $1_{m}$-preperiodic points of $\varphi_{p^{\ell},c}$ modulo $p\mathcal{O}_{K}$ is equal to zero.
\end{cor}

\begin{proof}
Since for every fixed $m\in \mathbb{Z}_{\geq 1}$, we noted in Remark \ref{re2.3} that the expected total number of distinct $1_{m}$-preperiodic integral points of every polynomial map $\varphi_{p^{\ell},c}$ modulo $p\mathcal{O}_{K}$ is equal to $p$, and also since we noted in [\cite{BK22}, Remark 2.5] that the expected total number of distinct $m$-periodic integral points of every polynomial map $\varphi_{p^{\ell},c}$ modulo $p\mathcal{O}_{K}$ is equal to $p$, we then note that for every fixed $m\in \mathbb{Z}_{\geq 1}$, the expected total number of distinct strictly $1_{m}$-preperiodic integral points of every polynomial map $\varphi_{p^{\ell},c}$ modulo $p\mathcal{O}_{K}$ is equal to the expected total number of distinct $1_{m}$-preperiodic integral points of every polynomial map $\varphi_{p^{\ell},c}$ modulo $p\mathcal{O}_{K}$ minus the expected total number of distinct $m$-periodic integral points of every polynomial map $\varphi_{p^{\ell},c}$ modulo $p\mathcal{O}_{K}$, i.e., we have that the expected total number of distinct strictly $1_{m}$-preperiodic integral points of every polynomial map $\varphi_{p^{\ell},c}$ modulo $p\mathcal{O}_{K}$ is equal to $p-p=0$. This then completes the whole proof, as required.
\end{proof}

\begin{rem}
As in Rem.\ref{re2.3}, we then also note that a dynamical translation of Corollary \ref{co2.6} is the claim that the total number of distinct strictly $1_{m}$-preperiodic integral orbits of  every $\varphi_{p^{\ell},c}$ iterated on $\mathcal{O}_{K}\slash p\mathcal{O}_{K}$ is zero.
\end{rem}

\begin{rem}\label{re2.6}
Recall in [\cite{BK22}, Remark 2.6] we noted that for every fixed (period) $m\in \mathbb{Z}_{\geq 1}$, the $m$-periodic point-counting function $N_{c}^{(m)}(p) = p$ for every fixed inert $p\geq 3$ and every coefficient $c$ divisible by $p$ or $N_{c}^{(m)}(p) = 0$ for every fixed inert $p\geq 3$ and every coefficient $c$ indivisible by $p$; and from which we then obtained an arithmetic-geometric insight on every $m$-periodic integral orbit of any $\varphi_{p^{\ell},c}$ modulo $p\mathcal{O}_{K}$, namely, that every $m$-periodic integral orbit of $\varphi_{p^{\ell},c}$ modulo $p\mathcal{O}_{K}$ is $1$-periodic integral orbit, and moreover, every $\varphi_{p^{\ell},c}$ modulo $p\mathcal{O}_{K}$ may have $p$ distinct $1$-periodic integral orbits or no such orbits. So now, we may also recall in Theorem \ref{2.3} that for every fixed (eventual period) $m\in \mathbb{Z}_{\geq 1}$, the $1_{m}$-preperiodic point-counting function $N_{c}^{(1_{m})}(p) = p$ for every fixed inert $p$ and every $c$ divisible by $p$ or $N_{c}^{(1_{m})}(p) = 0$ for every fixed inert $p$ and every $c$ indivisible by $p$. But now for every fixed $m\in \mathbb{Z}_{\geq 1}$, we then note $N_{c}^{(1_{m})}(p) = N_{c}^{(m)}(p) = p$ for every fixed inert $p$ and every $c$ divisible by $p$ or $N_{c}^{(1_{m})}(p) = N_{c}^{(m)}(p) = 0$ for every fixed inert $p$ and every $c$ indivisible by $p$. Moreover, it also follows from the argument of the first part of the Proof of Theorem \ref{2.3} and [\cite{BK22}, Proof of Theorem 2.3] that every $1_{m}$-preperiodic integral point of $\varphi_{p^{\ell},c}$ modulo $p\mathcal{O}_{K}$ is $m$-periodic integral point of $\varphi_{p^{\ell},c}$ modulo $p\mathcal{O}_{K}$ and hence $1$-periodic integral point of $\varphi_{p^{\ell},c}$ modulo $p\mathcal{O}_{K}$. Consequently, it then overall follows that every reduced map $\varphi_{p^{\ell},c}$ modulo $p\mathcal{O}_{K}$ may have $p$ distinct $1$-periodic integral orbits or no such orbits; a somewhat interesting precise arithmetic-geometric insight on all the $1_{m}$-preperiodic integral orbits of every map $\varphi_{p^{\ell},c}$ modulo $p\mathcal{O}_{K}$.
\end{rem}

\section{Number of $1_{m}$-PrePeriodic Integral Points of any Family of Polynomial Maps $\varphi_{(p-1)^{\ell},c}$}\label{sec3}

As in Sect.\ref{sec2}, we also wish to count the number of distinct $1_{m}$-preperiodic integral points of any $\varphi_{(p-1)^{\ell},c}$ modulo $p\mathcal{O}_{K}$, where $p\geq 5$ is any fixed prime, and $\ell, m \in \mathbb{Z}_{\geq 1}$ are any fixed integers. Again, let $p\geq 5$ be any prime, $\ell, m \in \mathbb{Z}_{\geq 1}$ be any fixed integers, $c\in \mathcal{O}_{K}$ be any point, and then define $1_{m}$-preperiodic point-counting function 
\begin{equation}\label{M_{c}}
M_{c}^{(1_{m})}(p) := \# \Biggl\{ z\in \mathcal{O}_{K}\slash p\mathcal{O}_{K} : \begin{aligned} \varphi_{(p-1)^{\ell},c}^{m}(z) -z \not \equiv 0 \ \text{(mod $p\mathcal{O}_{K}$)} \\ \ \varphi_{(p-1)^{\ell},c}^{1+m}(z) - \varphi_{(p-1)^{\ell},c}(z) \equiv 0 \ \text{(mod $p\mathcal{O}_{K}$)} \end{aligned} \Biggr\}.
\end{equation}\noindent Again, setting $\ell =1$ and so $\varphi_{(p-1)^{\ell}, c} = \varphi_{p-1,c}$, we first prove the following theorem and its generalization \ref{3.2}:

\begin{thm} \label{3.1}
Let $K\slash \mathbb{Q}$ be any number field of degree $n\geq 1$ with the ring of integers $\mathcal{O}_{K}$, and in which $5$ is inert. Let $\varphi_{4, c}$ be a quartic map defined by $\varphi_{4, c}(z) = z^4 + c$ for all $c, z\in\mathcal{O}_{K}$, and $M_{c}^{(1_{m})}(5)$ be as in \textnormal{(\ref{M_{c}})}. Then $M_{c}^{(1_{m})}(5) = 5$ for every $c\in 5\mathcal{O}_{K}$ or $M_{c}^{(1_{m})}(5) = 4$ for every $c\equiv \pm 1 \ (mod \ 5\mathcal{O}_{K})$ and fixed (even) $m$; otherwise $M_{c}^{(1_{m})}(5) = 0$ for every $c\equiv -1 \ (mod \ 5\mathcal{O}_{K})$ and fixed odd $m$ or every $c\not \equiv \pm 1, 0 \ (mod \ 5\mathcal{O}_{K})$ and fixed (even) $m$.
\end{thm}

\begin{proof}
Let $g(z)= \varphi_{4,c}^{1+m}(z) - \varphi_{4,c}(z) = \varphi_{4,c}^{1+m}(z) - z^4 - c$, and so $g(z)= \varphi_{4,c}^{1+m}(z) - z^4 - c$. Now observe that applying the  multinomial theorem repeatedly on  the term $\varphi_{4,c}^{1+m}(z)$, we then obtain that $\varphi_{4,c}^{1+m}(z)$ is a monic polynomial in $z$ of degree $4^{m+1}$ with integral coefficients in multiples of $c$. Thus, we may then write $\varphi_{4,c}^{1+m}(z) = z^{4^{m+1}} + h(z) +c$, where $h(z)$ is a non-constant polynomial in $z$ of deg$(h)<4^{m+1}$ with integral coefficients in multiples of $c$. But then $g(z)= z^{4^{m+1}} + h(z) +c - z^4 - c$, and so $g(z)= z^{4^{m+1}} -z^4 + h(z)$. Now for every coefficient $c\in 5\mathcal{O}_{K}$, reducing $g(z)$ modulo prime ideal $5\mathcal{O}_{K}$, it then follows $g(z)\equiv z^{4^{m+1}} - z^4$ (mod $5\mathcal{O}_{K}$), since also $h(z)\in c\mathcal{O}_{K}[z]$ and thus $h(z)\equiv 0$ (mod $5\mathcal{O}_{K}$); and so now $g(z)$ modulo $5\mathcal{O}_{K}$ is a polynomial defined over a finite field $\mathcal{O}_{K}\slash 5\mathcal{O}_{K}$ of order $5^{[K:\mathbb{Q}]}=5^n$. So now, recall the inclusion $\mathbb{F}_{5}\hookrightarrow \mathcal{O}_{K}\slash 5\mathcal{O}_{K}$ of fields and also recall (as a fact) $z^4 = 1$ for every element $z\in \mathbb{F}_{5}^{\times} =\mathbb{F}_{5}\setminus \{0\}$, it then follows $z^{4^{m+1}}= (z^4)^{4^{m}} = 1$ for every element $z\in \mathbb{F}_{5}^{\times}$ and for every fixed $m\in \mathbb{Z}_{\geq 1}$; and so the reduced polynomial $g(z)\equiv 0$ (mod $5\mathcal{O}_{K}$) for every nonzero point $z\in \mathbb{F}_{5}\subset \mathcal{O}_{K}\slash 5\mathcal{O}_{K}$. Moreover, since $z$ is also a linear factor of $g(z)\equiv z(z^{4^{m+1}-1} - z^3)$ (mod $5\mathcal{O}_{K}$), it then follows $z\equiv 0$ (mod $5\mathcal{O}_{K}$) is also root of $g(z)$ modulo $5\mathcal{O}_{K}$. But now we conclude $M_{c}^{(1_{m})}(5) = 5$. 

To see $M_{c}^{(1_{m})}(5) = 4$ for every coefficient $c\equiv 1$ (mod $5\mathcal{O}_{K}$) and for every fixed integer $m\in \mathbb{Z}_{\geq 1}$, we first write $g(z)= \varphi_{4,c}^{1+m}(z) - z^4 - c = \varphi_{4,c}(\varphi_{4,c}^{m}(z)) - z^4-c$, and so $g(z) = (\varphi_{4,c}^{m}(z))^4 - z^4$. Now for every element $z\in \mathbb{F}_{5}^{\times}$ (and so $z^4 =1$) and for every $c\equiv 1$ (mod $5\mathcal{O}_{K}$), reducing $\varphi_{4,c}^{m}(z)$ modulo $5\mathcal{O}_{K}$ after each iteration, we then obtain $\varphi_{4,c}^{m}(z)\equiv 2$ (mod $5\mathcal{O}_{K}$) for every fixed $m\in \mathbb{Z}_{\geq 1}$; and so also obtain $(\varphi_{4,c}^{m}(z))^4\equiv 1$ (mod $5\mathcal{O}_{K}$). But then we note that $g(z)=(\varphi_{4,c}^{m}(z))^4 - z^4\equiv 0$ (mod $5\mathcal{O}_{K}$) for every nonzero point $z\in \mathbb{F}_{5}\subset \mathcal{O}_{K} / 5\mathcal{O}_{K}$; and so we then conclude $M_{c}^{(1_{m})}(5) = 4$. We now show $M_{c}^{(1_{m})}(5) = 4$ for every coefficient $c\equiv -1$ (mod $5\mathcal{O}_{K}$) and for every fixed even integer $m\in \mathbb{Z}_{\geq 2}$. As before, we note that since $c\equiv -1$ (mod $5\mathcal{O}_{K}$) and also since $z^4 =1$ for every $z\in \mathbb{F}_{5}^{\times}$, reducing $\varphi_{4,c}^m(z)$ modulo $5\mathcal{O}_{K}$ after each iteration, we then obtain $\varphi_{4,c}^{m}(z)\equiv -1$ (mod $5\mathcal{O}_{K}$) for every fixed even $m\in \mathbb{Z}_{\geq 2}$; and so also obtain $(\varphi_{4,c}^{m}(z))^4\equiv 1$ (mod $5\mathcal{O}_{K}$). But then we note $g(z)=(\varphi_{4,c}^{m}(z))^4 - z^4\equiv 0$ (mod $5\mathcal{O}_{K}$) for every nonzero point $z\in \mathbb{F}_{5}\subset \mathcal{O}_{K}\slash5\mathcal{O}_{K}$; and so as before we then again conclude $M_{c}^{(1_{m})}(5) = 4$.

Finally, we now show $M_{c}^{(1_{m})}(5) = 0$ for every coefficient 
$c\equiv -1$ (mod $5\mathcal{O}_{K}$) and every fixed odd integer $m\in \mathbb{Z}_{\geq 1}$ or for every coefficient $c \not \equiv \pm 1, 0$ (mod $5\mathcal{O}_{K}$) and every fixed (even) integer $m\in \mathbb{Z}_{\geq 1}$. To do so, let's for the sake of a contradiction, suppose $g(z)\equiv 0$ (mod $5\mathcal{O}_{K}$) for some nonzero point $z\in \mathcal{O}_{K}\slash5\mathcal{O}_{K}$ and for every coefficient $c \equiv -1$ (mod $5\mathcal{O}_{K}$) and every fixed odd integer $m\in \mathbb{Z}_{\geq 1}$. Now since $c\equiv -1$ (mod $5\mathcal{O}_{K}$) and $z^4=1$ for every $z\in \mathbb{F}_{5}^{\times}\subset \mathcal{O}_{K}\slash5\mathcal{O}_{K}$, reducing $\varphi_{4,c}^m(z)$ modulo $5\mathcal{O}_{K}$ after each iteration, we then obtain $\varphi_{4,c}^{m}(z)\equiv 0$ (mod $5\mathcal{O}_{K}$) for every fixed odd $m\in \mathbb{Z}_{\geq 1}$; and so $(\varphi_{4,c}^{m}(z))^4\equiv 0$ (mod $5\mathcal{O}_{K}$). But then $g(z) = (\varphi_{4,c}^{m}(z))^4 - z^4\equiv -1$ (mod $5\mathcal{O}_{K}$) for some point $z\in \mathbb{F}_{5}^{\times}\subset \mathcal{O}_{K}\slash5\mathcal{O}_{K}$ and for every $c \equiv -1$ (mod $5\mathcal{O}_{K}$) and every fixed odd $m\in \mathbb{Z}_{\geq 1}$; and from which we then obtain a contradiction that $0\equiv -1$ (mod $5\mathcal{O}_{K}$). Otherwise, suppose $g(z)\equiv 0$ (mod $5\mathcal{O}_{K}$) for some point $z\in \mathcal{O}_{K}\slash5\mathcal{O}_{K}$ and for every coefficient $c \not \equiv  \pm 1, 0$ (mod $5\mathcal{O}_{K}$) and every fixed (even) integer $m\in \mathbb{Z}_{\geq 1}$. Now recall from earlier that $g(z)=z^{4^{m+1}} -z^4 + h(z)$ where $h(z)\in c\mathcal{O}_{K}[z]$, it then also follows by the above supposition that $z^{4^{m+1}} -z^4 + h(z)\equiv 0$ (mod $5\mathcal{O}_{K}$) for some $z\in \mathcal{O}_{K}\slash5\mathcal{O}_{K}$ and for every $c \not \equiv  \pm 1, 0$ (mod $5\mathcal{O}_{K}$) and every fixed (even) $m\in \mathbb{Z}_{\geq 1}$. But now we note that the congruence $z^{4^{m+1}} -z^4 + h(z)\equiv 0$ (mod $5\mathcal{O}_{K}$) can also happen if $z^{4^{m+1}} -z^4\equiv 0$ (mod $5\mathcal{O}_{K}$) and also $h(z)\equiv 0$ (mod $5\mathcal{O}_{K}$). Moreover, we also note that $z^{4^{m+1}} -z^4\equiv 0$ (mod $5\mathcal{O}_{K}$) for every $z\equiv 0$ (mod $5\mathcal{O}_{K}$) and every fixed (even) $m\in \mathbb{Z}_{\geq 1}$, which also happened earlier when $c\equiv 0$ (mod $5\mathcal{O}_{K}$); and so also yielding a contradiction. It then overall follows that $g(x)=\varphi_{4,c}^{1+m}(x)-\varphi_{4,c}(x)$ has no roots in $\mathcal{O}_{K} / 5\mathcal{O}_{K}$ for every coefficient 
$c\equiv -1$ (mod $5\mathcal{O}_{K}$) and every fixed odd integer $m\in \mathbb{Z}_{\geq 1}$ or for every coefficient $c \not \equiv \pm 1, 0$ (mod $5\mathcal{O}_{K}$) and every fixed (even) integer $m\in \mathbb{Z}_{\geq 1}$; and so we then conclude $M_{c}^{(1_{m})}(5) = 0$. This then completes the whole proof, as needed.
\end{proof} 

We now wish to generalize Theorem \ref{3.1} to any $\varphi_{p-1, c}$ for any given prime $p\geq 5$. More precisely, we prove 
that the number of distinct $1_{m}$-preperiodic integral points of any $\varphi_{p-1, c}$ modulo $p\mathcal{O}_{K}$ is $p$ or $p-1$ or $0$:

\begin{thm} \label{3.2}
Let $K\slash \mathbb{Q}$ be any number field of degree $n\geq 1$ with ring $\mathcal{O}_{K}$, and in which any fixed prime $p\geq 5$ is inert. Let $\varphi_{p-1, c}$ be defined by $\varphi_{p-1, c}(z) = z^{p-1} + c$ for all $c, z\in\mathcal{O}_{K}$, and  $M_{c}^{(1_{m})}(p)$ be as in \textnormal{(\ref{M_{c}})}. Then $M_{c}^{(1_{m})}(p) = p$ for every $c\in p\mathcal{O}_{K}$ or $M_{c}^{(1_{m})}(p) = p-1$ for any $c\equiv \pm 1 \ (mod \ p\mathcal{O}_{K})$ and fixed (even) $m$; otherwise $M_{c}^{(1_{m})}(p) = 0$ for every $c\equiv -1 \ (mod \ p\mathcal{O}_{K})$ and fixed odd $m$ or every $c\not \equiv \pm 1, 0 \ (mod \ p\mathcal{O}_{K})$ and fixed (even) $m$.
\end{thm}
\begin{proof}
By applying a similar argument as in the Proof of Theorem \ref{3.1}, we then obtain the count as desired. That is, let $g(z)= \varphi_{p-1,c}^{1+m}(z) - \varphi_{p-1,c}(z) = \varphi_{p-1,c}^{1+m}(z) - z^{p-1} - c$, and so $g(z)= \varphi_{p-1,c}^{1+m}(z) - z^{p-1} - c$. Now applying the  multinomial theorem repeatedly on $\varphi_{p-1,c}^{1+m}(z)$, we then obtain that $\varphi_{p-1,c}^{1+m}(z)$ is a monic polynomial in $z$ of degree $(p-1)^{m+1}$ with integral coefficients in multiples of $c$. Hence, we may then write $\varphi_{p-1,c}^{1+m}(z) = z^{(p-1)^{m+1}} + h(z) +c$, where $h(z)$ is a non-constant polynomial in $z$ of deg$(h)<(p-1)^{m+1}$ with integral coefficients in multiples of $c$. But then $g(z)= z^{(p-1)^{m+1}} + h(z) +c - z^{p-1} - c$, and so $g(z)= z^{(p-1)^{m+1}} -z^{p-1} + h(z)$. Now for every coefficient $c\in p\mathcal{O}_{K}$, reducing $g(z)$ modulo prime ideal $p\mathcal{O}_{K}$, it then follows $g(z)\equiv z^{(p-1)^{m+1}} - z^{p-1}$ (mod $p\mathcal{O}_{K}$), since also $h(z)\in c\mathcal{O}_{K}[z]$ and so $h(z)\equiv 0$ (mod $p\mathcal{O}_{K}$); and so $g(z)$ modulo $p\mathcal{O}_{K}$ is now a polynomial defined over a finite field $\mathcal{O}_{K}\slash p\mathcal{O}_{K}$. So now, recall the inclusion $\mathbb{F}_{p}\hookrightarrow \mathcal{O}_{K}\slash p\mathcal{O}_{K}$ of fields and also recall (as a well-known fact) $z^{p-1} = 1$ for every element $z\in \mathbb{F}_{p}^{\times}=\mathbb{F}_{p}\setminus \{0\}$, it then also follows $z^{(p-1)^{m+1}}= (z^{p-1})^{(p-1)^{m}} = 1$ for every element $z\in \mathbb{F}_{p}^{\times}$ and for every fixed $m\in \mathbb{Z}_{\geq 1}$; and so the reduced polynomial $g(z)\equiv 0$ (mod $p\mathcal{O}_{K}$) for every nonzero point $z\in \mathbb{F}_{p}\subset \mathcal{O}_{K}\slash p\mathcal{O}_{K}$. Moreover, since $z$ is a linear factor of $g(z)\equiv z(z^{(p-1)^{m+1}-1} - z^{p-2})$ (mod $p\mathcal{O}_{K}$), it then follows $z\equiv 0$ (mod $p\mathcal{O}_{K}$) is also root of $g(z)$ modulo $p\mathcal{O}_{K}$. But now we conclude $M_{c}^{(1_{m})}(p) = p$. 

To see $M_{c}^{(1_{m})}(p) = p-1$ for every coefficient $c\equiv 1$ (mod $p\mathcal{O}_{K}$) and for every fixed integer $m\in \mathbb{Z}_{\geq 1}$, we first write $g(z)= \varphi_{p-1,c}^{1+m}(z) - z^{p-1} - c = \varphi_{p-1,c}(\varphi_{p-1,c}^{m}(z)) - z^{p-1}-c$, and so $g(z) = (\varphi_{p-1,c}^{m}(z))^{p-1} - z^{p-1}$. Now for every element $z\in \mathbb{F}_{p}^{\times}$ (and so $z^{p-1} =1$) and for every $c\equiv 1$ (mod $p\mathcal{O}_{K}$), reducing $\varphi_{p-1,c}^{m}(z)$ modulo $p\mathcal{O}_{K}$ after each iteration, we then obtain $\varphi_{p-1,c}^{m}(z)\equiv 2$ (mod $p\mathcal{O}_{K}$) for every fixed $m\in \mathbb{Z}_{\geq 1}$; and hence also obtain $(\varphi_{p-1,c}^{m}(z))^{p-1}\equiv 1$ (mod $p\mathcal{O}_{K}$). But then we note that $g(z)=(\varphi_{p-1,c}^{m}(z))^{p-1} - z^{p-1}\equiv 0$ (mod $p\mathcal{O}_{K}$) for every nonzero point $z\in \mathbb{F}_{p}\subset \mathcal{O}_{K} / p\mathcal{O}_{K}$; and so we then conclude $M_{c}^{(1_{m})}(p) = p-1$. We now show $M_{c}^{(1_{m})}(p) = p-1$ for every coefficient $c\equiv -1$ (mod $p\mathcal{O}_{K}$) and for every fixed even integer $m\in \mathbb{Z}_{\geq 2}$. As before, we note that since $c\equiv -1$ (mod $p\mathcal{O}_{K}$) and also since $z^{p-1} =1$ for every $z\in \mathbb{F}_{p}^{\times}$, reducing $\varphi_{p-1,c}^m(z)$ modulo $p\mathcal{O}_{K}$ after each iteration, we then obtain $\varphi_{p-1,c}^{m}(z)\equiv -1$ (mod $p\mathcal{O}_{K}$) for every fixed even $m\in \mathbb{Z}_{\geq 2}$; and thus also obtain $(\varphi_{p-1,c}^{m}(z))^{p-1}\equiv 1$ (mod $p\mathcal{O}_{K}$). But then we note that the reduced polynomial $g(z)=(\varphi_{p-1,c}^{m}(z))^{p-1} - z^{p-1}\equiv 0$ (mod $p\mathcal{O}_{K}$) for every nonzero point $z\in \mathbb{F}_{p}\subset \mathcal{O}_{K}\slash p\mathcal{O}_{K}$; and thus as before we then conclude $M_{c}^{(1_{m})}(p) = p-1$. 

Finally, we now show $M_{c}^{(1_{m})}(p) = 0$ for every coefficient 
$c\equiv -1$ (mod $p\mathcal{O}_{K}$) and every fixed odd integer $m\in \mathbb{Z}_{\geq 1}$ or for every coefficient $c \not \equiv \pm 1, 0$ (mod $p\mathcal{O}_{K}$) and every fixed (even) integer $m\in \mathbb{Z}_{\geq 1}$. To do so, let's for the sake of a contradiction, suppose $g(z)\equiv 0$ (mod $p\mathcal{O}_{K}$) for some nonzero point $z\in \mathcal{O}_{K}\slash p\mathcal{O}_{K}$ and for every coefficient $c \equiv -1$ (mod $p\mathcal{O}_{K}$) and every fixed odd integer $m\in \mathbb{Z}_{\geq 1}$. Now since $c\equiv -1$ (mod $p\mathcal{O}_{K}$) and $z^{p-1}=1$ for every $z\in \mathbb{F}_{p}^{\times}\subset \mathcal{O}_{K}\slash p\mathcal{O}_{K}$, reducing $\varphi_{p-1,c}^m(z)$ modulo $p\mathcal{O}_{K}$ after each iteration, we then obtain $\varphi_{p-1,c}^{m}(z)\equiv 0$ (mod $p\mathcal{O}_{K}$) for every fixed odd integer $m\in \mathbb{Z}_{\geq 1}$; and so $(\varphi_{p-1,c}^{m}(z))^{p-1}\equiv 0$ (mod $p\mathcal{O}_{K}$). But then $g(z) = (\varphi_{p-1,c}^{m}(z))^{p-1} - z^{p-1}\equiv -1$ (mod $p\mathcal{O}_{K}$) for some point $z\in \mathbb{F}_{p}^{\times}\subset \mathcal{O}_{K}\slash p\mathcal{O}_{K}$ and for every $c \equiv -1$ (mod $p\mathcal{O}_{K}$) and every fixed odd $m\in \mathbb{Z}_{\geq 1}$; from which we then obtain a contradiction that $0\equiv -1$ (mod $p\mathcal{O}_{K}$). Otherwise, suppose $g(z)\equiv 0$ (mod $p\mathcal{O}_{K}$) for some point $z\in \mathcal{O}_{K}\slash p\mathcal{O}_{K}$ and for every coefficient $c \not \equiv  \pm 1, 0$ (mod $p\mathcal{O}_{K}$) and every fixed (even) integer $m\in \mathbb{Z}_{\geq 1}$. So now, recall from earlier that the polynomial $g(z)=z^{(p-1)^{m+1}} -z^{p-1} + h(z)$ where $h(z)\in c\mathcal{O}_{K}[z]$, it then also follows by the above supposition that $z^{(p-1)^{m+1}} -z^{p-1} + h(z)\equiv 0$ (mod $p\mathcal{O}_{K}$) for some $z\in \mathcal{O}_{K}\slash p\mathcal{O}_{K}$ and for every $c \not \equiv \pm 1, 0$ (mod $p\mathcal{O}_{K}$) and every fixed (even) $m\in \mathbb{Z}_{\geq 1}$. But then we note that the congruence $z^{(p-1)^{m+1}} -z^{p-1} + h(z)\equiv 0$ (mod $p\mathcal{O}_{K}$) can also happen whenever $z^{(p-1)^{m+1}} -z^{p-1}\equiv 0$ (mod $p\mathcal{O}_{K}$) and also $h(z)\equiv 0$ (mod $p\mathcal{O}_{K}$). Moreover, we also note that $z^{(p-1)^{m+1}} -z^{p-1}\equiv 0$ (mod $p\mathcal{O}_{K}$) for every point $z\equiv 0$ (mod $p\mathcal{O}_{K}$) and every fixed (even) $m\in \mathbb{Z}_{\geq 1}$, which also happened when $c\equiv 0$ (mod $p\mathcal{O}_{K}$); and from which we then also obtain a contradiction. It then overall follows that $g(x)=\varphi_{p-1,c}^{1+m}(x)-\varphi_{p-1,c}(x)$ has no roots in $\mathcal{O}_{K} / p\mathcal{O}_{K}$ for every coefficient 
$c\equiv -1$ (mod $p\mathcal{O}_{K}$) and every fixed odd integer $m\in \mathbb{Z}_{\geq 1}$ or for every coefficient $c \not \equiv \pm 1, 0$ (mod $p\mathcal{O}_{K}$) and every fixed (even) integer $m\in \mathbb{Z}_{\geq 1}$; and so we then conclude $M_{c}^{(1_{m})}(p) = 0$. This then completes the whole proof, as needed.
\end{proof}

Finally, we generalize Theorem \ref{3.2} further to any $\varphi_{(p-1)^{\ell}, c}$ for any prime $p\geq 5$ and any $\ell\in \mathbb{Z}_{\geq 1}$. That is, we prove the number of distinct $1_{m}$-preperiodic points of any $\varphi_{(p-1)^{\ell}, c}$ modulo $p\mathcal{O}_{K}$ is also $p$ or $p-1$ or zero:

\begin{thm} \label{3.3}
Let $K\slash \mathbb{Q}$ be any number field of degree $n\geq 1$ with ring $\mathcal{O}_{K}$, and in which any fixed prime $p\geq 5$ is inert. Let $\varphi_{(p-1)^{\ell}, c}(z) = z^{(p-1)^{\ell}} + c$ for all $c, z\in\mathcal{O}_{K}$ and $\ell\in \mathbb{Z}_{\geq 1}$. Let $M_{c}^{(1_{m})}(p)$ be as in \textnormal{(\ref{M_{c}})}. Then $M_{c}^{(1_{m})}(p) = p$ for every $c\in p\mathcal{O}_{K}$ or $M_{c}^{(1_{m})}(p) = p-1$ for any $c\equiv \pm 1 \ (mod \ p\mathcal{O}_{K})$ and fixed (even) $m$; otherwise $M_{c}^{(1_{m})}(p) = 0$ for every $c\equiv -1 \ (mod \ p\mathcal{O}_{K})$ and fixed odd $m$ or every $c\not \equiv \pm 1, 0 \ (mod \ p\mathcal{O}_{K})$ and fixed (even) $m$.
\end{thm}
\begin{proof}
By again applying a similar argument as in the Proof of Theorem \ref{3.2}, we then obtain the count as desired. That is, let $g(z)= \varphi_{(p-1)^{\ell},c}^{1+m}(z) - \varphi_{(p-1)^{\ell},c}(z) = \varphi_{(p-1)^{\ell},c}^{1+m}(z) - z^{(p-1)^{\ell}} - c$, and so $g(z)= \varphi_{(p-1)^{\ell},c}^{1+m}(z) - z^{(p-1)^{\ell}} - c$. So now, applying the  multinomial theorem repeatedly on the term $\varphi_{(p-1)^{\ell},c}^{1+m}(z)$, we then obtain that $\varphi_{(p-1)^{\ell},c}^{1+m}(z)$ is a monic polynomial in $z$ of degree $(p-1)^{(m+1)\ell}$ with integral coefficients in multiples of $c$. Hence, we may then write $\varphi_{(p-1)^{\ell},c}^{1+m}(z) = z^{(p-1)^{(m+1)\ell}} + h(z) +c$, where $h(z)$ is a non-constant polynomial in $z$ of deg$(h)<(p-1)^{(m+1)\ell}$ with integral coefficients in multiples of $c$. But now we write $g(z)= z^{(p-1)^{(m+1)\ell}} + h(z) +c - z^{(p-1)^{\ell}} - c$, and so $g(z)= z^{(p-1)^{(m+1)\ell}} -z^{(p-1)^{\ell}} + h(z)$. Now for every coefficient $c \in p\mathcal{O}_{K}$, reducing $g(z)$ modulo prime ideal $p\mathcal{O}_{K}$, it then follows $g(z)\equiv z^{(p-1)^{(m+1)\ell}} - z^{(p-1)^{\ell}}$ (mod $p\mathcal{O}_{K}$), since also $h(z)\in c\mathcal{O}_{K}[z]$ and thus $h(z)\equiv 0$ (mod $p\mathcal{O}_{K}$); and so now the reduced polynomial $g(z)$ modulo $p\mathcal{O}_{K}$ is now a polynomial defined over a finite field $\mathcal{O}_{K}\slash p\mathcal{O}_{K}$. Now recall the inclusion $\mathbb{F}_{p}\hookrightarrow\mathcal{O}_{K}\slash p\mathcal{O}_{K}$ of fields and $z^{p-1} = 1$ for every element $z\in \mathbb{F}_{p}^{\times}$, it then follows $z^{(p-1)^{(m+1)\ell}} = (z^{p-1})^{(p-1)^{(m+1)\ell-1}}=1$ for every element $z\in \mathbb{F}_{p}^{\times}$; and also follows $z^{(p-1)^{\ell}}=(z^{p-1})^{(p-1)\ell-1}=1$ for every element $z\in \mathbb{F}_{p}^{\times}$; and so the reduced polynomial $g(z)\equiv 0$ (mod $p\mathcal{O}_{K}$) for every nonzero point $z\in \mathbb{F}_{p}\subset\mathcal{O}_{K}\slash p\mathcal{O}_{K}$. Moreover, since $z$ is also a linear factor of $g(z)\equiv z(z^{(p-1)^{(m+1)\ell}-1} - z^{(p-1)^{\ell}-1})$ (mod $p\mathcal{O}_{K}$), it then follows $z\equiv 0$ (mod $p\mathcal{O}_{K}$) is also a root of $g(z)$ modulo $p\mathcal{O}_{K}$. But now we then conclude $M_{c}^{(1_{m})}(p) = p$. 

To see $M_{c}^{(1_{m})}(p) = p-1$ for every coefficient $c\equiv 1$ (mod $p\mathcal{O}_{K}$) and for every fixed integers $\ell, m\in \mathbb{Z}_{\geq 1}$, we first note that $g(z)= \varphi_{(p-1)^{\ell},c}^{1+m}(z) - z^{(p-1)^{\ell}} - c = \varphi_{(p-1)^{\ell},c}(\varphi_{(p-1)^{\ell},c}^{m}(z)) - z^{(p-1)^{\ell}}-c$, and so $g(z) = (\varphi_{(p-1)^{\ell},c}^{m}(z))^{(p-1)^{\ell}} - z^{(p-1)^{\ell}}$. Now for every point $z\in \mathbb{F}_{p}^{\times}$ (and so $z^{(p-1)^{\ell}} =1$ for every $\ell \in \mathbb{Z}_{\geq 1}$) and for every coefficient $c\equiv 1$ (mod $p\mathcal{O}_{K}$), reducing $\varphi_{(p-1)^{\ell},c}^{m}(z)$ modulo $p\mathcal{O}_{K}$ after each iteration, we then obtain $\varphi_{(p-1)^{\ell},c}^{m}(z)\equiv 2$ (mod $p\mathcal{O}_{K}$) for every fixed $\ell, m\in \mathbb{Z}_{\geq 1}$; and thus also obtain $(\varphi_{(p-1)^{\ell},c}^{m}(z))^{(p-1)^{\ell}}\equiv 1$ (mod $p\mathcal{O}_{K}$). But then we note $g(z)=(\varphi_{(p-1)^{\ell},c}^{m}(z))^{(p-1)^{\ell}} - z^{(p-1)^{\ell}}\equiv 0$ (mod $p\mathcal{O}_{K}$) for every nonzero point $z\in \mathbb{F}_{p}\subset \mathcal{O}_{K} / p\mathcal{O}_{K}$; and so we then conclude $M_{c}^{(1_{m})}(p) = p-1$. We now show $M_{c}^{(1_{m})}(p) = p-1$ for every coefficient $c\equiv -1$ (mod $p\mathcal{O}_{K}$) and for every fixed $\ell \in \mathbb{Z}_{\geq 1}$ and every fixed even integer $m\in \mathbb{Z}_{\geq 2}$. As before, since $c\equiv -1$ (mod $p\mathcal{O}_{K}$) and also since $z^{(p-1)^{\ell}} =1$ for every $z\in \mathbb{F}_{p}^{\times}$, reducing $\varphi_{(p-1)^{\ell},c}^m(z)$ modulo $p\mathcal{O}_{K}$ after each iteration, we then obtain $\varphi_{(p-1)^{\ell},c}^{m}(z)\equiv -1$ (mod $p\mathcal{O}_{K}$) for every fixed $\ell \in \mathbb{Z}_{ \geq 1}$
and fixed even $m\in \mathbb{Z}_{\geq 2}$; and hence also obtain $(\varphi_{(p-1)^{\ell},c}^{m}(z))^{(p-1)^{\ell}}\equiv 1$ (mod $p\mathcal{O}_{K}$). But then we note that $g(z)=(\varphi_{(p-1)^{\ell},c}^{m}(z))^{(p-1)^{\ell}} - z^{(p-1)^{\ell}}\equiv 0$ (mod $p\mathcal{O}_{K}$) for every nonzero point $z\in \mathbb{F}_{p}\subset \mathcal{O}_{K}\slash p\mathcal{O}_{K}$; and thus as before we again conclude $M_{c}^{(1_{m})}(p) = p-1$.

Finally, we now show $M_{c}^{(1_{m})}(p) = 0$ for every coefficient 
$c\equiv -1$ (mod $p\mathcal{O}_{K}$) and every fixed $\ell \in \mathbb{Z}_{\geq 1}$ and odd integer $m\in \mathbb{Z}_{\geq 1}$ or for every coefficient $c \not \equiv \pm 1, 0$ (mod $p\mathcal{O}_{K}$) and every fixed (even) integer $\ell, m\in \mathbb{Z}_{\geq 1}$. To do so, let's for the sake of a contradiction, suppose $g(z)\equiv 0$ (mod $p\mathcal{O}_{K}$) for some nonzero point $z\in \mathcal{O}_{K}\slash p\mathcal{O}_{K}$ and for every $c\equiv -1$ (mod $p\mathcal{O}_{K}$) and every fixed $\ell \in \mathbb{Z}_{\geq 1}$ and odd $m\in \mathbb{Z}_{\geq 1}$. Now since $c\equiv -1$ (mod $p\mathcal{O}_{K}$) and $z^{(p-1)^{\ell}}=1$ for every $z\in \mathbb{F}_{p}^{\times}\subset \mathcal{O}_{K}\slash p\mathcal{O}_{K}$, reducing $\varphi_{(p-1)^{\ell},c}^m(z)$ modulo $p\mathcal{O}_{K}$ after each iteration, we then obtain $\varphi_{(p-1)^{\ell},c}^{m}(z)\equiv 0$ (mod $p\mathcal{O}_{K}$) for every fixed $\ell \in \mathbb{Z}_{\geq 1}$ and odd $m\in \mathbb{Z}_{\geq 1}$; and so $(\varphi_{(p-1)^{\ell},c}^{m}(z))^{(p-1)^{\ell}}\equiv 0$ (mod $p\mathcal{O}_{K}$). But then $g(z) = (\varphi_{(p-1)^{\ell},c}^{m}(z))^{(p-1)^{\ell}} - z^{(p-1)^{\ell}}\equiv -1$ (mod $p\mathcal{O}_{K}$) for some $z\in \mathbb{F}_{p}^{\times}\subset \mathcal{O}_{K}\slash p\mathcal{O}_{K}$ and for every $c \equiv -1$ (mod $p\mathcal{O}_{K}$) and every fixed $\ell \in \mathbb{Z}_{\geq 1}$ and odd $m\in \mathbb{Z}_{\geq 1}$; from which we then obtain a contradiction that $0\equiv -1$ (mod $p\mathcal{O}_{K}$). Otherwise, suppose $g(z)\equiv 0$ (mod $p\mathcal{O}_{K}$) for some $z\in \mathcal{O}_{K}\slash p\mathcal{O}_{K}$ and for every coefficient $c \not \equiv \pm 1, 0$ (mod $p\mathcal{O}_{K}$) and for every fixed (even) integer $\ell, m\in \mathbb{Z}_{\geq 1}$. So now, recall from earlier that $g(z)=z^{(p-1)^{(m+1)\ell}} -z^{(p-1)^{\ell}} + h(z)$ where $h(z)\in c\mathcal{O}_{K}[z]$, it then also follows by the above supposition that $z^{(p-1)^{(m+1)\ell}} -z^{(p-1)^{\ell}} + h(z)\equiv 0$ (mod $p\mathcal{O}_{K}$) for some point $z\in \mathcal{O}_{K}\slash p\mathcal{O}_{K}$ and for every $c \not \equiv \pm 1, 0$ (mod $p\mathcal{O}_{K}$) and for every fixed (even) $\ell, m\in \mathbb{Z}_{\geq 1}$. But then we note that the congruence $z^{(p-1)^{(m+1)\ell}} -z^{(p-1)^{\ell}} + h(z)\equiv 0$ (mod $p\mathcal{O}_{K}$) can also happen if $z^{(p-1)^{(m+1)\ell}} -z^{(p-1)^{\ell}}\equiv 0$ (mod $p\mathcal{O}_{K}$) and also $h(z)\equiv 0$ (mod $p\mathcal{O}_{K}$). Moreover, we also note that $z^{(p-1)^{(m+1)\ell}} -z^{(p-1)^{\ell}}\equiv 0$ (mod $p\mathcal{O}_{K}$) for every point $z\equiv 0$ (mod $p\mathcal{O}_{K}$) and for every fixed (even) $\ell, m\in \mathbb{Z}_{\geq 1}$, which also happened earlier when $c\equiv 0$ (mod $p\mathcal{O}_{K}$); from which we then also obtain a contradiction. This then overall means that $g(x)=\varphi_{(p-1)^{\ell},c}^{1+m}(x)-\varphi_{(p-1)^{\ell},c}(x)$ has no roots in $\mathcal{O}_{K} / p\mathcal{O}_{K}$ for every $c\equiv -1$ (mod $p\mathcal{O}_{K}$) and every fixed odd $\ell, m\in \mathbb{Z}_{\geq 1}$ or for every $c \not \equiv \pm 1, 0$ (mod $p\mathcal{O}_{K}$) and every fixed (even) $\ell, m\in \mathbb{Z}_{\geq 1}$; and so we then conclude $M_{c}^{(1_{m})}(p) = 0$. This then completes the whole proof, as needed.
\end{proof}

Restricting on a subring $\mathbb{Z}\subset \mathcal{O}_{K}$ of integers, we obtain the following consequence of Theorem \ref{3.3} on the number of distinct $1_{m}$-preperiodic integral points of any $\varphi_{(p-1)^{\ell},c}$ (mod $p$) for any prime $p\geq 5$ and any $\ell \in \mathbb{Z}_{\geq 1}$:

\begin{cor} \label{cor3.4}
Let $p\geq 5$ be any fixed prime, and $\ell \geq 1$ be any fixed integer. Let $\varphi_{(p-1)^{\ell}, c}$ be a map defined by $\varphi_{(p-1)^{\ell}, c}(z) = z^{(p-1)^{\ell}} + c$ for all $c, z\in\mathbb{Z}$, and $M_{c}^{(1_{m})}(p)$ be defined as in \textnormal{(\ref{M_{c}})} with $\mathcal{O}_{K} / p\mathcal{O}_{K}$ replaced with $\mathbb{Z}\slash p\mathbb{Z}$. Then $M_{c}^{(1_{m})}(p) = p$ for every $c=pt$ or $M_{c}^{(1_{m})}(p) = p-1$ for every $c\equiv \pm 1 \ (mod \ p)$ and fixed (even) $m$; otherwise $M_{c}^{(1_{m})}(p) = 0$ for any $c\equiv -1 \ (mod \ p)$ and fixed odd $m$ or any $c\not \equiv \pm 1, 0 \ (mod \ p)$ and fixed (even) $m$.
\end{cor}

\begin{proof}
By applying a similar argument as in the Proof of Theorem \ref{3.3}, we then obtain the count as desired. That is, as before let $g(z)= \varphi_{(p-1)^{\ell},c}^{1+m}(z) - \varphi_{(p-1)^{\ell},c}(z) = \varphi_{(p-1)^{\ell},c}^{1+m}(z) - z^{(p-1)^{\ell}} - c$, and so $g(z)= \varphi_{(p-1)^{\ell},c}^{1+m}(z) - z^{(p-1)^{\ell}} - c$. Now applying the multinomial theorem on $\varphi_{(p-1)^{\ell},c}^{1+m}(z)$, it then follows $\varphi_{(p-1)^{\ell},c}^{1+m}(z)$ is a monic polynomial in $z$ of degree $(p-1)^{(m+1)\ell}$ with integer coefficients in multiples of $c$. Thus, we may then write $\varphi_{(p-1)^{\ell},c}^{1+m}(z) = z^{(p-1)^{(m+1)\ell}} + h(z) +c$, where $h(z)$ is a non-constant polynomial in $z$ of deg$(h)<(p-1)^{(m+1)\ell}$ with integer coefficients in multiples of $c$. But then we write $g(z)= z^{(p-1)^{(m+1)\ell}} + h(z) +c - z^{(p-1)^{\ell}} - c$, and so $g(z)= z^{(p-1)^{(m+1)\ell}} -z^{(p-1)^{\ell}} + h(z)$. Now for every coefficient $c =pt$, reducing $g(z)$ modulo $p$, we then obtain $g(z)\equiv z^{(p-1)^{(m+1)\ell}} - z^{(p-1)^{\ell}}$ (mod $p$), since also $h(z)\in c\mathbb{Z}[z]$ and so $h(z)\equiv 0$ (mod $p$); and thus now $g(z)$ modulo $p$ is now a polynomial defined over a field $\mathbb{Z}\slash p\mathbb{Z}$ of order $p$. Now recall from Fermat's Little Theorem (FLT) that $z^{p-1} = 1$ for every $z\in (\mathbb{Z}\slash p\mathbb{Z})^{\times} = \mathbb{Z}\slash p\mathbb{Z}\setminus \{0\}$, it then also follows $z^{(p-1)^{(m+1)\ell}} = 1$ (mod $p$) for every $z\in (\mathbb{Z}\slash p\mathbb{Z})^{\times}$; and also follows $z^{(p-1)^{\ell}} = 1$ (mod $p$) for every $z\in (\mathbb{Z}\slash p\mathbb{Z})^{\times}$; and so $g(z)\equiv 0$ (mod $p$) for every point $z\in (\mathbb{Z}\slash p\mathbb{Z})^{\times}$. Moreover, since $z$ is also a linear factor of $g(z)\equiv z(z^{(p-1)^{(m+1)\ell}-1} - z^{(p-1)^{\ell}-1})$ (mod $p$), it then follows $z\equiv 0$ (mod $p$) is also a root of $g(z)$ modulo $p$. But now as before we then conclude $M_{c}^{(1_{m})}(p) = p$.

To see $M_{c}^{(1_{m})}(p) = p-1$ for every coefficient $c\equiv 1$ (mod $p$) and for every fixed integers $\ell, m\in \mathbb{Z}_{\geq 1}$, we first note that $g(z)= \varphi_{(p-1)^{\ell},c}^{1+m}(z) - z^{(p-1)^{\ell}} - c = \varphi_{(p-1)^{\ell},c}(\varphi_{(p-1)^{\ell},c}^{m}(z)) - z^{(p-1)^{\ell}}-c$, and so $g(z) = (\varphi_{(p-1)^{\ell},c}^{m}(z))^{(p-1)^{\ell}} - z^{(p-1)^{\ell}}$. Now for every $z\in (\mathbb{Z}\slash p\mathbb{Z})^{\times}$ (and so $z^{(p-1)^{\ell}} =1$ for every $\ell \in \mathbb{Z}_{\geq 1}$) and for every coefficient $c\equiv 1$ (mod $p$), reducing $\varphi_{(p-1)^{\ell},c}^{m}(z)$ modulo $p$ after each iteration, we then obtain $\varphi_{(p-1)^{\ell},c}^{m}(z)\equiv 2$ (mod $p$) for every fixed $\ell, m\in \mathbb{Z}_{\geq 1}$; and hence $(\varphi_{(p-1)^{\ell},c}^{m}(z))^{(p-1)^{\ell}}\equiv 1$ (mod $p$). But then we note $g(z)=(\varphi_{(p-1)^{\ell},c}^{m}(z))^{(p-1)^{\ell}} - z^{(p-1)^{\ell}}\equiv 0$ (mod $p$) for every point $z\in (\mathbb{Z}\slash p\mathbb{Z})^{\times}$; and so we conclude $M_{c}^{(1_{m})}(p) = p-1$. We now show $M_{c}^{(1_{m})}(p) = p-1$ for every coefficient $c\equiv -1$ (mod $p$) and for every fixed $\ell \in \mathbb{Z}_{\geq 1}$ and even integer $m\in \mathbb{Z}_{\geq 2}$. As before, since $c\equiv -1$ (mod $p$) and also since $z^{(p-1)^{\ell}} =1$ for every $z\in (\mathbb{Z}\slash p\mathbb{Z})^{\times}$, reducing $\varphi_{(p-1)^{\ell},c}^m(z)$ modulo $p$ after each iteration, we then obtain $\varphi_{(p-1)^{\ell},c}^{m}(z)\equiv -1$ (mod $p$) for every fixed $\ell \in \mathbb{Z}_{\geq 1}$ and every fixed even $m\in \mathbb{Z}_{\geq 2}$; and so obtain $(\varphi_{(p-1)^{\ell},c}^{m}(z))^{(p-1)^{\ell}}\equiv 1$ (mod $p$). But then we note $g(z)=(\varphi_{(p-1)^{\ell},c}^{m}(z))^{(p-1)^{\ell}} - z^{(p-1)^{\ell}}\equiv 0$ (mod $p$) for every point $z\in (\mathbb{Z}\slash p\mathbb{Z})^{\times}$; and so conclude $M_{c}^{(1_{m})}(p) = p-1$. 

Finally, we now show $M_{c}^{(1_{m})}(p) = 0$ for every coefficient 
$c\equiv -1$ (mod $p$) and every fixed $\ell \in \mathbb{Z}_{\geq 1}$ and odd integer $m\in \mathbb{Z}_{\geq 1}$ or for every coefficient $c \not \equiv \pm 1, 0$ (mod $p$) and every fixed (even) integer $\ell, m\in \mathbb{Z}_{\geq 1}$. As before, let's for the sake of a contradiction, suppose $g(z)\equiv 0$ (mod $p$) for some nonzero $z\in \mathbb{Z}\slash p\mathbb{Z}$ and for every $c \equiv -1$ (mod $p$) and every fixed $\ell \in \mathbb{Z}_{\geq 1}$ and odd $m\in \mathbb{Z}_{\geq 1}$. Now since $c\equiv -1$ (mod $p$) and $z^{(p-1)^{\ell}}=1$ for every $z\in (\mathbb{Z}\slash p\mathbb{Z})^{\times}$, reducing $\varphi_{(p-1)^{\ell},c}^m(z)$ modulo $p$ after each iteration, we then obtain $\varphi_{(p-1)^{\ell},c}^{m}(z)\equiv 0$ (mod $p$) for every fixed $\ell \in \mathbb{Z}_{\geq 1}$ and odd $m\in \mathbb{Z}_{\geq 1}$; and so $(\varphi_{(p-1)^{\ell},c}^{m}(z))^{(p-1)^{\ell}}\equiv 0$ (mod $p$). But then we note $g(z) = (\varphi_{(p-1)^{\ell},c}^{m}(z))^{(p-1)^{\ell}} - z^{(p-1)^{\ell}}\equiv -1$ (mod $p$) for some point $z\in (\mathbb{Z}\slash p\mathbb{Z})^{\times}$ and for every $c \equiv -1$ (mod $p$) and every fixed $\ell \in \mathbb{Z}_{\geq 1}$ and odd $m\in \mathbb{Z}_{\geq 1}$; and from which we then obtain a contradiction that $0\equiv -1$ (mod $p$). Otherwise, suppose $g(z)\equiv 0$ (mod $p$) for some point $z\in \mathbb{Z}\slash p\mathbb{Z}$ and for every coefficient $c \not \equiv \pm 1, 0$ (mod $p$) and every fixed (even) integers $\ell, m\in \mathbb{Z}_{\geq 1}$. Now recall from earlier that $g(z)=z^{(p-1)^{(m+1)\ell}} -z^{(p-1)^{\ell}} + h(z)$ where $h(z)\in c\mathbb{Z}[z]$, it then also follows by the above supposition that $z^{(p-1)^{(m+1)\ell}} -z^{(p-1)^{\ell}} + h(z)\equiv 0$ (mod $p$) for some $z\in \mathbb{Z}\slash p\mathbb{Z}$ and for every $c \not \equiv \pm 1, 0$ (mod $p$) and every fixed (even) $\ell, m\in \mathbb{Z}_{\geq 1}$. But then we note that the congruence $z^{(p-1)^{(m+1)\ell}} -z^{(p-1)^{\ell}} + h(z)\equiv 0$ (mod $p$) can also happen if $z^{(p-1)^{(m+1)\ell}} -z^{(p-1)^{\ell}}\equiv 0$ (mod $p$) and also $h(z)\equiv 0$ (mod $p$). Moreover, we also note that $z^{(p-1)^{(m+1)\ell}} -z^{(p-1)^{\ell}}\equiv 0$ (mod $p$) for every point $z\equiv 0$ (mod $p$) and every fixed $\ell, m\in \mathbb{Z}_{\geq 1}$, which also happened earlier when $c\equiv 0$ (mod $p$); and so also yielding a contradiction. It then overall follows $g(x)=\varphi_{(p-1)^{\ell},c}^{1+m}(x)-\varphi_{(p-1)^{\ell},c}(x)$ has no roots in $\mathbb{Z}\slash p\mathbb{Z}$ for every coefficient $c\equiv -1$ (mod $p$) and every fixed $\ell \in \mathbb{Z}_{\geq 1}$ and odd $m\in \mathbb{Z}_{\geq 1}$ or every coefficient $c \not \equiv \pm 1, 0$ (mod $p$) and every fixed (even) $\ell, m\in \mathbb{Z}_{\geq 1}$; and so we conclude $M_{c}^{(1_{m})}(p) = 0$. This then completes the whole proof, as needed.
\end{proof}

\begin{rem}\label{re3.5}
As in Rem.\ref{re2.3} and with now Theorem \ref{3.3}, we may then to each distinct $1_{m}$-preperiodic integral point of $\varphi_{(p-1)^{\ell},c}$ associate $1_{m}$-preperiodic integral orbit. In doing so, we then also obtain a dynamical translation of (a quantitative form of \say{all or nothing} principle in) Theorem \ref{3.3}, namely, that the number of distinct $1_{m}$-preperiodic orbits that any $\varphi_{(p-1)^{\ell},c}$ has when iterated on $\mathcal{O}_{K} / p\mathcal{O}_{K}$ is $p$ or $p-1$ or $0$. Again, as mentioned in Intro.\ref{sec1} that the count obtained in Theorem \ref{3.3} on the number of distinct $1_{m}$-preperiodic integral points of any $\varphi_{(p-1)^{\ell},c}$ modulo $p\mathcal{O}_{K}$ may depend on $p$ (and so depend on deg$(\varphi_{(p-1)^{\ell},c})$ for every fixed $\ell \in \mathbb{Z}_{\geq 1}$) and however be independent of $n$ in one of the four possibilities considered; or the count in Theorem \ref{3.3} may neither depend on $p$ nor $n$ in the other possibilities. Moreover, we may also observe that the expected total number (namely, $p + 0 = (p-1)+1+0 = p$ for every fixed eventual period $m\in \mathbb{Z}_{\geq 1}$) of distinct $1_{m}$-preperiodic integral points in the whole family of polynomial maps $\varphi_{(p-1)^{\ell},c}$ modulo $p\mathcal{O}_{K}$ may not only also depend on $p$ and however also be independent of $n$, but may also grow to infinity whenever degree $(p-1)^{\ell}\to \infty$. Worth also observing is that the $1_{m}$-preperiodic point-counting function $M_{c}^{(1_{m})}(p)\to \infty$ or $M_{c}^{(1_{m})}(p)\to 0$ whenever $(p-1)^{\ell}\to \infty$; a somewhat interesting phenomenon coinciding with Remark \ref{re2.3}, however, differing not only from [\cite{BK22}, Remark 3.5], but also from the predicted upper bound $4$ of any $\varphi_{(p-1)^{\ell},c}$ modulo $p$ remarked in \ref{rem1.14} on Hutz's Conject.\ref{conjecture 3.2.1} and consequently also from the predicted upper bound $9$ remarked in \ref{rem1.16} on Hutz's Conject.\ref{conjecture 3.2.2}. As in Rem.\ref{re2.3}, we note that not only is \say{arithmetic governing the expected dynamics} in this setting, but also it might follow in some precise sense that \say{geometry would also govern the expected dynamics}; and as also loosely demonstrated in Sect.\ref{sec4b} where it is observed that the expected total number of distinct $1_{m}$-preperiodic integral points of any $\varphi_{(p-1)^{\ell},c}$ modulo $p\mathcal{O}_{K}$ may not only be independent of $m$, but may also be a constant even when time $m\to \infty$. 
\end{rem}

As in Section \ref{sec2}, we also note that as an immediate consequence of the expected count obtained in Theorem \ref{3.3} and in [\cite{BK22}, Theorem 3.3] on any map $\varphi_{(p-1)^{\ell},c}$ modulo $p\mathcal{O}_{K}$, we then also prove the following corollary on the expected total number of distinct strictly $1_{m}$-preperiodic integral points of any map $\varphi_{(p-1)^{\ell},c}$ modulo $p\mathcal{O}_{K}$:

\begin{cor}\label{cor.3.6}
Let $K\slash \mathbb{Q}$ be any number field of degree $n\geq 1$ with the ring of integers $\mathcal{O}_{K}$, and in which any fixed prime $p\geq 5$ is inert. Then the expected total number of distinct strictly $1_{m}$-preperiodic points of $\varphi_{(p-1)^{\ell},c}$ modulo $p\mathcal{O}_{K}$ is equal to $p-3$ for any odd $m\in \mathbb{Z}_{\geq 1}$ or $p-4$ for any even $m\in \mathbb{Z}_{\geq 2}$. Moreover, the expected total number of distinct strictly $1_{m}$-preperiodic points of $\varphi_{(p-1)^{\ell},c}$ modulo $p\mathcal{O}_{K}$ is equal to $p-4$ for every $m\in \mathbb{Z}_{\geq 1}$.
\end{cor}

\begin{proof}
Since we noted in Remark \ref{re3.5} that for any fixed $m\in \mathbb{Z}_{\geq 1}$, the expected total number of distinct $1_{m}$-preperiodic integral points of $\varphi_{(p-1)^{\ell},c}$ modulo $p\mathcal{O}_{K}$ is equal to $p$; and also since we noted in [\cite{BK22}, Rem.3.5] that the expected total number of distinct $m$-periodic integral points of any $\varphi_{(p-1)^{\ell},c}$ modulo $p$ is equal to $3$ for any fixed odd $m\in \mathbb{Z}_{\geq 1}$ or equal to $4$ for any fixed even $m\in \mathbb{Z}_{\geq 2}$, we then note that the expected total number of distinct strictly $1_{m}$-preperiodic integral points of $\varphi_{(p-1)^{\ell},c}$ modulo $p\mathcal{O}_{K}$ is equal to the expected total number of distinct $1_{m}$-preperiodic integral points of $\varphi_{(p-1)^{\ell},c}$ modulo $p\mathcal{O}_{K}$ minus the expected total number of distinct $m$-periodic integral points of $\varphi_{(p-1)^{\ell},c}$ modulo $p\mathcal{O}_{K}$, i.e., the expected total number of distinct strictly $1_{m}$-preperiodic integral points of $\varphi_{(p-1)^{\ell},c}$ modulo $p\mathcal{O}_{K}$ is equal to $p-3$ for any fixed odd $m\in \mathbb{Z}_{\geq 1}$ or equal to $p-4$ for any fixed even $m\in \mathbb{Z}_{\geq 2}$. Since we know from [\cite{BK22}, Remark 3.6] that for every fixed $m\in \mathbb{Z}_{\geq 1}$, the expected total number of distinct odd $m$-periodic and even $m$-periodic integral points of any $\varphi_{(p-1)^{\ell},c}$ modulo $p\mathcal{O}_{K}$ is equal to $4$, we then note that the expected total number of distinct strictly $1_{m}$-preperiodic integral points of $\varphi_{(p-1)^{\ell},c}$ modulo $p\mathcal{O}_{K}$ is equal to the expected total number of distinct $1_{m}$-preperiodic integral points of $\varphi_{(p-1)^{\ell},c}$ modulo $p\mathcal{O}_{K}$ minus the expected total number of distinct odd $m$-periodic and even $m$-periodic integral points of $\varphi_{(p-1)^{\ell},c}$ modulo $p\mathcal{O}_{K}$, i.e., the expected total number of distinct strictly $1_{m}$-preperiodic integral points of $\varphi_{(p-1)^{\ell},c}$ modulo $p\mathcal{O}_{K}$ is equal to $p-4$. This then completes the whole proof, as desired.
\end{proof}

\begin{rem}
As in Rem.\ref{re3.5}, we also note that a dynamical translation of Corollary \ref{cor.3.6} is that the total number of distinct strictly $1_{m}$-preperiodic integral orbits of any $\varphi_{(p-1)^{\ell},c}$ iterated on $\mathcal{O}_{K} / p\mathcal{O}_{K}$ is $p-4$ for any $m\in \mathbb{Z}_{\geq 1}$.
\end{rem}

\begin{rem}
As in Remark \ref{re2.6}, we may also recall from Theorem \ref{3.3} that for every fixed (eventual period) $m\in \mathbb{Z}_{\geq 1}$, the $1_{m}$-preperiodic point-counting function $M_{c}^{(1_{m})}(p) = p$ for every fixed inert $p$ and every $c$ divisible by $p$ or $M_{c}^{(1_{m})}(p) =0$ for every fixed inert $p$ and every $c$ indivisible by $p$. But now recalling also Remark \ref{re2.6}, we then note that for every fixed (eventual period) $m\in \mathbb{Z}_{\geq 1}$, the point-counting function $M_{c}^{(1_{m})}(p) = N_{c}^{(1_{m})}(p) = p$ for every fixed inert $p\geq 5$ and every $c$ divisible by $p$ or $M_{c}^{(1_{m})}(p) = N_{c}^{(1_{m})}(p) = 0$ for every fixed inert $p\geq 5$ and every $c$ indivisible by $p$. Moreover, from the argument of the first part of the Proof of Theorem \ref{3.3} and of Theorem \ref{2.3}, we then also note that every $1_{m}$-preperiodic integral point of $\varphi_{(p-1)^{\ell},c}$ modulo $p\mathcal{O}_{K}$ is $1_{m}$-preperiodic integral point of $\varphi_{p^{\ell},c}$ modulo $p\mathcal{O}_{K}$ and so (by Remark \ref{re2.6}) $1$-periodic integral point of $\varphi_{p^{\ell},c}$ modulo $p\mathcal{O}_{K}$. But then we note that with $\varphi_{(p-1)^{\ell},c}^{1+m}(z) - \varphi_{(p-1)^{\ell},c}(z)\equiv \varphi_{p^{\ell},c}(z)\equiv z $ (mod $p\mathcal{O}_{K}$) and so $\varphi_{(p-1)^{\ell},c}^{1+m}(z) \equiv z + \varphi_{(p-1)^{\ell},c}(z)$ (mod $p\mathcal{O}_{K}$), we then note that every $1_{m}$-preperiodic integral point of $\varphi_{(p-1)^{\ell},c}$ modulo $p\mathcal{O}_{K}$ is $(m+1)$-periodic integral point of $\varphi_{(p-1)^{\ell},c}$ modulo $p\mathcal{O}_{K}$, whenever $\varphi_{(p-1)^{\ell},c}(z)\equiv 0$ (mod $p\mathcal{O}_{K}$); and in which case we would then proceed from [\cite{BK22}, Remark 3.5 and 3.6] to conclude that expected total number of distinct $1_{m}$-preperiodic integral points in the whole family of maps $\varphi_{(p-1)^{\ell},c}$ modulo $p\mathcal{O}_{K}$ is equal to $3$ (for every fixed odd $m\in \mathbb{Z}_{\geq 1}$) or $4$ (for every fixed even $m\in \mathbb{Z}_{\geq 2}$); and which would then be a somewhat interesting precise arithmetic-geometric insight on all the $1_{m}$-preperiodic integral orbits of any map  $\varphi_{(p-1)^{\ell},c}$ modulo $p\mathcal{O}_{K}$.
\end{rem}

\section{On Dynamical Complexity of Forward PrePeriodic Orbit Structure of $\varphi_{p^{\ell},c}$ \& $\varphi_{(p-1)^{\ell},c}$}\label{sec4b}

Observe in Theorem \ref{2.3} that $N_{c}^{(1_{m})}(p)$ is independent of eventual period $m\in \mathbb{Z}_{\geq 1}$, and moreover we may have 
\begin{center}
    $\lim\limits_{m\to \infty} N_{c}^{(1_{m})}(p) = p$ or $0$.
\end{center}

As in \cite{BK22} viewing $N_{c}^{(1_{m})}(p)$ as now $1_{m}$-preperiodic orbit-counting function, we in this section wish to investigate very mildly the complexity of a discrete dynamical system $(\mathcal{O}_{K}\slash p\mathcal{O}_{K}$, $\varphi_{p^{\ell},c}$ modulo $p\mathcal{O}_{K}$). To that end, we do so via analyzing the behavior of exponential growth rate $\rho(\varphi_{p^{\ell},c})$ \cite{Kat}; and which shows $N_{c}^{(1_{m})}(p)$ grows by a factor $1=e^{\rho(\varphi_{p^{\ell},c})}$ whenever time $m\to \infty$ and so $N_{c}^{(1_{m})}(p)$ is a constant function. That is, we have: 

\begin{cor}\label{4.1}
Assume Theorem \ref{2.3}, and let $m\in \mathbb{Z}_{\geq 1}$ be any eventual period. Then the exponential growth rate of $1_{m}$-preperiodic orbit-counting function $N_{c}^{(1_{m})}(p)$ exists and is equal to zero. More precisely, we have 

\begin{center}
    $\rho(\varphi_{p^{\ell},c}) := \limsup\limits_{m\to \infty}\frac{\textnormal{log }(\textnormal{max} \{N_{c}^{(1_{m})}(p),1\})}{m} = 0$.
\end{center} 
\end{cor}
\begin{proof}
Since we know from Theorem \ref{2.3} that $N_{c}^{(1_{m})}(p) = p\text{ or } 0$ for any fixed eventual period $m\in \mathbb{Z}_{\geq 1}$, we then obtain $\frac{\text{log }(\text{max} \{N_{c}^{(1_{m})}(p),1\})}{m} = \frac{\textnormal{log }p}{m}$ or $0$. Now letting eventual period $m\to \infty$, we then obtain $\rho(\varphi_{p^{\ell},c}) = 0$.
\end{proof}
Similarly, observe in Theorem \ref{3.3} that $M_{c}^{(1_{m})}(p)$ is also independent of eventual period $m$, and moreover  
\begin{center}
    $\lim\limits_{m\to \infty}M_{c}^{(1_{m})}(p) = p, p-1 \text{ or } 0.$
\end{center}So now, as before we may also investigate very mildly the complexity of a discrete dynamical system $(\mathcal{O}_{K}\slash p\mathcal{O}_{K}$, $\varphi_{(p-1)^{\ell},c}$ modulo $p\mathcal{O}_{K}$) by determining the behavior of the exponential growth rate $\rho(\varphi_{(p-1)^{\ell},c})$ associated to such a system. With that in mind, we then obtain $M_{c}^{(1_{m})}(p)$ grows by a factor $1=e^{\rho(\varphi_{(p-1)^{\ell},c})}$ whenever time $m\to \infty$ and so $1_{m}$-preperiodic orbit-counting function $M_{c}^{(1_{m})}(p)$ is also a constant function. That is, we have:
\begin{cor}
Assume Theorem \ref{3.3}, and let $m\in \mathbb{Z}_{\geq 1}$ be any eventual period. Then the exponential growth rate of $1_{m}$-preperiodic orbit-counting function $M_{c}^{(1_{m})}(p)$ exists and is equal to zero. More precisely, we have 
\begin{center}
    $\rho(\varphi_{(p-1)^{\ell},c}) := \limsup\limits_{m\to \infty}\frac{\textnormal{log }(\textnormal{max} \{M_{c}^{(1_{m})}(p),1\})}{m} = 0$.
\end{center}
\end{cor}
\begin{proof}
By applying a similar argument as in the Proof of Cor. \ref{4.1}, we then obtain $\rho(\varphi_{(p-1)^{\ell},c})=0$ as desired.
\end{proof}

\section{On Average Number of $1_{m}$-PrePeriodic Points of any Polynomial Map $\varphi_{p^{\ell},c}$ \& $\varphi_{(p-1)^{\ell},c}$}\label{sec4}

In this section, we wish to restrict on $\mathbb{Z}\subset \mathcal{O}_{K}$ and then study the behavior of the function $N_{c}^{(1_{m})}(p)$ as $c\to \infty$. More precisely, we wish to determine: \say{\textit{What is the average value of the function $N_{c}^{(1_{m})}(p)$ as $c \to \infty$?}} The following corollary shows that the average value of the function $N_{c}^{(1_{m})}(p)$ may be zero or unbounded as $c\to \infty$:

\begin{cor}\label{co5.1}
Let $K\slash \mathbb{Q}$ be any number field of degree $n \geq 1$ with ring of integers $\mathcal{O}_{K}$, and in which any prime $p\geq 3$ is inert. Then the average value of $N_{c}^{(1_{m})}(p)$ is zero or unbounded as $c\to\infty$. More precisely, we have
\begin{myitemize}
    \item[\textnormal{(a)}] \textnormal{Avg} $N^{(1_{m})}_{c\neq pt}(p):= \lim\limits_{c \to\infty} \Large{\frac{\sum\limits_{3\leq p\leq c, \ p\nmid c \textnormal{ in } \mathcal{O}_{K}}N_{c}^{(1_{m})}(p)}{\Large{\sum\limits_{3\leq p\leq c, \ p\nmid c \textnormal{ in } \mathcal{O}_{K}}1}}} =  0$. 
    
    \item[\textnormal{(b)}] \textnormal{Avg} $N^{(1_{m})}_{c = pt}(p):= \lim\limits_{c \to\infty} \Large{\frac{\sum\limits_{3\leq p\leq c, \ p\mid c \textnormal{ in } \mathcal{O}_{K}}N_{c}^{(1_{m})}(p)}{\Large{\sum\limits_{3\leq p\leq c, \ p\mid c \textnormal{ in } \mathcal{O}_{K}}1}}} =  \infty$.
\end{myitemize}

\end{cor}
\begin{proof}
Since we know from Theorem \ref{2.3} that the number $N_{c}^{(1_{m})}(p) = 0$ for any inert prime $p\nmid c$ in $\mathcal{O}_{K}$, we then obtain $\lim\limits_{c\to\infty} \Large{\frac{\sum\limits_{3\leq p\leq c, \ p\nmid c \text{ in } \mathcal{O}_{K}}N_{c}^{(1_{m})}(p)}{\Large{\sum\limits_{3\leq p\leq c, \ p\nmid c \text{ in } \mathcal{O}_{K}}1}}} = 0$; and so the average value Avg $N_{c \neq pt}^{(1_{m})}(p) = 0$, as desired for (a). To see (b), we note that since we know from Theorem \ref{2.3} that $N_{c}^{(1_{m})}(p) = p$ for any inert prime $p\mid c$ in $\mathcal{O}_{K}$, it then follows $\sum\limits_{3\leq p\leq c, \ p\mid c \textnormal{ in } \mathcal{O}_{K}} N_{c}^{(1_{m})}(p) = \sum\limits_{3\leq p\leq c, \ p\mid c \textnormal{ in } \mathcal{O}_{K}}p =: \sigma_{1,p}(c)$ and  $\sum\limits_{3\leq p\leq c, \ p\mid c \textnormal{ in } \mathcal{O}_{K}} 1  = \omega(c)$, where $\sigma_{1}(n)$ (resp. $\omega(n)$) is the number of divisors (resp. the number of distinct prime divisors) of any $n\in \mathbb{Z}_{\geq 1}$; and so obtain $\frac{\sum\limits_{3\leq p\leq c, \ p\mid c \textnormal{ in } \mathcal{O}_{K}} N_{c}^{(1_{m})}(p)}{\sum\limits_{3\leq p\leq c, \ p\mid c \textnormal{ in } \mathcal{O}_{K}} 1} = \frac{\sigma_{1,p}(c)}{\omega(c)}$. But then applying a similar argument as in [\cite{BK3}, Proof of Cor. 7.1(c)], we then obtain that the average value Avg $N_{c = pt}^{(1_{m})}(p) = \infty$. This then completes the whole proof, as desired.  
\end{proof}

\begin{rem} \label{Rem5.2}
From arithmetic statistics to arithmetic dynamics, we note that Corollary \ref{co5.1} shows that any $\varphi_{p^{\ell},c}$ iterated on the space $\mathcal{O}_{K} / p\mathcal{O}_{K}$ has on average zero or unbounded number of distinct $1_{m}$-preperiodic integral orbits as $c\to \infty$; a somewhat interesting averaging phenomenon coinciding with an averaging phenomenon remarked in \cite{BK3, BK22} on the average number of distinct $m$-periodic integral orbits of $\varphi_{p^{\ell},c}$  iterated on $\mathcal{O}_{K} / p \mathcal{O}_{K}$.
\end{rem}

For any fixed prime $p\geq 3$ and for any fixed $m\in \mathbb{Z}_{\geq1}$, let $N_{c}^{+(1_{m})}(p)$ be the expected total number of distinct strictly $1_{m}$-preperiodic integral points of any polynomial map $\varphi_{p^{\ell}, c}$ modulo $p\mathcal{O}_{K}$. So now, unlike in Corollary \ref{co5.1} in which the average value of the number $N_{c}^{+(1_{m})}(p)$ can be unbounded, the following corollary shows that the average value of the number $N_{c}^{+(1_{m})}(p)$ is equal to zero and therefore always bounded as $c\to \infty$:

\begin{cor}\label{cor5A}
Let $p\geq 3$ be any prime, and $m\in \mathbb{Z}_{\geq 1}$ be any fixed integer. Let $N_{c}^{+(1_{m})}(p)$ be the number defined as before. Then the average value of $N_{c}^{+(1_{m})}(p)$ is equal to zero as $c\to\infty$. \textit{More precisely, we have} 
\begin{center}
    \textnormal{Avg} $N_{c = pt}^{+(1_{m})}(p) := \lim\limits_{c\to\infty} \Large{\frac{\sum\limits_{3\leq p\leq c, \ p\mid c \textnormal{ in } \mathcal{O}_{K}}N_{c}^{+(1_{m})}(p)}{\Large{\sum\limits_{3\leq p\leq c, \ p\mid c \textnormal{ in } \mathcal{O}_{K}}1}}} = 0.$ 
\end{center}
\end{cor}

\begin{proof}
Since we know from Corollary \ref{co2.6} that the number $N_{c}^{+(1_{m})}(p)=p-p=0$ for every $p$ and every fixed $m\in \mathbb{Z}_{\geq 1}$, and also since we may recall in [\cite{BK333}, Rem.2.5] and in Rem.\ref{re2.3} that the count $p$ was obtained from the possibility when the coefficient $c$ was divisible by inert $p$, then averaging over primes $p\mid c$ and letting $c\to \infty$, we then obtain $\lim\limits_{c\to\infty} \Large{\frac{\sum\limits_{3\leq p\leq c, \ p\mid c \text{ in } \mathcal{O}_{K}}N_{c}^{+(1_{m})}(p)}{\Large{\sum\limits_{3\leq p\leq c, \ p\mid c \text{ in } \mathcal{O}_{K}}1}}} = 0$; from which we then conclude \textnormal{Avg} $N_{c = pt}^{+(1_{m})}(p) = 0$, as needed.  
\end{proof}

\begin{rem} 
As in Remark \ref{Rem5.2}, we then also note that Corollary \ref{cor5A} shows that on average, the expected total number of distinct strictly $1_{m}$-preperiodic integral orbits of any map $\varphi_{p^{\ell},c}$ modulo $p\mathcal{O}_{K}$ is zero even as $c\to \infty$.
\end{rem}

Similarly, we also wish to determine: \say{\textit{What is the average value of the function $M_{c}^{(1_{m})}(p)$ as $c \to \infty$?}} The following corollary shows that the average value of $M_{c}^{(1_{m})}(p)$ may also be zero or unbounded as $c\to \infty$:

\begin{cor}\label{co5.5}
Let $K\slash \mathbb{Q}$ be any number field of degree $n \geq 1$ with ring of integers $\mathcal{O}_{K}$, and in which any prime $p\geq 5$ is inert. Then the average value of $M_{c}^{(1_{m})}(p)$ is zero or unbounded as $c\to\infty$. Specifically, we have 

\begin{myitemize}
    \item[\textnormal{(a)}] \textnormal{Avg} $M_{c \not \equiv\pm1, 0 \ (\textnormal{mod }p)}^{(1_{m})}(p) := \lim\limits_{c\to\infty} \Large{\frac{\sum\limits_{5\leq p\leq c, \ p\nmid c, (c\pm1) \textnormal{ or }5\leq p\leq (c+1), \ p\mid (c+1) \textnormal{ and } m=2k+1}M_{c}^{(1_{m})}(p)}{\Large{\sum\limits_{5\leq p\leq c, \ p\nmid c, (c\pm1) \textnormal{ or }5\leq p\leq (c+1), \ p\mid (c+1) \textnormal{ and } m=2k+1}1}}} = 0.$

    \item[\textnormal{(b)}] \textnormal{Avg} $M_{c \equiv\pm1, 0 \ (\textnormal{mod }p)}^{(1_{m})}(p):= \lim\limits_{c \to\infty} \Large{\frac{\sum\limits_{5\leq p\leq c, \ p\mid c \textnormal{ or }  5\leq p\leq (c\pm 1), \ p\mid (c\pm 1) \ (\textnormal{and} \ m=2k)}M_{c}^{(1_{m})}(p)}{\Large{\sum\limits_{5\leq p\leq c, \ p\mid c \textnormal{ or }  5\leq p\leq (c\pm 1), \ p\mid (c\pm 1) \ (\textnormal{and} \ m=2k)}1}}} =  \infty$.
\end{myitemize}
\end{cor}

\begin{proof}
Recalling from Theorem \ref{3.3} that $M_{c}^{(1_{m})}(p) = 0$ for any inert $p\nmid c, (c\pm1)$ and fixed (even) $m\in \mathbb{Z}_{\geq 2}$ or inert $p\mid (c+1)$ and  fixed odd $m\in \mathbb{Z}_{\geq 1}$, we then note $\lim\limits_{c\to\infty} \Large{\frac{\sum\limits_{5\leq p\leq c, \ p\nmid c, (c\pm1) \textnormal{ or }5\leq p\leq (c+1), \ p\mid (c+1) \textnormal{ and } m=2k+1}M_{c}^{(1_{m})}(p)}{\Large{\sum\limits_{5\leq p\leq c, \ p\nmid c, (c\pm1) \textnormal{ or }5\leq p\leq (c+1), \ p\mid (c+1) \textnormal{ and } m=2k+1}1}}} = 0$; and so we conclude \textnormal{Avg} $M_{c \not \equiv\pm1, 0 \ (\textnormal{mod }p)}^{(1_{m})}(p)=0$, as desired for (a). To see (b), we first recall from Theorem \ref{3.3} that the number $M_{c}^{(1_{m})}(p) = p$ for any inert $p\mid c$ or that the number $M_{c}^{(1_{m})}(p) = p-1$ for any inert $p\mid (c\pm 1)$ and fixed (even) $m\in \mathbb{Z}_{\geq 2}$. But then applying a similar argument as in the Proof of Cor. \ref{co5.1}(b), we then conclude the average value \textnormal{Avg} $M_{c \equiv\pm1, 0 \ (\textnormal{mod }p)}^{(1_{m})}(p)=\infty$. This then completes the whole proof, as required. 
\end{proof} 

\begin{rem} \label{4.4}
As before, we note that from arithmetic statistics to arithmetic dynamics, Corollary \ref{co5.5} shows that any $\varphi_{(p-1)^{\ell},c}$ iterated on the space $\mathcal{O}_{K} / p \mathcal{O}_{K}$ has on average zero or unbounded number distinct $1_{m}$-preperiodic orbits as $c\to \infty$; a somewhat interesting averaging phenomenon coinciding precisely with an averaging phenomenon in Remark \ref{Rem5.2}, however, differing very significantly from an averaging phenomenon observed in \cite{BK2, BK22} on the number of distinct $m$-periodic integral points of every map  $\varphi_{(p-1)^{\ell}, c}$ modulo $p\mathcal{O}_{K}$.
\end{rem}

For any fixed prime $p\geq 5$ and for any fixed $m\in \mathbb{Z}_{\geq1}$, let $M_{c}^{+(1_{m})}(p)$ be the expected total number of distinct strictly $1_{m}$-preperiodic integral points of any $\varphi_{(p-1)^{\ell}, c}$ modulo $p\mathcal{O}_{K}$. As in Cor.\ref{co5.5} in which the average value can be unbounded, the following corollary shows that the average value of $M_{c}^{+(1_{m})}(p)$ is also unbounded:

\begin{cor}\label{cor5B}
Let $p\geq 5$ be any prime, and $m\in \mathbb{Z}_{\geq 1}$ be any fixed integer. Let $M_{c}^{+(1_{m})}(p)$ be the number defined as before. Then the average value of $M_{c}^{+(1_{m})}(p)$ is unbounded as $c\to\infty$. \textit{Specifically, we have} 
\begin{center}
    \textnormal{Avg} $M_{c \equiv 0 \ (\textnormal{mod }p)}^{+(1_{m})}(p) := \lim\limits_{c\to\infty} \Large{\frac{\sum\limits_{5\leq p\leq c, \ p\mid c \textnormal{ in } \mathcal{O}_{K}}M_{c}^{+(1_{m})}(p)}{\Large{\sum\limits_{5\leq p\leq c, \ p\mid c  \textnormal{ in } \mathcal{O}_{K}}1}}} = \infty.$ 
\end{center}
\end{cor}

\begin{proof}
Since we know from Corollary \ref{cor.3.6} that $M_{c}^{+(1_{m})}(p)=p-4=0$ for every $p\geq 5$ and for every fixed $m\in \mathbb{Z}_{\geq 1}$, and also since we may recall in Rem.\ref{re3.5} that the count $p$ in the difference $p-4$ was obtained from the possibility when the coefficient $c$ was divisible by inert $p$, then averaging over inert primes $p\mid c$ and then applying a similar argument as in the Proof of Corollary \ref{co5.5}, we then obtain $\lim\limits_{c\to\infty} \Large{\frac{\sum\limits_{5\leq p\leq c, \ p\mid c \textnormal{ in } \mathcal{O}_{K}}M_{c}^{+(1_{m})}(p)}{\Large{\sum\limits_{5\leq p\leq c, \ p\mid c \textnormal{ in } \mathcal{O}_{K}}1}}} = \infty$; from which we then conclude \textnormal{Avg} $M_{c \equiv 0 \ (\textnormal{mod }p)}^{+(1_{m})}(p) = \infty$. This then completes the whole proof, as desired.  
\end{proof}

\begin{rem} 
As in Rem.\ref{4.4}, we then also note that Cor.\ref{cor5B} shows that on average, the expected total number of distinct strictly $1_{m}$-preperiodic integral orbits of every $\varphi_{(p-1)^{\ell},c}$ modulo $p\mathcal{O}_{K}$ is unbounded even as $c\to \infty$.
\end{rem}

\section{On Density of Monic Integer Polynomials $\varphi_{p^{\ell},c}(x)\in \mathcal{O}_{K}[x]$ having Number $N_{c}^{(1_{m})}(p) = p$}\label{sec5}

As in [\cite{BK22}, Section 6] we in this and the next section, wish to restrict our counting on the subring $\mathbb{Z}\subset \mathcal{O}_{K}$ and then determine: \say{\textit{For a prime $p\geq 3$ and for any fixed $\ell \in \mathbb{Z}_{\geq 1}$ and fixed eventual period $m \in \mathbb{Z}_{\geq 1}$, what is the density of monic integer polynomials $\varphi_{p^{\ell},c}(x)\in \mathbb{Z}[x]$ with $p$ distinct $1_{m}$-preperiodic integral points modulo $p$?}} The following corollary shows that for any fixed $\ell \in \mathbb{Z}_{\geq 1}$ and any fixed eventual period $m \in \mathbb{Z}_{\geq 1}$, there are very few monic integer polynomials $\varphi_{p^{\ell},c}(x)\in \mathbb{Z}[x]$ with exactly $p$ distinct $1_{m}$-preperiodic integral points modulo $p$:

\begin{cor}\label{6.1}
Let $K\slash \mathbb{Q}$ be any number field of degree $n\geq 1$ with the ring of integers $\mathcal{O}_{K}$, and in which any prime $p\geq 3$ is inert. Let $\ell, m \in \mathbb{Z}_{\geq 1}$ be any fixed integers. Then the density of monic integer polynomials $\varphi_{p^{\ell},c}(x) = x^{p^{\ell}} + c\in \mathcal{O}_{K}[x]$ with $N_{c}^{(1_{m})}(p) = p$ exists and is equal to $0 \%$ as $c\to \infty$. More precisely, we have 
\begin{center}
    $\lim\limits_{c\to\infty} \Large{\frac{\# \{\varphi_{p^{\ell},c}(x)\in \mathbb{Z}[x] \ : \ 3\leq p\leq c \ \textnormal{and} \ N_{c}^{(1_{m})}(p) \ = \ p\}}{\Large{\# \{\varphi_{p^{\ell},c}(x) \in \mathbb{Z}[x] \ : \ 3\leq p\leq c \}}}} = \ 0.$
\end{center}
\end{cor}
\begin{proof}
Since the defining condition $N_{c}^{(1_{m})}(p) = p$ is as we proved in Theorem \ref{2.3} and hence in Corollary \ref{cor2.4} determined whenever $c$ is divisible by $p$, we may then count $\# \{\varphi_{p^{\ell},c}(x) \in \mathbb{Z}[x] : 3\leq p\leq c \ \text{and} \ N_{c}^{(1_{m})}(p) \ = \ p\}$ by counting the number $\# \{\varphi_{p^{\ell},c}(x)\in \mathbb{Z}[x] : 3\leq p\leq c \ \text{and} \ p\mid c \ \text{for \ any \ fixed} \ c \}$. In that case, we then write 
\begin{center}
$\Large{\frac{\# \{\varphi_{p^{\ell},c}(x) \in \mathbb{Z}[x] \ : \ 3\leq p\leq c \ \text{and} \ N_{c}^{(1_{m})}(p) \ = \ p\}}{\Large{\# \{\varphi_{p^{\ell},c}(x) \in \mathbb{Z}[x] \ : \ 3\leq p\leq c \}}}} = \Large{\frac{\# \{\varphi_{p^{\ell},c}(x)\in \mathbb{Z}[x] \ : \ 3\leq p\leq c \ \text{and} \ p\mid c \ \text{for any fixed} \ c \}}{\Large{\# \{\varphi_{p^{\ell},c}(x) \in \mathbb{Z}[x] \ : \ 3\leq p\leq c \}}}}$. 
\end{center}\indent Moreover, for any fixed integer $c\geq 3$, the numerator of the foregoing quotient may be rewritten to then obtain
\begin{center}
$\# \{\varphi_{p^{\ell},c}(x) \in \mathbb{Z}[x] : 3\leq p\leq c \ \text{and} \ p\mid c \} = \# \{p : 3\leq p\leq c \text{ and } p\mid c \} = \sum_{3\leq p\leq c, \ p\mid c}1 = \omega (c)$, 
\end{center}where $\omega(n)$ is by definition the number of distinct prime factors of $n$. Writing $\# \{\varphi_{p^{\ell},c}(x) \in \mathbb{Z}[x]  : 3\leq p\leq c \} = \sum_{3\leq p\leq c} 1 = \pi(c)$, where $\pi(n)$ is by definition the number of primes at most $n$, we then note that the quotient 
\begin{center}
$\Large{\frac{\# \{\varphi_{p^{\ell},c}(x)\in \mathbb{Z}[x] \ : \ 3\leq p\leq c \ \text{and} \ p\mid c \ \text{for any fixed} \ c \}}{\Large{\# \{\varphi_{p^{\ell},c}(x)\in \mathbb{Z}[x] \ : \ 3\leq p\leq c \}}}} = \frac{\omega(c)}{\pi(c)}$.
\end{center}
\indent But now by applying a similar argument as in [\cite{BK22}, Proof of Cor. 6.1], we then obtain the limit $0$ as desired.
\end{proof}

\noindent Note that one may also interpret Cor. \ref{6.1} as saying that for any fixed $\ell \in \mathbb{Z}_{\geq 1}$ and fixed eventual period $m \in \mathbb{Z}_{\geq 1}$, the probability of choosing randomly a monic polynomial $\varphi_{p^{\ell},c}(x)\in \mathbb{Z}[x]\subset \mathcal{O}_{K}[x]$ with exactly $p$ distinct $1_{m}$-preperiodic integral points modulo $p$ is equal to zero;  a somewhat interesting probabilistic phenomenon coinciding with a phenomenon remarked in [\cite{BK3, BK22}, Corollary 9.1, Corollary 6.1, resp.] on the probability of choosing randomly a monic polynomial $\varphi_{p^{\ell},c}(x)\in \mathbb{Z}[x]$ having $p$ distinct $m$-periodic integral points modulo $p$.

\section{On Density of Monic Integer Polynomials $\varphi_{(p-1)^{\ell},c}(x)\in \mathcal{O}_{K}[x]$ with $M_{c}^{(1_{m})}(p) = p$ or $p-1$}\label{sec6}

As in Section \ref{sec5}, we also wish to determine: \say{\textit{For a prime $p\geq 5$ and for any fixed $\ell\in \mathbb{Z}_{\geq 1}$ and fixed eventual period $m\in \mathbb{Z}_{\geq 1}$, what is the density of monics $\varphi_{(p-1)^{\ell},c}(x)\in \mathbb{Z}[x]$ with $p$ distinct $1_{m}$-preperiodic integral points modulo $p$?}} The following corollary shows that for any fixed $\ell\in \mathbb{Z}_{\geq 1}$ and fixed eventual period $m\in \mathbb{Z}_{\geq 1}$, there are very few monic polynomials $\varphi_{(p-1)^{\ell},c}(x)\in \mathbb{Z}[x]$ having $p$ distinct $1_{m}$-preperiodic integral points modulo $p$:

\begin{cor}\label{7.1}
Let $K\slash \mathbb{Q}$ be any number field of degree $n\geq 1$ with the ring of integers $\mathcal{O}_{K}$, and in which any prime $p\geq 5$ is inert. Let $\ell, m \in \mathbb{Z}_{\geq 1}$ be any fixed integers. Then the density of monic integer polynomials $\varphi_{(p-1)^{\ell},c}(x) = x^{(p-1)^{\ell}} + c\in \mathcal{O}_{K}[x]$ with $M_{c}^{(1_{m})}(p) = p$ is equal to $0 \%$ as $c\to \infty$. More precisely, we have 
\begin{center}
    $\lim\limits_{c\to\infty} \Large{\frac{\# \{\varphi_{(p-1)^{\ell},c}(x) \in \mathbb{Z}[x]\ : \ 5\leq p\leq c \ \textnormal{and} \ M_{c}^{(1_{m})}(p) \ = \ p\}}{\Large{\# \{\varphi_{(p-1)^{\ell},c}(x) \in \mathbb{Z}[x]\ : \ 5\leq p\leq c \}}}} = \ 0.$
\end{center}
\end{cor}
\begin{proof}
Because $M_{c}^{(1_{m})}(p) = p$ is as proved in Thm. \ref{3.2} and thus in Cor. \ref{cor3.4} determined when the coefficient $c$ is divisible by $p$, we may then count the number $\# \{\varphi_{(p-1)^{\ell},c}(x) \in \mathbb{Z}[x] : 5\leq p\leq c \ \text{and} \ M_{c}^{(1_{m})}(p) \ = \ p\}$ by counting the number $\# \{\varphi_{(p-1)^{\ell},c}(x)\in \mathbb{Z}[x] : 5\leq p\leq c \ \text{and} \ p\mid c \ \text{for \ any \ fixed} \ c \}$. In that case, we then write 
\begin{center}
$\Large{\frac{\# \{\varphi_{(p-1)^{\ell},c}(x) \in \mathbb{Z}[x] \ : \ 5\leq p\leq c \ \text{and} \ M_{c}^{(1_{m})}(p) \ = \ p\}}{\Large{\# \{\varphi_{(p-1)^{\ell},c}(x) \in \mathbb{Z}[x] \ : \ 5\leq p\leq c \}}}} = \Large{\frac{\# \{\varphi_{(p-1)^{\ell},c}(x)\in \mathbb{Z}[x] \ : \ 5\leq p\leq c \ \text{and} \ p\mid c \ \text{for any fixed} \ c \}}{\Large{\# \{\varphi_{(p-1)^{\ell},c}(x) \in \mathbb{Z}[x] \ : \ 5\leq p\leq c \}}}}$. 
\end{center}\indent But then we note that applying a similar argument as in the Proof of Corollary \ref{6.1}, we then immediately obtain that the limit exists and moreover the limit is equal to $0$. This then completes the whole proof, as desired.
\end{proof} 

\noindent As before, we may also interpret Cor.\ref{7.1} as saying that for any fixed $\ell \in \mathbb{Z}_{\geq 1}$ and fixed eventual period $m \in \mathbb{Z}_{\geq 1}$, the probability of choosing randomly $\varphi_{(p-1)^{\ell},c}(x)\in \mathbb{Z}[x]$ with $p$ distinct $1_{m}$-preperiodic integral points modulo $p$ is zero; a somewhat interesting probabilistic phenomenon coinciding with a phenomenon remarked in Sect.\ref{sec5}.

The following corollary shows that for any fixed $\ell \in \mathbb{Z}_{\geq 1}$ and eventual period $m\in \mathbb{Z}_{\geq 1}$, the probability of choosing randomly $\varphi_{(p-1)^{\ell},c}(x)\in \mathbb{Z}[x]$ with $p-1$ distinct $1_{m}$-preperiodic integral points modulo $p$ is also zero:

\begin{cor}\label{7.2}
Let $K\slash \mathbb{Q}$ be any number field of degree $n\geq 1$ with the ring of integers $\mathcal{O}_{K}$, and in which any prime $p\geq 5$ is inert. Let $\ell, m\in \mathbb{Z}_{ \geq 1}$ be any fixed integers. Then the density of monic integer polynomials $\varphi_{(p-1)^{\ell},c}(x) = x^{(p-1)^{\ell}} + c\in \mathcal{O}_{K}[x]$ with $M_{c}^{(1_{m})}(p) = p-1$ is equal to $0 \%$ as $c\to \infty$. More precisely, we have 
\begin{center}
    $\lim\limits_{c\to\infty} \Large{\frac{\# \{\varphi_{(p-1)^{\ell},c}(x) \in \mathbb{Z}[x]\ : \ 5\leq p\leq c \ \textnormal{and} \ M_{c}^{(1_{m})}(p) \ = \ p-1\}}{\Large{\# \{\varphi_{(p-1)^{\ell},c}(x) \in \mathbb{Z}[x]\ : \ 5\leq p\leq c \}}}} = \ 0.$
\end{center}
\end{cor}

\begin{proof}
Since $M_{c}^{(1_{m})}(p) = p-1$ is as proved in Theorem \ref{3.3} and so in Cor. \ref{cor3.4} determined when the coefficient $c$ is such that $c\pm 1$ is divisible by $p$; and so we may count $\# \{\varphi_{(p-1)^{\ell},c}(x) \in \mathbb{Z}[x] : 5\leq p\leq c \ \text{and} \ M_{c}^{(1_{m})}(p) \ = \ p-1\}$ by counting the number $\# \{\varphi_{(p-1)^{\ell},c}(x)\in \mathbb{Z}[x] : 5\leq p\leq c \ \text{and} \ p\mid (c\pm 1) \ \text{for \ any \ fixed} \ c \}$. But now since $c-1<c$, then if the number $\# \{p : 5\leq p\leq c \ \text{and} \ p\mid (c-1) \}< \# \{p : 5\leq p\leq c \ \text{and} \ p\mid c \}$, we then obtain that
\begin{center}
$\Large{\frac{\# \{\varphi_{(p-1)^{\ell},c}(x) \in \mathbb{Z}[x] \ : \ 5\leq p\leq c \ \text{and} \ p\mid (c-1) \ \text{for any fixed} \ c\}}{\Large{\# \{\varphi_{(p-1)^{\ell},c}(x) \in \mathbb{Z}[x] \ : \ 5\leq p\leq c \}}}} < \Large{\frac{\# \{\varphi_{(p-1)^{\ell},c}(x)\in \mathbb{Z}[x] \ : \ 5\leq p\leq c \ \text{and} \ p\mid c \ \text{for any fixed} \ c \}}{\Large{\# \{\varphi_{(p-1)^{\ell},c}(x) \in \mathbb{Z}[x] \ : \ 5\leq p\leq c \}}}}.$ 
\end{center}But then applying a similar argument as in [\cite{BK1}, Proof of Cor. 6.2], we then obtain that the limit is equal to $0$ in this case. Otherwise, if the number $\# \{p : 5\leq p\leq c \ \text{and} \ p\mid c \}< \# \{p : 5\leq p\leq c \ \text{and} \ p\mid (c-1) \}$, we obtain
\begin{center}
$\Large{\frac{\# \{\varphi_{(p-1)^{\ell},c}(x) \in \mathbb{Z}[x] \ : \ 5\leq p\leq c \ \text{and} \ p\mid c \ \text{for any fixed} \ c\}}{\Large{\# \{\varphi_{(p-1)^{\ell},c}(x) \in \mathbb{Z}[x] \ : \ 5\leq p\leq c \}}}} < \Large{\frac{\# \{\varphi_{(p-1)^{\ell},c}(x)\in \mathbb{Z}[x] \ : \ 5\leq p\leq c \ \text{and} \ p\mid (c-1) \ \text{for any fixed} \ c \}}{\Large{\# \{\varphi_{(p-1)^{\ell},c}(x) \in \mathbb{Z}[x] \ : \ 5\leq p\leq c \}}}}.$
\end{center}So now, taking limit as $c\to \infty$ on both sides of the above inequality and applying Corollary \ref{7.1} and then applying a similar argument as in the Proof of Corollary \ref{6.1} to obtain an upper bound zero, we then obtain
\begin{center}
$\lim\limits_{c\to\infty}\Large{\frac{\# \{\varphi_{(p-1)^{\ell},c}(x) \in \mathbb{Z}[x] \ : \ 5\leq p\leq c \ \text{and} \ p\mid (c-1)\}}{\Large{\# \{\varphi_{(p-1)^{\ell},c}(x) \in \mathbb{Z}[x] \ : \ 5\leq p\leq c \}}}} = 0 = \lim\limits_{c\to\infty}\Large{\frac{\# \{\varphi_{(p-1)^{\ell},c}(x) \in \mathbb{Z}[x] \ : \ 5\leq p\leq c \ \text{and} \ p\mid (c+1)\}}{\Large{\# \{\varphi_{(p-1)^{\ell},c}(x) \in \mathbb{Z}[x] \ : \ 5\leq p\leq c \}}}}$ 
\end{center}where the second limit follows from observing that $c<c+1$ and then applying an argument that is very similar to an argument that has been given in the case $c-1<c$. This then completes the whole proof, as needed.
\end{proof}

As before, we also have the following corollary showing that for any fixed eventual period $m \in \mathbb{Z}_{\geq 1}$, there are very few monics $\varphi_{(p-1)^{\ell},c}(x)\in \mathbb{Z}[x]$ having $p-4$ distinct strictly $1_{m}$-preperiodic integral points modulo $p$:

\begin{cor}
Let $K\slash \mathbb{Q}$ be any number field of degree $n\geq 1$ with the ring of integers $\mathcal{O}_{K}$, and in which any prime $p\geq 5$ is inert. Let $\ell, m\in \mathbb{Z}_{ \geq 1}$ be any fixed integers. Then the density of monic integer polynomials $\varphi_{(p-1)^{\ell},c}(x) = x^{(p-1)^{\ell}} + c\in \mathcal{O}_{K}[x]$ with $M_{c}^{+(1_{m})}(p) = p-4$ is equal to $0 \%$ as $c\to \infty$. More precisely, we have 
\begin{center}
    $\lim\limits_{c\to\infty} \Large{\frac{\# \{\varphi_{(p-1)^{\ell},c}(x) \in \mathbb{Z}[x]\ : \ 5\leq p\leq c \ \textnormal{and} \ M_{c}^{+(1_{m})}(p) \ = \ p-4\}}{\Large{\# \{\varphi_{(p-1)^{\ell},c}(x) \in \mathbb{Z}[x]\ : \ 5\leq p\leq c \}}}} = \ 0.$
\end{center}
\end{cor}
\begin{proof}
Recalling that the condition $M_{c}^{+(1_{m})}(p) = p-4$ is as noted in the Proof of Corollary \ref{cor5B} determined when the coefficient $c$ is divisible by inert $p$, we may count $\# \{\varphi_{(p-1)^{\ell},c}(x) \in \mathbb{Z}[x] : 5\leq p\leq c \ \text{and} \ M_{c}^{+(1_{m})}(p) \ = \ p-4\}$ by counting $\# \{\varphi_{(p-1)^{\ell},c}(x)\in \mathbb{Z}[x] : 5\leq p\leq c \ \text{and} \ p\mid c \ \text{for \ any \ fixed} \ c \}$. Now applying a similar argument as in Proof of Cor. \ref{7.1}, we then obtain that the limit is equal to $0$, which completes the whole proof, as desired. 
\end{proof}

\section{On Density of Polynomials $\varphi_{p^{\ell},c}(x)$ with $N_{c}^{(1_{m})}(p) = 0$ and $\varphi_{(p-1)^{\ell},c}(x)$ with $M_{c}^{(1_{m})}(p) = 0$}\label{sec7}

\noindent Recall in Corollary \ref{6.1} that a density of $0\%$ of integer polynomials $\varphi_{p^{\ell},c}(x)\in\mathcal{O}_{K}[x]$ have number $N_{c}^{(1_{m})}(p) = p$; and so the density of integer polynomials $\varphi_{p^{\ell},c}^{1+m}(x)-\varphi_{p^{\ell},c}(x)\in \mathcal{O}_{K}[x]$ that are reducible modulo $p$ is $0\%$. Now, we also wish to determine: \say{\textit{For any prime $p\geq 3$ and for any fixed $\ell,m \in \mathbb{Z}_{\geq 1}$, what is the density of integer polynomials $\varphi_{p^{\ell},c}(x)\in \mathcal{O}_{K}[x]$ with no $1_{m}$-preperiodic integral points modulo $p$?}} The following corollary shows that for any fixed $\ell, m \in \mathbb{Z}_{\geq 1}$, the probability of choosing randomly a monic integer polynomial $\varphi_{p^{\ell},c}(x)\in \mathbb{Z}[x]\subset \mathcal{O}_{K}[x]$ such that $\mathbb{Q}[x]\slash (\varphi_{p^{\ell},c}^{1+m}(x)-\varphi_{p^{\ell},c}(x))$ is a number field of degree $p^{(m+1)\ell}$ is equal to 1: 

\begin{cor}\label{8.1}
Let $K\slash \mathbb{Q}$ be any number field of any degree $n\geq 1$ with the ring of integers $\mathcal{O}_{K}$, and in which any prime $p\geq 3$ is inert. Let $\ell, m \in \mathbb{Z}_{\geq 1}$ be any fixed integers. Then the density of monic integer polynomials $\varphi_{p^{\ell},c}(x)=x^{p^{\ell}} + c\in \mathcal{O}_{K}[x]$ with $N_{c}^{(1_{m})}(p) = 0$ exists and is equal to $100 \%$ as $c\to \infty$. Specifically, we have 
\begin{center}
    $\lim\limits_{c\to\infty} \Large{\frac{\# \{\varphi_{p^{\ell},c}(x)\in \mathbb{Z}[x] \ : \ 3\leq p\leq c \ \textnormal{and} \ N_{c}^{(1_{m})}(p) \ = \ 0 \}}{\Large{\# \{\varphi_{p^{\ell},c}(x) \in \mathbb{Z}[x] \ : \ 3\leq p\leq c \}}}} = \ 1.$
\end{center}
\end{cor}
\begin{proof}
Since the number $N_{c}^{(1_{m})}(p) = p$ or $0$ for any given prime $p\geq 3$ and fixed $\ell, m\in \mathbb{Z}_{\geq 1}$, and since we also proved density in Cor. \ref{6.1}, we then obtain the desired density (i.e., we obtain that the limit is equal to 1). 
\end{proof}

\noindent Note that the foregoing corollary also shows that for any fixed $\ell, m \in \mathbb{Z}_{\geq 1}$, there are infinitely many polynomials $\varphi_{p^{\ell},c}(x)\in \mathbb{Z}[x]\subset \mathbb{Q}[x]$ such that for $f(x) = \varphi_{p^{\ell},c}^{1+m}(x)-\varphi_{p^{\ell},c}(x)$, the quotient $\mathbb{Q}_{f} = \mathbb{Q}[x]\slash (f(x))$ induced by $f$ is an algebraic number field of odd degree $\kappa=p^{(m+1)\ell}$. Comparing the densities in Corollaries \ref{6.1} and \ref{8.1}, one may then observe that in the whole family of monic integer polynomials $\varphi_{p^{\ell},c}(x) = x^{p^{\ell}} +c$, almost all such monic integer polynomials $\varphi_{p^{\ell},c}(x)$ have no $1_{m}$-preperiodic integral points modulo $p$; and from which it then also follows that almost all monic integer polynomials $f(x)$ are irreducible over $\mathbb{Q}$. Consequently, this may then imply that the average value of the number $N_{c}^{(1_{m})}(p)$ in the whole family of monic polynomials $\varphi_{p^{\ell},c}(x)$ is zero.

Similarly, recall in Corollary \ref{7.1} and \ref{7.2} that a density of $0\%$ of monic integer polynomials $\varphi_{(p-1)^{\ell},c}(x)$ have $M_{c}^{(1_{m})}(p) = p$ or $p-1$, resp.; and so the density of $\varphi_{(p-1)^{\ell},c}^{1+m}(x)-\varphi_{(p-1)^{\ell},c}(x)\in \mathbb{Z}[x]$ that are reducible modulo $p$ is $0\%$. So now, we also wish to determine: \say{\textit{For any prime $p\geq 5$ and for any fixed $\ell, m \in \mathbb{Z}_{\geq 1}$, what is the density of polynomials $\varphi_{(p-1)^{\ell},c}(x)\in \mathbb{Z}[x]$ with no $1_{m}$-preperiodic integral points (mod $p$)?}} The corollary below shows that for any fixed $\ell, m \in \mathbb{Z}_{\geq 1}$, the probability of choosing randomly a polynomial $\varphi_{(p-1)^{\ell},c}(x)\in \mathbb{Z}[x]$ so that $\mathbb{Q}[x]\slash (\varphi_{(p-1)^{\ell},c}^{1+m}(x)-\varphi_{(p-1)^{\ell},c}(x))$ is an algebraic number field of degree $(p-1)^{(m+1)\ell}$ is also equal to 1:
\begin{cor} \label{8.2}
Let $K\slash \mathbb{Q}$ be any number field of degree $n\geq 1$ with the ring of integers $\mathcal{O}_{K}$, and in which any prime $p\geq 5$ is inert. Let $\ell, m \in \mathbb{Z}_{\geq 1}$ be any fixed integers. Then the density of monic integer polynomials $\varphi_{(p-1)^{\ell}, c}(x) = x^{(p-1)^{\ell}}+c\in \mathcal{O}_{K}[x]$ with $M_{c}^{(1_{m})}(p) = 0$ exists and is equal to $100 \%$ as $c\to \infty$. That is, we have 
\begin{center}
    $\lim\limits_{c\to\infty} \Large{\frac{\# \{\varphi_{(p-1)^{\ell}, c}(x)\in \mathbb{Z}[x] \ : \ 5\leq p\leq c \ \textnormal{and} \ M_{c}^{(1_{m})}(p) \ = \ 0 \}}{\Large{\# \{\varphi_{(p-1)^{\ell},c}(x) \in \mathbb{Z}[x] \ : \ 5\leq p\leq c \}}}} = \ 1.$
\end{center}
\end{cor}
\begin{proof}
Recall $M_{c}^{(1_{m})}(p) = p, p-1$ or $0$ for any given $p\geq 5$ and  fixed $\ell, m \in \mathbb{Z}_{\geq 1}$, and since we also proved the densities in Cor.\ref{7.1} and \ref{7.2}, we then obtain the desired density (i.e., we obtain that the limit is equal to 1).
\end{proof}
\noindent As before, Corollary \ref{8.2} also shows that for any fixed integers $\ell, m \in \mathbb{Z}_{\geq 1}$, there are infinitely many monic polynomials $\varphi_{(p-1)^{\ell},c}(x)\in \mathbb{Z}[x]\subset \mathbb{Q}[x]$ such that for $g(x) = \varphi_{(p-1)^{\ell},c}^{1+m}(x)-\varphi_{(p-1)^{\ell},c}(x)$, the quotient ring $\mathbb{Q}_{g} = \mathbb{Q}[x]\slash (g(x))$ induced by $g$ is an algebraic number field of even degree $\upsilon=(p-1)^{(m+1)\ell}$. As before, if we also compare the densities in Cor. \ref{7.1}, \ref{7.2}, and \ref{8.2}, we may again observe that in the whole family of monic polynomials $\varphi_{(p-1)^{\ell},c}(x) = x^{(p-1)^{\ell}} +c\in \mathbb{Z}[x]$, almost all such polynomials $\varphi_{(p-1)^{\ell},c}(x)$ have no $1_{m}$-preperiodic integral points modulo $p$; from which it then also follows that almost all $g(x)\in \mathbb{Z}[x]$ are irreducible over $\mathbb{Q}$. This may also imply that the average value of $M_{c}^{(1_{m})}(p)$ in the whole family of $\varphi_{(p-1)^{\ell},c}(x)\in \mathbb{Z}[x]$ is also zero.

Recall that every number field $K$ is always naturally equipped with a ring $\mathcal{O}_{K}$ of integers in $K$. So now, it then follows that $\mathbb{Q}_{f}=\mathbb{Q}[x]\slash (f(x))$ has a ring of integers $\mathcal{O}_{\mathbb{Q}_{f}}$; and moreover applying (as in \cite{BK22}) here [\cite{sch1}, Theorem 1.2], we then also obtain the following corollary showing that the probability of choosing randomly a monic polynomial $f(x)=\varphi_{p^{\ell},c}^{1+m}(x)-\varphi_{p^{\ell},c}(x)\in \mathbb{Z}[x]$ arising from a polynomial discrete dynamical system in Section \ref{sec2} (and ascertained by Corollary \ref{8.1}) such that $\mathbb{Z}[x]\slash (f(x))$ is the ring of integers of $\mathbb{Q}_{f}$ is $\approx 60.7927\%$:

\begin{cor} \label{7.3}
Assume Corollary \ref{8.1}. When monic polynomials $f(x)\in \mathbb{Z}[x]$ are ordered by height $H(f)$ as defined in \textnormal{\cite{sch1}}, the density of polynomials $f(x)$ such that $\mathbb{Z}_{f}=\mathbb{Z}[x]\slash (f(x))$ is the ring of integers of $\mathbb{Q}_{f}$ is $\zeta(2)^{-1}$. 
\end{cor}

\begin{proof}
From Corollary \ref{8.1}, there are infinitely many irreducible monic polynomials $f(x)=\varphi_{p^{\ell},c}^{1+m}(x)-\varphi_{p^{\ell},c}(x)\in \mathbb{Z}[x]$ such that $\mathbb{Q}_{f} = \mathbb{Q}[x]\slash (f(x))$ is a number field of deg$(f) = p^{(m+1)\ell}$; and moreover associated to each $\mathbb{Q}_{f}$ is the ring of integers $\mathcal{O}_{\mathbb{Q}_{f}}$. This then means that the family of monic polynomials $f(x)\in \mathbb{Z}[x]$ such that $\mathbb{Q}_{f}$ is a number field of degree $\kappa=p^{(m+1)\ell}$, is not empty. So now, applying here [\cite{sch1}, Theorem 1.2] on the underlying family of such monic polynomials $f\in \mathbb{Z}[x]$ ordered by height $H(f)$ as defined in \cite{sch1} such that $\mathcal{O}_{\mathbb{Q}_{f}} = \mathbb{Z}[x]\slash (f(x))$, it then follows that the density of such polynomials $f(x)\in \mathbb{Z}[x]$ is equal to $\zeta(2)^{-1} \approx 60.7927\%$, as needed.
\end{proof}

Similarly, we note that every number field $\mathbb{Q}_{g}=\mathbb{Q}[x]\slash (g(x))$ is also naturally equipped with a ring of integers $\mathcal{O}_{\mathbb{Q}_{g}}$. So now, by again taking great advantage of [\cite{sch1}, Theorem 1.2], we then also obtain immediately the following corollary showing that the probability of choosing randomly a monic integer polynomial $g(x) = \varphi_{(p-1)^{\ell},c}^{1+m}(x)-\varphi_{(p-1)^{\ell},c}(x)\in \mathbb{Z}[x]$ arising from a polynomial discrete dynamical system in Section \ref{sec3} (and ascertained by Corollary \ref{8.2}), such that the quotient $\mathbb{Z}[x]\slash (g(x))$ is the ring of integers of $\mathbb{Q}_{g}$ is also $\approx 60.7927\%$:

\begin{cor}
Assume Corollary \ref{8.2}. When monic polynomials $g(x)\in \mathbb{Z}[x]$ are ordered by height $H(g)$ as defined in \textnormal{\cite{sch1}}, the density of polynomials $g(x)$ such that $\mathbb{Z}_{g}=\mathbb{Z}[x]\slash (g(x))$ is the ring of integers of $\mathbb{Q}_{g}$ is $\zeta(2)^{-1}$. 
\end{cor}

\begin{proof}
By applying a similar argument as in the Proof of Corollary \ref{7.3}, we then obtain the density, as needed.
\end{proof}

\section{On Artin-Mazur Zeta Function of $\overline{\varphi_{p^{\ell},c}}$ with $N_{c}^{(1_{m})}(p) = 0$ \& $\overline{\varphi_{(p-1)^{\ell},c}}$ with $M_{c}^{(1_{m})}(p) = 0$}

As in [\cite{BK22}, Section 9] recall from [\cite{AM}, Page 84] that for any given topological space $X$ and for any given map $f:X\to X$, we can attach to $f$ a zeta function $\zeta_{f}$ defined in the following way, that for any point $s\in \mathbb{C}$, we set  
\begin{equation}\label{eqAM}
    \zeta_{f}(s)=\textnormal{exp}\biggl(\sum_{m=1}^{\infty} \frac{N_{m}(f)\cdot s^m}{m}\biggr)
\end{equation} where $N_{m}(f)$ is the number of isolated periodic points of $f$, of period $m$. So now, inspired by the definition and work of Artin-Mazur \cite{AM} on their zeta function $\zeta_{f}(s)$, we in this section wish to define (in a similar way) and then determine the Artin-Mazur zeta functions arising from a polynomial discrete dynamical system in Section \ref{sec2} and \ref{sec3} (and ascertained by Corollary \ref{8.1} and \ref{8.2}). With that in mind, we note that since $\mathcal{O}_{K}\slash p\mathcal{O}_{K}$ is a finite set of $p^{[K:\mathbb{Q}]}=p^n$ elements and so $\mathcal{O}_{K}\slash p\mathcal{O}_{K}$ can be equipped with a topology, then denoting the reduced polynomial map $\varphi_{p^{\ell},c}$ modulo $p\mathcal{O}_{K}$ by $\overline{\varphi_{p^{\ell},c}}$ and then also replacing the number $N_{m}(f)$ with the number $N_{c}^{(1_{m})}(p)$ in Equation (\ref{eqAM}), we then let $\zeta_{\overline{\varphi_{p^{\ell},c}}}$ be the Artin-Mazur zeta function corresponding to a reduced polynomial map $\overline{\varphi_{p^{\ell},c}}: \mathcal{O}_{K}\slash p\mathcal{O}_{K} \to \mathcal{O}_{K}\slash p\mathcal{O}_{K}$. Inspired by Artin-Mazur's work \cite{AM} on their zeta function $\zeta_{f}$, we in our case then have the following corollary showing that the Artin-Mazur zeta functions $\zeta_{\overline{\varphi_{p^{\ell},c}}}$ associated with infinitely many polynomial maps $\overline{\varphi_{p^{\ell},c}}$ are constant functions on the whole complex plane $\mathbb{C}$ and hence algebraic functions of $s$:  

\begin{cor}\label{cAM}
Assume Corollary \ref{8.1} or second part of Theorem \ref{2.3}, and denote the polynomial map $\varphi_{p^{\ell},c}$ modulo prime $p\mathcal{O}_{K}$ by $\overline{\varphi_{p^{\ell},c}}$. Then the Artin-Mazur zeta function $\zeta_{\overline{\varphi_{p^{\ell},c}}}(s)=1$ for every complex number $s\in \mathbb{C}$.
\end{cor}

\begin{proof}
To see this, we note that from Corollary \ref{8.1} the existence of an infinite family of polynomials $\varphi_{p^{\ell},c}(x)\in (\mathbb{Z}\slash p\mathbb{Z})[x]$ with $N_{c}^{(1_{m})}(p)=0$, for every fixed eventual period $m\in \mathbb{Z}_{\geq 1}$; or note that from the second part of Theorem \ref{2.3} that the monic polynomial $\varphi_{p^{\ell},c}(x)\in (\mathcal{O}_{K}\slash p\mathcal{O}_{K})[x]$ has $N_{c}^{(1_{m})}(p)=0$ for every $c\not \equiv 0\ (\text{mod} \ p\mathcal{O}_{K})$ and fixed eventual period $m\in \mathbb{Z}_{\geq 1}$. But now since $N_{c}^{(1_{m})}(p)=0$ for every fixed $m\in \mathbb{Z}_{\geq 1}$, we then note 
\begin{equation}
    \zeta_{\overline{\varphi_{p^{\ell},c}}}(s)=\textnormal{exp}\biggl(\sum_{m=1}^{\infty} \frac{N_{c}^{(1_{m})}(p)\cdot s^m}{m}\biggr)=\textnormal{exp}(0)=1,
\end{equation} and thus $\zeta_{\overline{\varphi_{p^{\ell},c}}}(s)=1$ for every complex number $s\in \mathbb{C}$. This then completes the whole proof, as desired.
\end{proof}

Similarly, denoting the reduced polynomial map $\varphi_{(p-1)^{\ell},c}$ modulo $p\mathcal{O}_{K}$ by $\overline{\varphi_{(p-1)^{\ell},c}}$ and then also replacing $N_{m}(f)$ with $M_{c}^{(1_{m})}(p)$ in Equation (\ref{eqAM}), we then also let $\zeta_{\overline{\varphi_{(p-1)^{\ell},c}}}$ be the Artin-Mazur zeta function induced by the map $\overline{\varphi_{(p-1)^{\ell},c}}: \mathcal{O}_{K}\slash p\mathcal{O}_{K} \to \mathcal{O}_{K}\slash p\mathcal{O}_{K}$. So now, inspired again by \cite{AM}, we then also have the following corollary showing that the Artin-Mazur zeta functions $\zeta_{\overline{\varphi_{(p-1)^{\ell},c}}}$ associated with infinitely many polynomial maps $\overline{\varphi_{(p-1)^{\ell},c}}$ are also constant functions on the whole complex plane $\mathbb{C}$ and thus also algebraic functions of $s$:

\begin{cor}
Assume Corollary \ref{8.2} or second part of Theorem \ref{3.3}, and denote the polynomial map $\varphi_{(p-1)^{\ell},c}$ modulo prime $p\mathcal{O}_{K}$ by $\overline{\varphi_{(p-1)^{\ell},c}}$. Then the Artin-Mazur zeta function $\zeta_{\overline{\varphi_{(p-1)^{\ell},c}}}(s)=1$ for every complex $s\in \mathbb{C}$.
\end{cor}

\begin{proof}
Applying a similar argument as in Proof of Corollary \ref{cAM}, we then obtain the conclusion, as desired.
\end{proof}

\section{On Number of Monic Polynomials $f\in \mathcal{O}_{K}[x]$ \& $g\in \mathcal{O}_{K}[x]$ with Primitive Galois groups}\label{sec8a}

Recall from Corollary \ref{8.1} that there is an infinite family of irreducible monic polynomials $f(x) = \varphi_{p^{\ell},c}^{1+m}(x)-\varphi_{p^{\ell},c}(x)\in \mathbb{Z}[x]\subset \mathcal{O}_{K}[x]$ such that the quotient $\mathbb{Q}_{f}=\mathbb{Q}[x]\slash (f(x))$ is a number field of degree $\kappa=p^{(m+1)\ell}$. Moreover, to each such irreducible monic integer polynomial $f\in \mathbb{Q}[x]$, let $G_{f}$ be the Galois group of $f$ over $\mathbb{Q}$. 

So now, inspired (as in \cite{BK22}) by Bhargava's work \cite{gav1} on Van der Waerden's Conjecture, we then also wish to determine the number of irreducible monic polynomials  $f\in \mathbb{Z}[x]$ arising from a polynomial discrete dynamical system in Section \ref{sec2} (and ascertained by Corollary \ref{8.1}), of bounded height and such that $G_{f}$ is a primitive Galois group not equal to the full symmetric group $S_{\kappa}$. To that end, we (assuming Corollary \ref{8.1}) wish to first adhere to the setup and definition of coefficient height $h(f)$ in \cite{gav1}). That is, for any fixed $\kappa=p^{(m+1)\ell}$, we let $E_{\kappa}(H)$ be the number of monic polynomials $f(x) = \varphi_{p^{\ell},c}^{1+m}(x)-\varphi_{p^{\ell},c}(x)\in \mathbb{Z}[x]$ with $h(f)\leq H$ and such that $G_{f}\neq S_{\kappa}$. But now by taking great advantage of [\cite{gav1}, Thm. 1] of Bhargava, we then obtain the following:

\begin{cor}\label{c9.1}
Assume Corollary \ref{8.1}, and let $E_{\kappa}(H)$ be defined as before. Then we have $E_{\kappa}(H)=O(H^{\kappa-1})$.
\end{cor}

\begin{proof}
From Corollary \ref{8.1}, there are infinitely many irreducible monic polynomials $f(x)=\varphi_{p^{\ell},c}^{1+m}(x)-\varphi_{p^{\ell},c}(x)\in \mathbb{Z}[x]$ of degree $\kappa=p^{(m+1)\ell}$. Now for every $f\in \mathbb{Z}[x]\subset \mathbb{Q}[x]$ of fixed degree $\kappa=p^{(m+1)\ell}$, let $G_{f}$ be the Galois group of $f$ over $\mathbb{Q}$. Now applying [\cite{gav1}, Theorem 1] on the set of monic polynomials $f(x)\in \mathbb{Z}[x]$ of degree $\kappa$ with $h(f)\leq H$ and such that $G_{f}$ is primitive and $G_{f}\neq S_{\kappa}$, we then immediately obtain the count, as needed.
\end{proof}

Similarly, we may also recall from Corollary \ref{8.2} that there is an infinite family of irreducible monic polynomials $g(x) = \varphi_{(p-1)^{\ell},c}^{1+m}(x)-\varphi_{(p-1)^{\ell},c}(x)\in \mathbb{Z}[x]\subset \mathcal{O}_{K}[x]$ such that $\mathbb{Q}_{g}=\mathbb{Q}[x]\slash (g(x))$ is a number field of degree $\upsilon=(p-1)^{(m+1)\ell}$. And moreover, to each such irreducible integer polynomial $g\in \mathbb{Q}[x]$, let $G_{g}$ be the Galois group of $g$ over $\mathbb{Q}$. So now, inspired again work \cite{gav1} of Bhargava on Van der Waerden's Conjecture, we then also wish to determine the number of irreducible monic polynomials  $g\in \mathbb{Z}[x]$ arising from a polynomial discrete dynamical system in Section \ref{sec3} (and ascertained by Corollary \ref{8.2}), of bounded height and such that $G_{g}$ is a primitive Galois group not equal to the full symmetric group $S_{\upsilon}$. To that end, we (assuming Corollary \ref{8.2}) again first import the setup and definition of coefficient height $h(g)$ in \cite{gav1}). That is, for any fixed $\upsilon=(p-1)^{(m+1)\ell}$, we let $E_{\upsilon}(H)$ be the number of monic degree-$\upsilon$ polynomials $g(x) = \varphi_{(p-1)^{\ell},c}^{1+m}(x)-\varphi_{(p-1)^{\ell},c}(x)\in \mathbb{Z}[x]$ such that $h(g)\leq H$ and $G_{g}\neq S_{\upsilon}$. Now again taking great advantage of [\cite{gav1}, Thm. 1], we then also obtain the following:

\begin{cor}
Assume Corollary \ref{8.2}, and let $E_{\upsilon}(H)$ be defined as before. Then we have $E_{\upsilon}(H)=O(H^{\upsilon-1})$.
\end{cor}

\begin{proof}
By applying a similar argument as in the Proof of Cor. \ref{c9.1}, we then obtain the conclusion, as needed.
\end{proof}

\section{On Number of Algebraic Number fields $K_{f}$ \& $L_{g}$ with Bounded Absolute Discriminant}\label{sec8}

\subsection*{On Fields $\mathbb{Q}_{f}$ \& $\mathbb{Q}_{g}$ with Bounded Absolute Discriminant \& Prescribed Galois group}
As in Section \ref{sec8a}, recall from Corollary \ref{8.1} that there is an infinite family of irreducible monic polynomials $f(x) = \varphi_{p^{\ell},c}^{1+m}(x)-\varphi_{p^{\ell},c}(x)\in \mathbb{Z}[x]$ such that $\mathbb{Q}_{f} = \mathbb{Q}[x]\slash (f(x))$ is a number field of degree $\kappa=p^{(m+1)\ell}$. Similarly, recall from Corollary \ref{8.2} that one can always find an infinite family of irreducible monic polynomials $g(x) = \varphi_{(p-1)^{\ell},c}^{1+m}(x)-\varphi_{(p-1)^{\ell},c}(x)\in \mathbb{Z}[x]$ such that $\mathbb{Q}_{g} = \mathbb{Q}[x]\slash (g(x))$ is a number field of degree $\upsilon=(p-1)^{(m+1)\ell}$. So now, inspired (as in \cite{BK22}) by number field-counting advances in arithmetic statistics, we in this section also wish to count the number of fields $\mathbb{Q}_{f}$ and $\mathbb{Q}_{g}$ induced by irreducible monic integer polynomials $f$ and $g$ arising from polynomial discrete dynamical systems in Section \ref{sec2} and \ref{sec3} (and ascertained by Corollary \ref{8.1} and \ref{8.2}). To do so, we (as in \cite{BK22}) define and then determine the asymptotic behavior of the following counting functions  
\begin{equation}\label{N_{k}}
N_{\kappa}(X) := \# \Bigl\{\mathbb{Q}_{f}\slash \mathbb{Q} : [\mathbb{Q}_{f} : \mathbb{Q}] = \kappa \textnormal{ and } |\text{Disc}(\mathbb{Q}_{f})|\leq X \Bigr\}
\end{equation} 
\begin{equation}\label{M_{l}}
M_{\upsilon}(X) := \# \Bigl\{\mathbb{Q}_{g}\slash \mathbb{Q} : [\mathbb{Q}_{g} : \mathbb{Q}] = \upsilon \textnormal{ and} \ |\text{Disc}(\mathbb{Q}_{g})|\leq X \Bigr\}
\end{equation} as a positive real number $X\to \infty$. But now motivated (as in \cite{BK22}) by number field-counting work of Lemke Oliver-Thorne \cite{lem} and then applying the first part of [\cite{lem}, Theorem 1.2] on $N_{\kappa}(X)$, we then obtain the following:

\begin{cor} \label{9.1} Assume Corollary \ref{8.1}, and let $N_{\kappa}(X)$ be the number defined as in \textnormal{(\ref{N_{k}})}. Then we have 
\begin{equation}\label{N_{k}(x)} 
N_{\kappa}(X)\ll_{\kappa}X^{2d - \frac{d(d-1)(d+4)}{6\kappa}}\ll X^{\frac{8\sqrt{\kappa}}{3}}, \text{where d is the least integer for which } \binom{d+2}{2}\geq 2\kappa + 1.
\end{equation}
\end{cor}

\begin{proof}
To see the inequality \textnormal{(\ref{N_{k}(x)})}, we first recall from Corollary \ref{8.1} that there are infinitely many irreducible monic integer polynomials $f(x) = \varphi_{p^{\ell},c}^{1+m}(x)-\varphi_{p^{\ell},c}(x)\in \mathbb{Q}[x]$ such that $\mathbb{Q}_{f}=\mathbb{Q}[x]\slash (f(x))$ is a number field of degree $\kappa=p^{(m+1)\ell}$. This then also means that the set of degree-$\kappa$ algebraic number field $\mathbb{Q}_{f}$ is not empty. But now applying here [\cite{lem}, Theorem 1.2 (1)] on the number $N_{\kappa}(X)$, we then obtain inequality \textnormal{(\ref{N_{k}(x)})}, as needed.
\end{proof}

Motivated again by the same work \cite{lem} of Lemke Oliver-Thorne, we again take great advantage of the first part of [\cite{lem}, Theorem 1.2] by applying it on $M_{\upsilon}(X)$. In doing so, we then immediate obtain the following:

\begin{cor}Assume Corollary \ref{8.2}, and let $M_{\upsilon}(X)$ be the number defined as in \textnormal{(\ref{M_{l}})}. Then we have 
\begin{equation}\label{M_{l}(x)}
M_{\upsilon}(X)\ll_{\upsilon}X^{2d - \frac{d(d-1)(d+4)}{6\upsilon}}\ll X^{\frac{8\sqrt{\upsilon}}{3}}, \text{where d is the least integer for which } \binom{d+2}{2}\geq 2\upsilon + 1.
\end{equation}
\end{cor}

\begin{proof}
Applying a similar argument as in Proof of Corollary \ref{9.1}, we then obtain inequality \textnormal{(\ref{M_{l}(x)})} as needed.
\end{proof}

We recall that a number field $K$ is  \say{\textit{monogenic}} if there exists an algebraic number $\alpha \in K$ such that the ring of integers $\mathcal{O}_{K}$ is the subring $\mathbb{Z}[\alpha]$ generated by $\alpha$ over $\mathbb{Z}$, i.e., $\mathcal{O}_{K}= \mathbb{Z}[\alpha]$. So now, inspired (as in \cite{BK22}), we also wish to count the number of fields $\mathbb{Q}_{f}=\mathbb{Q}[x]\slash (f(x))$ induced by monic polynomials $f(x) = \varphi_{p^{\ell},c}^{1+m}(x)-\varphi_{p^{\ell},c}(x)\in \mathbb{Z}[x]$ arising from a polynomial discrete dynamical system in Section \ref{sec2} (and ascertained by \ref{8.1}), that are monogenic with $|\Delta(\mathbb{Q}_{f})| < X$ and have associated Galois group Gal$(\mathbb{Q}_{f}\slash \mathbb{Q})$ equal to symmetric group $S_{p^{(m+1)\ell}}$. To do so, we take great advantage of [\cite{sch1}, Cor. 1.3] of Bhargava-Shankar-Wang and then obtain:

\begin{cor}\label{8.3}
Assume Corollary \ref{8.1}. The number of isomorphism classes of algebraic number fields $\mathbb{Q}_{f}$ of degree $\kappa=p^{(m+1)\ell}$ and with $|\Delta(\mathbb{Q}_{f})| < X$ that are monogenic and have associated Galois group $S_{\kappa}$ is $\gg X^{\frac{1}{2} + \frac{1}{\kappa}}$.
\end{cor}

\begin{proof}
To see this, we recall from Cor. \ref{8.1} the existence of infinitely many irreducible monic polynomials $f(x) = \varphi_{p^{\ell},c}^{1+m}(x)-\varphi_{p^{\ell},c}(x)\in \mathbb{Z}[x]$  such that the quotient $\mathbb{Q}_{f}=\mathbb{Q}[x]\slash (f(x))$ induced by $f$ is a number field of degree $\kappa=p^{(m+1)\ell}$. This then means that the set of degree-$\kappa$ number fields $\mathbb{Q}_{f}$ is not empty. But now applying [\cite{sch1}, Corollary 1.3] on the underlying fields $\mathbb{Q}_{f}$ with $|\Delta(\mathbb{Q}_{f})| < X$ that are monogenic and with associated Galois group $S_{k}$, it then follows that the number of isomorphism classes of such fields $\mathbb{Q}_{f}$ is $\gg X^{\frac{1}{2} + \frac{1}{\kappa}}$, as required.
\end{proof}

Similarly, we again take great advantage of [\cite{sch1}, Corollary 1.3] to then also count in the following corollary the number of fields $\mathbb{Q}_{g}=\mathbb{Q}[x]\slash (g(x))$ induced by monic polynomials $g(x) = \varphi_{(p-1)^{\ell},c}^{1+m}(x)-\varphi_{(p-1)^{\ell},c}(x)\in \mathbb{Z}[x]$ arising from a polynomial discrete dynamical system in Section \ref{sec3} (and ascertained by Corollary \ref{8.2}), that are monogenic with $|\Delta(\mathbb{Q}_{g})| < X$ and associated Galois group Gal$(\mathbb{Q}_{g}\slash \mathbb{Q})$ equal to the symmetric group $S_{(p-1)^{(m+1)\ell}}$:
\begin{cor}
Assume Corollary \ref{8.2}. The number of isomorphism classes of algebraic number fields $\mathbb{Q}_{g}$ of degree $\upsilon=(p-1)^{(m+1)\ell}$ and $|\Delta(\mathbb{Q}_{g})| < X$ that are monogenic and have associated Galois group $S_{\upsilon}$ is $\gg X^{\frac{1}{2} + \frac{1}{\upsilon}}$.
\end{cor}

\begin{proof}
Applying a similar argument as in the Proof of Corollary \ref{8.3}, we then obtain the count, as required.
\end{proof}

\subsection*{On Fields $K_{f}$ \& $L_{g}$ with Bounded Absolute Discriminant \& Prescribed Galois group}

Recall that we proved in Corollary \ref{8.1} the existence of an infinite family of irreducible monic integer polynomials $f(x) = \varphi_{p^{\ell},c}^{1+m}(x)-\varphi_{p^{\ell},c}(x)\in \mathbb{Q}[x]\subset K[x]$ for any fixed $\ell, m \in \mathbb{Z}_{\geq 1}$; and more to this, we may also recall that the second part of Theorem \ref{2.3} (i.e., the part in which $N_{c}^{(1_{m})}(p) = 0$ for every coefficient $c\not \equiv 0\ (\text{mod} \ p\mathcal{O}_{K})$) implies that the monic polynomial $f(x) = \varphi_{p^{\ell},c}^{1+m}(x)-\varphi_{p^{\ell},c}(x) \in \mathcal{O}_{K}[x]$ is irreducible modulo prime ideal $p\mathcal{O}_{K}$. Now as in Section \ref{sec7}, we may to every such irreducible polynomial $f\in \mathcal{O}_{K}[x]$ associate a field $K_{f} = K[x]\slash (f(x))$, which is again a number field of degree $\kappa = p^{(m+1)\ell}$ over $K$. Moreover, since we also obtain an inclusion $\mathbb{Q}\hookrightarrow K \hookrightarrow K_{f}$ of number fields, we then also note that the degree $t=[K_{f} : \mathbb{Q}] = [K : \mathbb{Q}] \cdot [K_{f} : K] = n\kappa$, for fixed degree $[K: \mathbb{Q}]=n\geq 1$. So now, as before we also wish to count the number of fields $K_{f}\slash \mathbb{Q}$ induced by irreducible monic integral polynomials $f\in \mathcal{O}_{K}[x]$ arising from a polynomial discrete dynamical system in Section \ref{sec2}. To that end, we again define and then also determine the asymptotic behavior of the following counting function
\begin{equation}\label{N_{m}}
N_{t}(X) := \# \Bigl\{K_{f}\slash \mathbb{Q} : [K_{f} : \mathbb{Q}] = t \textnormal{ and } |\text{Disc}(K_{f})|\leq X \Bigr\}
\end{equation} as a positive real number $X\to \infty$. But now motivated as before by \cite{lem} of Lemke Oliver-Thorne and then applying here the first part of their [\cite{lem}, Theorem 1.2] on the function $N_{t}(X)$, we then also obtain the following:

\begin{cor} \label{11.5} Fix any number field $K\slash \mathbb{Q}$ of degree $n\geq 1$ with the ring of integers $\mathcal{O}_{K}$. Assume Corollary \ref{8.1} or second part of Theorem \ref{2.3}, and let $N_{t}(X)$ be the number defined as in \textnormal{(\ref{N_{m}})}. Then we have 
\begin{equation}\label{N_{m}(x)} 
N_{t}(X)\ll_{t}X^{2d - \frac{d(d-1)(d+4)}{6t}}\ll X^{\frac{8\sqrt{t}}{3}}, \text{where d is the least integer for which } \binom{d+2}{2}\geq 2t + 1.
\end{equation}
\end{cor}

\begin{proof}
To see the inequality \textnormal{(\ref{N_{m}(x)})}, we first recall from Corollary \ref{8.1} the existence of infinitely many monic polynomials $f(x)$ over $\mathbb{Q}\subset K_{f}$ such that $K_{f}\slash \mathbb{Q}$ is an algebraic number field of degree $t=np^{(m+1)\ell}$, or recall from the second part of Theorem \ref{2.3} the existence of monic integral polynomials $f(x) = \varphi_{p^{\ell},c}^{1+m}(x)-\varphi_{p^{\ell},c}(x)\in K[x]$ that are irreducible modulo any fixed prime ideal $p\mathcal{O}_{K}$ for every coefficient $c\not \in p\mathcal{O}_{K}$ and hence induce degree-$t$ number fields $K_{f}\slash \mathbb{Q}$. This then also means that the set of degree-$t$ number fields $K_{f}\slash \mathbb{Q}$ is not empty. But then applying here [\cite{lem}, Theorem 1.2 (1)] on $N_{t}(X)$, we then immediately obtain the inequality \textnormal{(\ref{N_{m}(x)})}, as needed.
\end{proof}

Similarly, recall that we proved in Corollary \ref{8.2} the existence of an infinite family of irreducible monic integer polynomials $g(x) = \varphi_{(p-1)^{\ell},c}^{1+m}(x)-\varphi_{(p-1)^{\ell},c}(x) \in \mathbb{Q}[x]\subset K[x]$ for any fixed $\ell, m \in \mathbb{Z}_{\geq 1}$; and more to this, we again recall that the second part of Theorem \ref{3.3} (i.e., the part in which $M_{c}^{(1_{m})}(p) = 0$ for every coefficient $c\not \equiv \pm1, 0\ (\text{mod} \ p\mathcal{O}_{K})$) implies the monic polynomial $g(x) = \varphi_{(p-1)^{\ell},c}^{1+m}(x)-\varphi_{(p-1)^{\ell},c}(x) \in \mathcal{O}_{K}[x]$ is irreducible modulo prime $p\mathcal{O}_{K}$. As before, we may to each irreducible polynomial $g\in \mathcal{O}_{K}[x]$ associate a number field $L_{g} = K[x]\slash (g(x))$ of even degree $\upsilon = (p-1)^{(m+1)\ell}$ over $K$. Since we now also obtain an inclusion $\mathbb{Q}\hookrightarrow K \hookrightarrow L_{g}$ of number fields, we then also note $r=[L_{g} : \mathbb{Q}] = [K : \mathbb{Q}] \cdot [L_{g} : K] = n\upsilon$, for any fixed $n=[K : \mathbb{Q}]$. So now, we also wish to count number fields $L_{g}\slash \mathbb{Q}$ induced by irreducible monic polynomials $g\in \mathcal{O}_{K}[x]$ arising from a polynomial discrete dynamical system in Section \ref{sec3}. With that in mind, we again define and then also determine  
\begin{equation}\label{M_{r}}
M_{r}(X) := \# \Bigl\{L_{g}\slash \mathbb{Q} : [L_{g} : \mathbb{Q}] = r \textnormal{ and} \ |\text{Disc}(L_{g})|\leq X \Bigr\}
\end{equation} as a positive real number $X\to \infty$. By again taking great advantage of [\cite{lem}, Theorem 1.2 (1)], we then obtain:

\begin{cor} Fix any number field $K\slash \mathbb{Q}$ of degree $n\geq 1$ with the ring of integers $\mathcal{O}_{K}$. Assume Corollary \ref{8.2} or second part of Theorem \ref{3.3}, and let $M_{r}(X)$ be the number defined as in \textnormal{(\ref{M_{r}})}. Then we have 
\begin{equation}\label{M_{r}(x)}
M_{r}(X)\ll_{r}X^{2d - \frac{d(d-1)(d+4)}{6r}}\ll X^{\frac{8\sqrt{r}}{3}}, \text{where d is the least integer for which } \binom{d+2}{2}\geq 2r + 1.
\end{equation}
\end{cor}

\begin{proof}
Applying a similar argument as in the Proof of Cor. \ref{11.5}, we then obtain inequality \textnormal{(\ref{M_{r}(x)})}, as needed.
\end{proof}

As before, we also wish to apply again that same result [\cite{sch1}, Corollary 1.3] due to Bhargava-Shankar-Wang on the number of degree-$t$ number fields $K_{f}\slash \mathbb{Q}$ induced by irreducible monic polynomials $f\in \mathcal{O}_{K}[x]$ arising from a polynomial discrete dynamical system in Section \ref{sec2}, that are monogenic and such that the associated Galois group Gal$(K_{f}\slash \mathbb{Q})$ is equal to the symmetric group $S_{t}$. In doing so, we then obtain the following corollary:

\begin{cor}\label{11.7}
Assume Corollary \ref{8.1} or second part of Theorem \ref{2.3}. The number of isomorphism classes of number fields $K_{f}\slash \mathbb{Q}$ of degree $t$ and $|\Delta(K_{f})| < X$ that are monogenic and having Galois group $S_{t}$ is $\gg X^{\frac{1}{2} + \frac{1}{t}}$.
\end{cor}

\begin{proof}
To see this, we first recall from Corollary \ref{8.1} the existence of infinitely many monic polynomials $f(x)$ over $\mathbb{Q}\subset K_{f}$ such that $K_{f}\slash \mathbb{Q}$ is an algebraic  number field of degree $t=np^{(m+1)\ell}$, or recall from the second part of Theorem \ref{2.3} the existence of monic integral polynomials $f(x) = \varphi_{p^{\ell},c}^{1+m}(x)-\varphi_{p^{\ell},c}(x)\in K[x]$ that are irreducible modulo any fixed prime ideal $p\mathcal{O}_{K}$ for every coefficient $c\not \in p\mathcal{O}_{K}$ and hence induce degree-$t$ number fields $K_{f}\slash \mathbb{Q}$. This then means that the set of degree-$t$ number fields $K_{f}\slash \mathbb{Q}$ is not empty. But now applying [\cite{sch1}, Corollary 1.3] on the underlying fields $K_{f}$ with $|\Delta(K_{f})| < X$ that are monogenic and have associated Galois group $S_{t}$, we then obtain that the number of isomorphism classes of such fields $K_{f}$ is $\gg X^{\frac{1}{2} + \frac{1}{t}}$, as required.
\end{proof}

Similarly, we also wish to apply again that same result [\cite{sch1}, Corollary 1.3] due to Bhargava-Shankar-Wang on the number of degree-$r$ number fields $L_{g}\slash \mathbb{Q}$ induced by irreducible polynomials $g\in \mathcal{O}_{K}[x]$ arising from a polynomial discrete dynamical system in Section \ref{sec3}, that are monogenic and such that the associated Galois group Gal$(L_{g}\slash \mathbb{Q})$ is equal to the symmetric group $S_{r}$. In doing so, we then also obtain the following corollary:
\begin{cor}
Assume Corollary \ref{8.2} or second part of Theorem \ref{3.3}. The number of isomorphism classes of number fields $L_{g}\slash \mathbb{Q}$ of degree $r$ and $|\Delta(L_{g})| < X$ that are monogenic and having Galois group $S_{r}$ is $\gg X^{\frac{1}{2} + \frac{1}{r}}$.
\end{cor}

\begin{proof}
Applying a similar argument as in the Proof of Corollary \ref{11.7}, we then obtain the count, as required.
\end{proof}

\section{On the Number of Algebraic Number fields $K_{f}$ and $L_{g}$ with Prescribed Class Number}\label{sec9}

Recall that for any number field $K$ with ring of integers $\mathcal{O}_{K}$, we then have an abelian group called \say{\textit{ideal class group}} $\textnormal{Cl}(K)$ corresponding to $K$; and whose order (called the \say{\textit{class number}} of $K$ (denoted as $h_{K}$)) is finite.

So now, recall from Corollary \ref{8.1} that there is an infinite family of irreducible monic polynomials $f(x) = \varphi_{p^{\ell},c}^{1+m}(x)-\varphi_{p^{\ell},c}(x)\in \mathbb{Z}[x]$ such that $\mathbb{Q}_{f}=\mathbb{Q}[x]\slash (f(x))$ is a number field of degree $p^{(m+1)\ell}$; and moreover to every such field $\mathbb{Q}_{f}$, we then also have a class group $\textnormal{Cl}(\mathbb{Q}_{f})$ with finite $h_{\mathbb{Q}_{f}}$. So now, inspired (as in [\cite{BK22}, Section 12]) by work \cite{ho} of Ho-Shankar-Varma on odd degree number fields with odd class number, we in this section also wish to count the number of fields $\mathbb{Q}_{f}$ induced by irreducible polynomials $f\in \mathbb{Z}[x]$ arising from a polynomial discrete dynamical system in Section \ref{sec2} (and ascertained by Corollary \ref{8.1}), with associated Galois group $S_{p^{(m+1)\ell}}$ and with prescribed $h_{\mathbb{Q}_{f}}$. With that in mind, we take great advantage of [\cite{ho}, Theorem 4] and then obtain the following corollary on the existence of infinitely many $S_{p^{(m+1)\ell}}$-fields $\mathbb{Q}_{f}$ with odd class number:

\begin{cor}\label{12.1}
Assume Corollary \ref{8.1}, and let $\kappa=p^{(m+1)\ell}$ be any fixed odd integer. Then there exist infinitely many $S_{\kappa}$-algebraic number fields $\mathbb{Q}_{f}$ of odd degree $\kappa$  having odd class number.  More precisely, we have 
\begin{center}
$\#\Bigl\{ \mathbb{Q}_{f} : |\Delta(\mathbb{Q}_{f})| < X \textnormal{ and } 2\nmid |\textnormal{Cl}(\mathbb{Q}_{f})| \Bigr\}\gg X^{\frac{\kappa + 1}{2\kappa -2}}$,
\end{center} where the implied constants depend on degree $\kappa$ and on an arbitrary finite set $S$ of primes given as in \textnormal{\cite{ho}}.
\end{cor}

\begin{proof}
From Cor. \ref{8.1}, it follows that the family of fields $\mathbb{Q}_{f}$ of degree $\kappa = p^{(m+1)\ell}$ is not empty. Now since $\kappa$ is an odd integer, we then see that the claim follows from [\cite{ho}, Thm. 4(a)] by setting $\mathbb{Q}_{f}=K$ as needed.
\end{proof}

Similarly, recall from the second part of Theorem \ref{2.3} the existence of monic integral polynomials $f(x) = \varphi_{p^{\ell},c}^{1+m}(x)-\varphi_{p^{\ell},c}(x)\in \mathcal{O}_{K}[x]$ that are irreducible modulo prime $p\mathcal{O}_{K}$; and moreover every such irreducible monic polymomial $f\in \mathcal{O}_{K}[x]$ induces an algebraic number field $K_{f}\slash \mathbb{Q}$ of degree $t=np^{(m+1)\ell}$, for any fixed $n=[K:\mathbb{Q}]$. Now assuming that $n$ is an odd integer and so is also degree $t$, we then also obtain the following corollary on the number of fields $K_{f}\slash \mathbb{Q}$ induced by irreducible monic polynomials $f$ arising from a polynomial discrete dynamical system in Section \ref{sec2}, with associated Galois group $S_{np^{(m+1)\ell}}$ and also having odd class number:  

\begin{cor}
Assume second part of Theorem \ref{2.3}, and let $t=np^{(m+1)\ell}$ be any fixed odd integer. There exist infinitely many $S_{t}$-number fields $K_{f}$ of degree $t$ with odd class number.  More precisely, we have 
\begin{center}
$\# \Bigl\{K_{f}\slash \mathbb{Q} : |\Delta(K_{f})| < X \textnormal{ and } 2\nmid |\textnormal{Cl}(K_{f})| \Bigr\}\gg X^{\frac{t + 1}{2t -2}}$,
\end{center} where the implied constants depend on degree $t$ and on an arbitrary finite set $S$ of primes given as in \textnormal{\cite{ho}}.
\end{cor}

\begin{proof}
Applying a similar argument as in the Proof of Corollary \ref{12.1}, we then obtain the count, as needed. 
\end{proof}

As before, we may also recall from Corollary \ref{8.2} the existence of an infinite family of irreducible monic polynomials $g(x) = \varphi_{(p-1)^{\ell},c}^{1+m}(x)-\varphi_{(p-1)^{\ell},c}(x)\in \mathbb{Z}[x]$ such that $\mathbb{Q}_{g} = \mathbb{Q}[x]\slash (g(x))$ is a number field of degree $(p-1)^{(m+1)\ell}$. But now we also observe that to each such obtained number field $\mathbb{Q}_{g}$, we can associate a finite class group $\textnormal{Cl}(\mathbb{Q}_{g})$ and so $h_{\mathbb{Q}_{g}}$ is finite. So now, by taking again great advantage of work on class groups of number fields in arithmetic statistics and in particular the work \cite{Sia} of Siad on $S_{n}$-number fields $K$ of any even degree $n\geq 4$ and signature $(r_{1}, r_{2})$ where $r_{1}$ are the real embeddings of $K$  and $r_{2}$ are the pairs of conjugate complex embeddings of $K$, we then also obtain the following corollary on the number of number fields $\mathbb{Q}_{g}\slash \mathbb{Q}$ induced by irreducible monic polynomials $g\in \mathbb{Z}[x]$ arising from a polynomial discrete dynamical system in Section \ref{sec3} (and ascertained by Corollary \ref{8.2}), with associated Galois group $S_{(p-1)^{(m+1)\ell}}$ and odd class number: 

\begin{cor}\label{12.3}
Assume Corollary \ref{8.2}, and let $\upsilon=(p-1)^{(m+1)\ell}$ be any even integer. There are infinitely many degree-$\upsilon$ monogenic fields $\mathbb{Q}_{g}$ of any signature and associated Galois group $S_{\upsilon}$ having odd class number. 
\end{cor}
\begin{proof}
To see this, we note that by Cor. \ref{8.2}, it follows that the family of number fields $\mathbb{Q}_{g}$ of degree $\upsilon = (p-1)^{(m+1)\ell}$ is not empty. Now since $\upsilon$ is even, we see that the claim follows from [\cite{Sia}, Cor. 10] as required.
\end{proof}

As before, recall from the second part of Theorem \ref{3.3} the existence of monic polynomials $g(x) = \varphi_{(p-1)^{\ell},c}^{1+m}(x)-\varphi_{(p-1)^{\ell},c}(x)\in \mathcal{O}_{K}[x]$ that are irreducible modulo prime $p\mathcal{O}_{K}$; and moreover each such $g\in \mathcal{O}_{K}[x]$ induces an algebraic number field $L_{g}\slash \mathbb{Q}$ of degree $r=n(p-1)^{(m+1)\ell}$. As before, we also have the following corollary on the number of fields $L_{g}$ induced by irreducible polynomials $g\in \mathcal{O}_{K}[x]$ arising from a polynomial discrete dynamical system in Section \ref{sec3}, with associated Galois group $S_{n(p-1)^{(m+1)\ell}}$ and having odd class number:

\begin{cor}
Assume second part of Theorem \ref{3.3}, and let $r=n(p-1)^{(m+1)\ell}$ be any fixed even integer. Then there are infinitely many monogenic $S_{r}$-fields $L_{g}$ of degree $r$ and any signature having odd class number. 
\end{cor}

\begin{proof}
Applying a similar argument as in the Proof of Corollary \ref{12.3}, we then obtain the count, as required. 
\end{proof}

\section{On the Number of Number fields $K_{f}$ and $L_{g}$ having Prescribed Narrow Class Number}

As in Section \ref{sec9} we recall from algebraic number theory that for any number field $K$ with ring of integers $\mathcal{O}_{K}$, the \say{\textit{narrow class group}} $\textnormal{Cl}^{+}(K)$ of $K$ is  defined as $\textnormal{I}(\mathcal{O}_{K}) \slash \textnormal{P}^{+}(\mathcal{O}_{K})$, where $\textnormal{I}(\mathcal{O}_{K})$ is a free abelian group of all non-zero fractional ideals of $\mathcal{O}_{K}$, and $\textnormal{P}^{+}(\mathcal{O}_{K})$ is the group of non-zero principal fractional ideals $\alpha \mathcal{O}_{K}$ such that $\sigma(\alpha)$ is positive for every embedding $\sigma : K\to \mathbb{R}$. The order of $\textnormal{Cl}^{+}(K)$ is called the \say{\textit{narrow class number}}.

So now, inspired again by that same work \cite{ho} on number fields with odd narrow class number, we in this section also wish to count the number of fields $\mathbb{Q}_{f}$ induced by irreducible polynomials $f\in \mathbb{Z}[x]$ arising from a polynomial discrete dynamical system in Section \ref{sec2} (and ascertained by Corollary \ref{8.1}), with associated Galois group $S_{p^{(m+1)\ell}}$ and prescribed narrow class number. To do so, we again take great advantage of [\cite{ho}, Theorem 4(b)] and then obtain the following corollary on the number of $S_{p^{(m+1)\ell}}$-number fields $\mathbb{Q}_{f}$ with odd narrow class number, whenever we know that the number $r_{2}$ of pairs of conjugate complex embeddings of $\mathbb{Q}_{f}$ is at least one:

\begin{cor}
Assume Corollary \ref{8.1}, and fix any $\kappa=p^{(m+1)\ell}$. If $r_{2}\geq 1$, then there exist infinitely many degree-$\kappa$ $S_{\kappa}$-number fields $\mathbb{Q}_{f}$ having signature $(r_{1}, r_{2})$ and odd narrow class number.  Specifically, we have 
\begin{center}
$\# \Bigl\{ \mathbb{Q}_{f} : |\Delta(\mathbb{Q}_{f})| < X \textnormal{ and } 2\nmid |\textnormal{Cl}^{+}(\mathbb{Q}_{f})| \Bigr\}\gg X^{\frac{\kappa + 1}{2\kappa -2}}$,
\end{center} where the implied constants depend on degree $\kappa$ and on an arbitrary finite set $S$ of primes given as in \textnormal{\cite{ho}}.
\end{cor}

\begin{proof}
Because of Corollary \ref{8.1}, it then follows that there are infinitely many irreducible polynomials $f(x) = \varphi_{p^{\ell},c}^{1+m}(x)-\varphi_{p^{\ell},c}(x)\in \mathbb{Z}[x]$ such that $\mathbb{Q}_{f}=\mathbb{Q}[x]\slash (f(x))$ is a number field of degree $\kappa=p^{(m+1)\ell}$. Thus the family of degree-$\kappa$ fields $\mathbb{Q}_{f}$ is not empty. So now, since $\kappa$ is odd and also since $r_{2}\geq 1$ by assumption, applying here [\cite{ho}, Theorem 4(b)] on the underlying family of such fields $\mathbb{Q}_{f}=K$, we then obtain the count, as required.
\end{proof}

Motivated again by that same work \cite{ho} of Ho-Shankar-Varma on number fields with units of every signature, we then apply their work on number fields $\mathbb{Q}_{f}$ induced by irreducible polynomials $f\in \mathbb{Z}[x]$ arising from a polynomial discrete dynamical system in Section \ref{sec2} (and ascertained by Corollary \ref{8.1}); and then obtain:

\begin{cor}\label{13.2}
Assume Corollary \ref{8.1}, and let $\kappa=p^{(m+1)\ell}$ be any fixed odd integer. Suppose $r_{2}\geq 1$, then there exist infinitely many $S_{\kappa}$-number fields $\mathbb{Q}_{f}$ of degree $\kappa$ with signature $(r_{1}, r_{2})$ for which $|\textnormal{Cl}^{+}(\mathbb{Q}_{f})|=|\textnormal{Cl}(\mathbb{Q}_{f})|$. In particular, there exist infinitely many such $S_{\kappa}$-number fields $\mathbb{Q}_{f}$ of degree $\kappa$ that have units of every signature.
\end{cor}

\begin{proof}
To see this, we note that from Corollary \ref{8.1}, it follows that the family of number fields $\mathbb{Q}_{f}$ of degree $\kappa = p^{(m+1)\ell}$ is not empty. But now since $\kappa$ is an odd integer and also since $r_{2}\geq 1$ by assumption, applying here [\cite{ho}, Corollary 5] on the underlying family of such number fields $\mathbb{Q}_{f}$, we then obtain the count, as required.
\end{proof}

Similarly, assuming the degree $n=[K:\mathbb{Q}]$ is odd, we then also obtain the following corollary on the number of fields $K_{f}\slash \mathbb{Q}$ induced by irreducible polynomials $f\in \mathcal{O}_{K}[x]$ arising from a polynomial discrete dynamical system in Section \ref{sec2}, with associated Galois group $S_{np^{(m+1)\ell}}$ and also having odd narrow class number:

\begin{cor}
Assume second part of Theorem \ref{2.3}, and fix any odd  $t=np^{(m+1)\ell}$. If $r_{2}\geq 1$, then there exist infinitely many degree-$t$ $S_{t}$-fields $K_{f}$ having signature $(r_{1}, r_{2})$ and odd narrow class number.  More precisely,  
\begin{center}
$\# \Bigl\{K_{f}\slash \mathbb{Q} : |\Delta(K_{f})| < X \textnormal{ and } 2\nmid |\textnormal{Cl}^{+}(K_{f})| \Bigr\}\gg X^{\frac{t + 1}{2t -2}}$,
\end{center} where the implied constants depend on degree $t$ and on an arbitrary finite set $S$ of primes given as in \textnormal{\cite{ho}}.
\end{cor}

\begin{proof}
Because of second part of Theorem \ref{2.3}, it then follows that there are exists irreducible monic integral polynomials $f(x) = \varphi_{p^{\ell},c}^{1+m}(x)-\varphi_{p^{\ell},c}(x)\in \mathcal{O}_{K}[x]\subset K[x]$; and moreover as noted earlier that every such irreducible monic polymomial $f\in \mathcal{O}_{K}[x]$ induces a number field $K_{f}=K[x]\slash (f(x))$ of degree $t=np^{(m+1)\ell}$ over $\mathbb{Q}$, for every fixed odd degree $n=[K:\mathbb{Q}]$. This then also means that the family of degree-$t$ number fields $K_{f}\slash \mathbb{Q}$ is not empty. But now since $t>3$ is an odd integer and also since $r_{2}\geq 1$ by assumption, applying here [\cite{ho}, Theorem 4(b)] on the underlying family of such number fields $K_{f}\slash \mathbb{Q}$, we then obtain the count, as required. 
\end{proof}

\begin{cor}
Assume second part of Theorem \ref{2.3}, and fix any odd $t=np^{(m+1)\ell}$. Suppose $r_{2}\geq 1$, then there exist infinitely many $S_{t}$-number fields $K_{f}$ of degree $\kappa$ with signature $(r_{1}, r_{2})$ for which $|\textnormal{Cl}^{+}(K_{f})|=|\textnormal{Cl}(K_{f})|$. In particular, there exist infinitely many such $S_{t}$-number fields $K_{f}$ of degree $t$ that have units of every signature.
\end{cor}

\begin{proof}
Applying a similar argument as in the Proof of Corollary \ref{13.2}, we then obtain the count, as needed.
\end{proof}

Similarly, motivated again by that same work \cite{Sia} of Siad on even degree number fields with units of every signature, we then apply here his work on fields $\mathbb{Q}_{g}$ induced by irreducible polynomials $g\in \mathbb{Z}[x]$ arising from a polynomial discrete dynamical system in Section \ref{sec3} (and ascertained by Corollary \ref{8.2}); and then obtain:

\begin{cor}
Assume Corollary \ref{8.2}, and let $\upsilon=(p-1)^{(m+1)\ell}$ be any even integer. Then there are an infinite number of degree-$\upsilon$ monogenic $S_{\upsilon}$-fields $\mathbb{Q}_{g}$ of any signature $(r_{1}, r_{2})$ that have units of every signature. 
\end{cor}

\begin{proof}
To see this, we note that because of Corollary \ref{8.2}, it then follows that there are infinitely many irreducible polynomials $g(x) = \varphi_{(p-1)^{\ell},c}^{1+m}(x)-\varphi_{(p-1)^{\ell},c}(x)\in \mathbb{Z}[x]$ such that $\mathbb{Q}_{g}=\mathbb{Q}[x]\slash (g(x))$ is a number field of degree $\upsilon=(p-1)^{(m+1)\ell}$. This then also means that the family of degree-$\upsilon$ number fields $\mathbb{Q}_{g}$ is not empty. But now since $\upsilon> 4$ is an even integer, we then note that applying here [\cite{Sia}, Corollary 11] on the underlying family of degree-$\upsilon$ number fields $\mathbb{Q}_{g}$, we then obtain the count; and which then completes the whole proof, as needed. 
\end{proof}

As before, we then also obtain the following corollary on the number of even degree number fields $L_{g}\slash \mathbb{Q}$ induced by irreducible monic polynomials $g\in \mathcal{O}_{K}[x]$ arising from a polynomial discrete dynamical system in Section \ref{sec3}, with associated Galois group $S_{n(p-1)^{(m+1)\ell}}$ and of every signature and having units of every signature:

\begin{cor}
Assume second part of Theorem \ref{3.3}, and fix any even $r=n(p-1)^{(m+1)\ell}$. Then there are an infinite number of degree-$r$ monogenic $S_{r}$-fields $L_{g}$ of any signature $(r_{1}, r_{2})$ that have units of every signature.
\end{cor}

\begin{proof}
Because of second part of Theorem \ref{3.3}, it then follows that there are exists irreducible monic integral polynomials $g(x) = \varphi_{(p-1)^{\ell},c}^{1+m}(x)-\varphi_{(p-1)^{\ell},c}(x)\in \mathcal{O}_{K}[x]\subset K[x]$; and moreover as noted earlier that every such irreducible monic polymomial $g\in \mathcal{O}_{K}[x]$ induces a number field $L_{g}=K[x]\slash (g(x))$ of degree $r=n(p-1)^{(m+1)\ell}$ over $\mathbb{Q}$, for every fixed degree $n=[K:\mathbb{Q}]$. Hence the family of degree-$r$ number fields $L_{g}$ is not empty. But now since $r> 4$ is an even integer, we then note that applying here [\cite{Sia}, Corollary 11] on the underlying family of degree-$r$ number fields $L_{g}$, we then obtain the count; and which then completes the whole proof, as needed.
\end{proof}

\section{On the Equidistribution of Families of Artin $L$-Functions induced by Fields $K_{f}$ and $L_{g}$}
As in [\cite{BK22}, Section 13] recall that for every degree-$n$ every number field $K$ with the ring of integers $\mathcal{O}_{K}$, we then have a Dedekind zeta function $\zeta_{K}$ corresponding to $K$. Moreover, we may also recall from [\cite{Nico}, Page 10] that this zeta function $\zeta_{K}(s)$ factors as $\zeta_{K}(s)=\zeta_{\mathbb{Q}}(s)L(s, \rho_{K}$), where $L(s, \rho_{K})$ is the Artin $L$-function corresponding to an Artin representation $\rho_{K}: \text{Gal}(\mathbb{Q})\to \text{Gal}(M\slash \mathbb{Q}) \hookrightarrow S_{n}\to \text{GL}_{n-1}(\mathbb{C})$, and $M$ is the normal closure of $K$.

So now, for every degree-$\kappa$ number field $\mathbb{Q}_{f}$ obtained from a polynomial discrete dynamical system in Section \ref{sec2} (and ascertained by Corollary \ref{8.1}), we then have a Dedekind zeta function $\zeta_{\mathbb{Q}_{f}}$ corresponding to $\mathbb{Q}_{f}$. Moreover, as also noted in \cite{BK22} that from the work of Shankar-S\"{o}dergren-Templier [\cite{Nico}, Page 2], this obtained Dedekind zeta function $\zeta_{\mathbb{Q}_{f}}(s)=\zeta(s)L(s, \rho_{\mathbb{Q}_{f}}$), where $\zeta(s)$ is the Riemann zeta function, $L(s, \rho_{\mathbb{Q}_{f}})$ is the Artin $L$-function, $\rho_{\mathbb{Q}_{f}}: \text{Gal}(M_{f}\slash \mathbb{Q}) \hookrightarrow S_{\kappa}\to \text{GL}_{\kappa-1}(\mathbb{C})$ is a representation, and $M_{f}$ being the normal closure of $\mathbb{Q}_{f}$.

Now motivated (as in \cite{BK22}) by remarkable work \cite{Nico} of Shankar-S\"{o}dergren-Templier on equidistribution of Artin $L$-functions arising from number fields induced by irreducible monic integer polynomials, we in the same spirit as in \cite{Nico} also wish to study the distribution of Artin $L$-functions $L(s, \rho_{\mathbb{Q}_{f}})$ arising from number fields $\mathbb{Q}_{f}$ induced by irreducible monic polynomials $f$ obtained from a polynomial discrete dynamical system in Section \ref{sec2}. To do so, we (assuming Corollary \ref{8.1}) wish to first adhere to the setup and notation in \cite{Nico}. That is, let $V(\mathbb{Z})^{\text{irr}}$ be the space consisting of irreducible monic integer polynomials  $f(x)=\varphi_{p^{\ell},c}^{1+m}(x)-\varphi_{p^{\ell},c}(x)$ of fixed degree $\kappa=p^{(m+1)\ell}$,  and let $V(\mathbb{Z})^{\text{max}}\subset V(\mathbb{Z})^{\text{irr}}$ be a subset consisting of irreducible monic integer polynomials $f$ such that $R_{f}=\mathbb{Z}[x]\slash (f(x))$ is a maximal order in $\mathbb{Q}_{f}=\mathbb{Q}[x]\slash (f(x))$. Following \cite{Nico}, it also follows here that the additive group $G_{a}(\mathbb{Z})=\mathbb{Z}$ necessarily acts naturally on our space $V(\mathbb{Z})^{\text{irr}}$ via translation, namely, $(b \cdot f)(x):= f(x+b)$ for every element $b\in \mathbb{Z}$ and for every $f\in V(\mathbb{Z})^{\text{irr}}$; and moreover, this action of $G_{a}(\mathbb{Z})=\mathbb{Z}$ by translation also necessarily preserves each of the sets $V(\mathbb{Z})^{\text{irr}}$ and $V(\mathbb{Z})^{\text{max}}$. Now let $\mathfrak{F}_{1}$ be a family consisting of the $\mathbb{Z}$-orbits on $V(\mathbb{Z})^{\text{max}}$. It then follows (from \cite{Nico}) that the family $\mathfrak{F}_{1}$ necessarily parametrizes degree-$\kappa$ monogenized fields $(\mathbb{Q}_{f}, \alpha)$ over $\mathbb{Q}$ up to isomorphism. We note that (from [\cite{Nico}, Subsection 2.3]) this same $\mathfrak{F}_{1}$ parametrizing  degree-$\kappa$ monogenized number fields $(\mathbb{Q}_{f}, \alpha)$ is also the family of associated $L$-functions $L(s, \rho_{\mathbb{Q}_{f}})$.

So now, by taking great advantage of a theorem [\cite{Nico}, Theorem 1.1] due to Shankar-S\"{o}dergren-Templier, we then also obtain the following corollary on the family $\mathfrak{F}_{1}$ parametrizing degree-$\kappa$ monogenized fields $(\mathbb{Q}_{f}, \alpha)$:

\begin{cor}\label{14.1}
Assume Corollary \ref{8.1}, and let $\mathfrak{F}_{1}$ be as before. Then $\mathfrak{F}_{1}$ parametrizing monogenized degree-$\kappa$ fields  ordered by height $h(f)$ as defined in \textnormal{\cite{Nico}} satisfies Sato-Tate equidistribution in the sense of \textnormal{[\cite{Sar}, Conj.1]}. 
\end{cor}

\begin{proof}
Since we know from Corollary \ref{8.1} that there are infinitely many irreducible monic integer polynomials $f$ such that $\mathbb{Q}_{f}$ is a number field of degree $\kappa=p^{(m+1)\ell}$, then this also means that the family of degree-$\kappa$ number fields $\mathbb{Q}_{f}\slash \mathbb{Q}$ is not empty. Now letting $\alpha$ be the image of $x$ in $R_{f}=\mathbb{Z}[x]\slash (f(x))$ and so (by \cite{Nico}) the pair $(\mathbb{Q}_{f}, \alpha)$ is a degree-$\kappa$ monogenized field, it then follows that the family of monogenized degree-$\kappa$ fields $(\mathbb{Q}_{f}, \alpha)$ is not empty; which also means that the family $\mathfrak{F}_{1}$ parametrizing degree-$\kappa$ monogenized fields $(\mathbb{Q}_{f}, \alpha)$ is not empty. But now applying [\cite{Nico}, Thm. 1.1] to the underlying family $\mathfrak{F}_{1}$ ordered by height $h(f)$ as defined in [\cite{Nico}, Page 3], it then follows that $\mathfrak{F}_{1}$ satisfies Sato-Tate equidistribution in the sense of \textnormal{[\cite{Sar}, Conjecture 1]} as needed.
\end{proof}

Similarly, for every degree-$\upsilon$ field $\mathbb{Q}_{g}$ obtained from a polynomial discrete dynamical system in Section \ref{sec3} (and ascertained by Corollary \ref{8.2}), we also have a Dedekind zeta function $\zeta_{\mathbb{Q}_{g}}$ corresponding to $\mathbb{Q}_{g}$. Moreover, it again follows from \cite{Nico} that the Dedekind zeta function $\zeta_{\mathbb{Q}_{g}}(s)=\zeta(s)L(s, \rho_{\mathbb{Q}_{g}}$), where $L(s, \rho_{\mathbb{Q}_{g}})$ is the Artin $L$-function,  $\rho_{\mathbb{Q}_{g}}: \text{Gal}(M_{g}\slash \mathbb{Q}) \hookrightarrow S_{\upsilon}\to \text{GL}_{\upsilon-1}(\mathbb{C})$ is an Artin representation, and $M_{g}$ the normal closure of $\mathbb{Q}_{g}$. 

So now, in again the same spirit as in \cite{Nico}, we also wish to study the distribution of Artin $L$-functions $L(s, \rho_{\mathbb{Q}_{g}})$ arising from fields $\mathbb{Q}_{g}$ induced by irreducible polynomials $g$ obtained from a polynomial discrete dynamical system in Section \ref{sec3}. To that end, we (also assuming Corollary \ref{8.2}) adhere again to the setup and notation in \cite{Nico}. That is, we again let $W(\mathbb{Z})^{\text{irr}}$ be the space consisting of irreducible monic integer polynomials  $g(x)=\varphi_{(p-1)^{\ell},c}^{1+m}(x)-\varphi_{(p-1)^{\ell},c}(x)$ of fixed degree $\upsilon=(p-1)^{(m+1)\ell}$,  and let $W(\mathbb{Z})^{\text{max}}\subset W(\mathbb{Z})^{\text{irr}}$ be a subset consisting of irreducible polynomials $g$ such that $R_{g}=\mathbb{Z}[x]\slash (g(x))$ is a maximal order in $\mathbb{Q}_{g}=\mathbb{Q}[x]\slash (g(x))$. Following again \cite{Nico}, it also follows here that $G_{a}(\mathbb{Z})=\mathbb{Z}$ necessarily acts naturally on $W(\mathbb{Z})^{\text{irr}}$ via translation, namely, $(b \cdot g)(x):= g(x+b)$ for every $b\in \mathbb{Z}$ and for every $g\in W(\mathbb{Z})^{\text{irr}}$; and moreover, this action of $G_{a}(\mathbb{Z})=\mathbb{Z}$ by translation also necessarily preserves each of $W(\mathbb{Z})^{\text{irr}}$ and $W(\mathbb{Z})^{\text{max}}$. Now let $\mathfrak{F}_{2}$ be a family consisting of the $\mathbb{Z}$-orbits on $W(\mathbb{Z})^{\text{max}}$. It then follows (from \cite{Nico}) that the family $\mathfrak{F}_{2}$ necessarily parametrizes degree-$\upsilon$ monogenized number fields $(\mathbb{Q}_{g}, \beta)$ up to isomorphism. As before, we also note that (from [\cite{Nico}, Subsect.2.3]) this same family $\mathfrak{F}_{2}$ parametrizing  degree-$\upsilon$ monogenized fields $(\mathbb{Q}_{g}, \beta)$ is also the family of associated $L$-functions $L(s, \rho_{\mathbb{Q}_{g}})$. Again, taking great advantage of [\cite{Nico}, Theorem 1.1], we then obtain the following corollary on $\mathfrak{F}_{2}$:

\begin{cor}
Assume Corollary \ref{8.2}, and let $\mathfrak{F}_{2}$ be as before. Then $\mathfrak{F}_{2}$ parametrizing monogenized degree-$\upsilon$ fields  ordered by height $h(g)$ as defined in \textnormal{\cite{Nico}} satisfies Sato-Tate equidistribution in the sense of \textnormal{[\cite{Sar}, Conj.1]}. 
\end{cor}

\begin{proof}
By applying a similar argument as in the Proof of Corollary \ref{14.1}, it then also follows immediately that the family $\mathfrak{F}_{2}$ satisfies  Sato-Tate equidistribution in the sense of \textnormal{[\cite{Sar}, Conjecture 1]} as also indeed needed.
\end{proof} 

Similarly, recall from the second part of Theorem \ref{2.3} the existence of monic integral polynomials $f(x) = \varphi_{p^{\ell},c}^{1+m}(x)-\varphi_{p^{\ell},c}(x)\in K[x]$ that are irreducible modulo prime $p\mathcal{O}_{K}$; and moreover every such irreducible $f\in \mathcal{O}_{K}[x]$ induces a number field $K_{f}\slash \mathbb{Q}$ of degree $t=np^{(m+1)\ell}$, for every fixed $n=[K: \mathbb{Q}]$. Moreover, it also follows from primitive element theorem that we can write $K_{f}= \mathbb{Q}(\gamma)\simeq \mathbb{Q}[x]\slash (h_{1}(x)) = \mathbb{Q}_{h_{1}}$, where $\gamma$ is some algebraic number in $K_{f}$ and $h_{1}\in \mathbb{Q}[x]$ is the characteristic polynomial (also the minimal polynomial) of $\gamma$. So now, for every degree-$t=np^{(m+1)\ell}$ number field $\mathbb{Q}_{h_{1}}$ induced by a polynomial discrete dynamical system in second part of Theorem \ref{2.3}, we have a corresponding zeta function $\zeta_{\mathbb{Q}_{h_{1}}}$; and which also factors $\zeta_{\mathbb{Q}_{h_{1}}}(s)=\zeta(s)L(s, \rho_{\mathbb{Q}_{h_{1}}}$), where $L(s, \rho_{\mathbb{Q}_{h_{1}}})$ is the Artin $L$-function corresponding to a representation $\rho_{\mathbb{Q}_{h_{1}}}: \text{Gal}(M^{(t)}_{h_{1}}\slash \mathbb{Q}) \hookrightarrow S_{t}\to \text{GL}_{t-1}(\mathbb{C})$, and $M^{(t)}_{h_{1}}$ is the normal closure of $\mathbb{Q}_{h_{1}}$. Since $\gamma\in K_{f}$ can be some element such that $K_{f}=\mathbb{Q}(\gamma)\simeq \mathbb{Q}_{h_{1}}$, we again also wish to study the distribution of Artin $L$-functions $L(s, \rho_{\mathbb{Q}_{h_{1}}})$ arising from fields $\mathbb{Q}_{h_{1}}$ induced by irreducible polynomials $h_{1}$. To that end, we let $V^{(t)}(\mathbb{Z})^{\text{irr}}$ be the space consisting of irreducible monic integer polynomials  $h_{1}$ of fixed degree $t=np^{(m+1)\ell}$,  and let $V^{(t)}(\mathbb{Z})^{\text{max}}\subset V^{(t)}(\mathbb{Z})^{\text{irr}}$ be a subset consisting of irreducible monic integer polynomials $h_{1}$ such that $R_{h_{1}}=\mathbb{Z}[x]\slash (h_{1}(x))$ is a maximal order in $\mathbb{Q}_{h_{1}}=\mathbb{Q}[x]\slash (h_{1}(x))$. As before, additive group $G_{a}(\mathbb{Z})=\mathbb{Z}$ acts naturally on our space $V^{(t)}(\mathbb{Z})^{\text{irr}}$ via translation, namely, $(b \cdot h_{1})(x):= h_{1}(x+b)$ for every $b\in \mathbb{Z}$ and  every $h_{1}\in V^{(t)}(\mathbb{Z})^{\text{irr}}$; and moreover this action of $G_{a}(\mathbb{Z})=\mathbb{Z}$ by translation again preserves each of $V^{(t)}(\mathbb{Z})^{\text{irr}}$ and $V^{(t)}(\mathbb{Z})^{\text{max}}$. So now, let $\mathfrak{F}^{(t)}_{1}$ be a family consisting of the $\mathbb{Z}$-orbits on $V^{(t)}(\mathbb{Z})^{\text{max}}$; and which as before parametrizes degree-$t$ monogenized number fields $(\mathbb{Q}_{h_{1}}, \gamma)$ over $\mathbb{Q}$ up to isomorphism. Note that as before, we also treat  the family $\mathfrak{F}^{(t)}_{1}$ to be the family of associated $L$-functions $L(s, \rho_{\mathbb{Q}_{h_{1}}})$. But now, by again taking  great advantage of [\cite{Nico}, Theorem 1.1], we then also immediately obtain the following corollary on  $\mathfrak{F}^{(t)}_{1}$:

\begin{cor}\label{14.3}
Assume second part of Theorem \ref{2.3}, and let $t=np^{(m+1)\ell}$ be any fixed odd integer. Let $\mathfrak{F}^{(t)}_{1}$ be a family of $\mathbb{Z}$-orbits defined as before. Then the family $\mathfrak{F}^{(t)}_{1}$ parametrizing monogenized degree-$t$ number fields ordered by height $h(h_{1})$ as given in \textnormal{\cite{Nico}} satisfies Sato-Tate equidistribution in the sense of \textnormal{[\cite{Sar}, Conjecture 1]}. 
\end{cor} 

\begin{proof}
To see this, we note that by second part of Theorem \ref{2.3}, it then follows that the family of number fields $K_{f}\slash \mathbb{Q}$ of degree $t = np^{(m+1)\ell}$ is not empty. Moreover, it then also follows from the discussion (right before the corollary that we are proving) that the family of degree-$t$ algebraic number field $\mathbb{Q}_{h_{1}}\slash \mathbb{Q}$ is non-empty. But now, by applying again a similar argument as in the Proof of Corollary \ref{14.1}, it then also follows immediately that the family $\mathfrak{F}^{(t)}_{1}$ satisfies  Sato-Tate equidistribution in the sense of \textnormal{[\cite{Sar}, Conjecture 1]} as also indeed needed.
\end{proof}

As before, we may also recall from the second part of Theorem \ref{3.3} the existence of monic integral polynomials $g(x) = \varphi_{(p-1)^{\ell},c}^{1+m}(x)-\varphi_{(p-1)^{\ell},c}(x)\in K[x]$ that are irreducible modulo prime $p\mathcal{O}_{K}$; and moreover every such irreducible $g\in \mathcal{O}_{K}[x]$ induces a number field $L_{g}\slash \mathbb{Q}$ of degree $r=n(p-1)^{(m+1)\ell}$, for every fixed degree $n=[K: \mathbb{Q}]$. Moreover, by primitive element theorem, we note that $L_{g}= \mathbb{Q}(\nu)\simeq \mathbb{Q}[x]\slash (h_{2}(x)) = \mathbb{Q}_{h_{2}}$, where $\nu$ is some algebraic number in $L_{g}$ and $h_{2}\in \mathbb{Q}[x]$ is the characteristic polynomial (also the minimal polynomial) of $\nu$. So now, for every degree-$r=n(p-1)^{(m+1)\ell}$ number field $\mathbb{Q}_{h_{2}}$ induced by a polynomial discrete dynamical system in the second part of Theorem \ref{3.3}, we have a corresponding zeta function $\zeta_{\mathbb{Q}_{h_{2}}}$; and which also factors $\zeta_{\mathbb{Q}_{h_{2}}}(s)=\zeta(s)L(s, \rho_{\mathbb{Q}_{h_{2}}}$), where $L(s, \rho_{\mathbb{Q}_{h_{2}}})$ is the Artin $L$-function corresponding to a representation $\rho_{\mathbb{Q}_{h_{2}}}: \text{Gal}(M^{(r)}_{h_{2}}\slash \mathbb{Q}) \hookrightarrow S_{r}\to \text{GL}_{r-1}(\mathbb{C})$, and $M^{(r)}_{h_{2}}$ is the normal closure of $\mathbb{Q}_{h_{2}}$. Since $\nu\in L_{g}$ can be some element such that $L_{g}=\mathbb{Q}(\nu)\simeq \mathbb{Q}_{h_{2}}$, we again also wish to study the distribution of Artin $L$-functions $L(s, \rho_{\mathbb{Q}_{h_{2}}})$ arising from fields $\mathbb{Q}_{h_{2}}$ induced by irreducible polynomials $h_{2}$. To that end, we let $V^{(r)}(\mathbb{Z})^{\text{irr}}$ be the space consisting of irreducible monic integer polynomials  $h_{2}$ of fixed degree $r=n(p-1)^{(m+1)\ell}$,  and let $V^{(r)}(\mathbb{Z})^{\text{max}}\subset V^{(r)}(\mathbb{Z})^{\text{irr}}$ be a subset consisting of irreducible monic integer polynomials $h_{2}$ such that $R_{h_{2}}=\mathbb{Z}[x]\slash (h_{2}(x))$ is a maximal order in $\mathbb{Q}_{h_{2}}=\mathbb{Q}[x]\slash (h_{2}(x))$. As before, $G_{a}(\mathbb{Z})=\mathbb{Z}$ acts naturally on $V^{(r)}(\mathbb{Z})^{\text{irr}}$ via translation, namely, $(b \cdot h_{2})(x):= h_{2}(x+b)$ for every $b\in \mathbb{Z}$ and  every $h_{2}\in V^{(r)}(\mathbb{Z})^{\text{irr}}$; and moreover this action of $G_{a}(\mathbb{Z})$ by translation preserves each of $V^{(r)}(\mathbb{Z})^{\text{irr}}$ and $V^{(r)}(\mathbb{Z})^{\text{max}}$. So now, let $\mathfrak{F}^{(r)}_{2}$ be a family consisting of the $\mathbb{Z}$-orbits on $V^{(r)}(\mathbb{Z})^{\text{max}}$; and which parametrizes degree-$r$ monogenized number fields $(\mathbb{Q}_{h_{2}}, \nu)$ over $\mathbb{Q}$ up to isomorphism. As before, we also treat $\mathfrak{F}^{(r)}_{2}$ to be the family of associated $L$-functions $L(s, \rho_{\mathbb{Q}_{h_{2}}})$. Now by again taking great advantage of [\cite{Nico}, Thm.1.1], we then also obtain the following corollary: 

\begin{cor}
Assume second part of Theorem \ref{3.3}, and let $r=n(p-1)^{(m+1)\ell}$ be any fixed even integer. Let $\mathfrak{F}^{(r)}_{2}$ be a family of $\mathbb{Z}$-orbits defined as before. Then the family $\mathfrak{F}^{(r)}_{2}$ parametrizing monogenized degree-$r$ fields ordered by height $h(h_{2})$ as given in \textnormal{\cite{Nico}} satisfies Sato-Tate equidistribution in the sense of \textnormal{[\cite{Sar}, Conjecture 1]}. 
\end{cor} 

\begin{proof}
By applying a similar argument as in the Proof of Corollary \ref{14.3}, it then also follows immediately that the family $\mathfrak{F}^{(r)}_{2}$ satisfies  Sato-Tate equidistribution in the sense of \textnormal{[\cite{Sar}, Conjecture 1]} as also indeed needed.
\end{proof}

\section*{\textbf{Acknowledgments}}
I’m very grateful to Prof. Ilia Binder, Prof. Arul Shankar, and Prof. Jacob Tsimerman for everything. To Prof. Tsimerman, you are truly a grandmaster, and I'm very fortunate to have witnessed on July 23, 2026 your great power of determination and effort! Any opinions expressed in this article belong solely to me, Brian Kintu; and should never be taken as a reflection of the views of anyone that has been happily acknowledged by the author.

\bibliography{References}

@article {Poonen,
    AUTHOR = {Poonen, B.},
     TITLE = {The classification of rational preperiodic points of quadratic
              polynomials over {${\bf Q}$}: a refined conjecture},
   JOURNAL = {Math. Z.},
  FJOURNAL = {Mathematische Zeitschrift},
    VOLUME = {228},
      YEAR = {1998},
    NUMBER = {1},
     PAGES = {11--29},}

@article {Russo,
    AUTHOR = {Walde, R. and Russo, P.},
     TITLE = {Rational periodic points of the quadratic function
              {$Q_c(x)=x^2+c$}},
   JOURNAL = {Amer. Math. Monthly},
  FJOURNAL = {American Mathematical Monthly},
    VOLUME = {101},
      YEAR = {1994},
    NUMBER = {4},
     PAGES = {318--331},}

@book {Silverman,
    AUTHOR = {Silverman, J H.},
     TITLE = {The arithmetic of dynamical systems},
    SERIES = {Graduate Texts in Mathematics},
    VOLUME = {241},
 PUBLISHER = {Springer, New York},
      YEAR = {2007},
     PAGES = {x+511},
      ISBN = {978-0-387-69903-5},}

@article {Flynn,
    AUTHOR = {Flynn, E. V. and Poonen, Bjorn and Schaefer, Edward F.},
     TITLE = {Cycles of quadratic polynomials and rational points on a
              genus-{$2$} curve},
   JOURNAL = {Duke Math. J.},
  FJOURNAL = {Duke Mathematical Journal},
    VOLUME = {90},
      YEAR = {1997},
    NUMBER = {3},
     PAGES = {435--463},}

@article {Ingram,
    AUTHOR = {Hutz, B. and Ingram, P.},
     TITLE = {On {P}oonen's conjecture concerning rational preperiodic
              points of quadratic maps},
   JOURNAL = {Rocky Mountain J. Math.},
  FJOURNAL = {The Rocky Mountain Journal of Mathematics},
    VOLUME = {43},
      YEAR = {2013},
    NUMBER = {1},
     PAGES = {193--204},}

@article {par2,
    AUTHOR = {Panraksa, C.},
     TITLE = {Rational periodic points of $x^d + c$ and Fermat-Catalan equations},
   JOURNAL = {International Journal of Number Theory.},
  FJOURNAL = {World Scientific},
    VOLUME = {18},
      YEAR = {2022},
    NUMBER = {05},
     PAGES = {1111-1129},}

@article {Narkie,
    AUTHOR = {Narkiewicz, W.},
     TITLE = {On a class of monic binomials},
   JOURNAL = {Proc. Steklov Inst. Math.},
  FJOURNAL = {Proceedings of the Steklov Institute of Mathematics},
    VOLUME = {280},
      YEAR = {2013},
    NUMBER = {suppl. 2},
     PAGES = {S65--S70},}

@article {Doyle,
    AUTHOR = {Doyle, J R. and Faber, X. and Krumm, D.},
     TITLE = {Preperiodic points for quadratic polynomials over quadratic
              fields},
   JOURNAL = {New York J. Math.},
  FJOURNAL = {New York Journal of Mathematics},
    VOLUME = {20},
      YEAR = {2014},
     PAGES = {507--605},}

@article {North,
    AUTHOR = {Northcott, D. G.},
     TITLE = {Periodic points on an algebraic variety},
   JOURNAL = {Ann. of Math. (2)},
  FJOURNAL = {Annals of Mathematics. Second Series},
    VOLUME = {51},
      YEAR = {1950},
     PAGES = {167--177},}

@article {Morton,
    AUTHOR = {Morton, P. and Silverman, J H.},
     TITLE = {Rational periodic points of rational functions},
   JOURNAL = {Internat. Math. Res. Notices},
  FJOURNAL = {International Mathematics Research Notices},
      YEAR = {1994},
    NUMBER = {2},
     PAGES = {97--110},}

@article {Hutz,
    AUTHOR = {Hutz, B.},
     TITLE = {Determination of all rational preperiodic points for morphisms
              of {PN}},
   JOURNAL = {Math. Comp.},
  FJOURNAL = {Mathematics of Computation},
    VOLUME = {84},
      YEAR = {2015},
    NUMBER = {291},
     PAGES = {289--308},}

@article {Call,
    AUTHOR = {Call, G S. and Goldstine, S W.},
     TITLE = {Canonical heights on projective space},
   JOURNAL = {J. Number Theory},
  FJOURNAL = {Journal of Number Theory},
    VOLUME = {63},
      YEAR = {1997},
    NUMBER = {2},
     PAGES = {211--243},}

@book{par1,
    AUTHOR = {C. Panraska },
     TITLE = {Arithmetic dynamics of quadratic polynomials and dynamical units, Phd dissertation},
 PUBLISHER = { University of Maryland, College Park},
      YEAR = {(2011), pp. 1-42},}

@article {detto,
    AUTHOR = {Benedetto, R L.},
     TITLE = {Preperiodic points of polynomials over global fields},
   JOURNAL = {J. Reine Angew. Math.},
  FJOURNAL = {Journal f\"{u}r die Reine und Angewandte Mathematik. [Crelle's
              Journal]},
    VOLUME = {608},
      YEAR = {2007},
     PAGES = {123--153},}

@article {Doy,
    AUTHOR = {Doyle, John R.},
     TITLE = {Preperiodic points for quadratic polynomials with small cycles over quadratic fields},
   JOURNAL = {Math. Z},
  FJOURNAL = {Springer Link},
    VOLUME = {289},
      YEAR = {2018},
    NUMBER = {1-2},
     PAGES = {729--786},}

@article{Sil,
    AUTHOR = {J. Silverman},
     TITLE = {Integer points, Diophantine approximation, and iteration of rational maps},
   JOURNAL = {Duke Mathematical Journal},
   VOLUME  = {71.3:793-829},
     YEAR  = {1993},}

@article{sch1,
    AUTHOR = {Bhargava, M. and Shankar, A. and Wang, X.},
     TITLE = {Squarefree values of polynomial discriminants {I}},
   JOURNAL = {Invent. Math.},
    VOLUME = {Vol. 228},
      YEAR = {(2022), pp. 1-37},}

@article{lem,
    AUTHOR = {F. Thorne, R. J. Lemke Oliver and},
     TITLE = {Upper bounds on number fields of given degree and bounded discriminant},
 PUBLISHER = {Duke Mathematical Journal},
    VOLUME = {Duke Mathematical Journal, Vol. 171, No. 15},
      YEAR = {(2022), pp. 1-11},}

@book{BK1,
    AUTHOR = {B. Kintu},
     TITLE = {Counting the number of integral fixed points of a discrete dynamical system with applications from arithmetic statistics, I},
 PUBLISHER = {https://arxiv.org/pdf/2501.04026},
      YEAR = {pp. 1-14},}

@book{BK22,
    AUTHOR = {B. Kintu},
     TITLE = {Counting the number of $m$-periodic $\mathcal{O}_{K}$-points of a discrete dynamical system with applications from arithmetic statistics, V},
 PUBLISHER = {https://arxiv.org/pdf/2508.16393},
      YEAR = {pp. 1-27},
}

@book{BK111,
    AUTHOR = {B. Kintu},
     TITLE = {Counting the number of $n$-periodic integral points of a discrete dynamical system with applications from arithmetic statistics, IV},
 PUBLISHER = {https://arxiv.org/pdf/2507.08601},
      YEAR = {pp. 1-18},}

@book {Dev,
    AUTHOR = {Devaney, Robert L.},
     TITLE = {An introduction to chaotic dynamical systems},
    SERIES = {Addison-Wesley Studies in Nonlinearity},
   EDITION = {Second},
 PUBLISHER = {Addison-Wesley Publishing Company, Advanced Book Program,
              Redwood City, CA},
      YEAR = {1989},
     PAGES = {xviii+336},
      ISBN = {0-201-13046-7},}

@article {Narkie1,
    AUTHOR = {Narkiewicz, W.},
     TITLE = {Cycle-lengths of a class of monic binomials},
   JOURNAL = {Functiones et Approximatio},
  FJOURNAL = {Proceedings of the Steklov Institute of Mathematics},
    VOLUME = {42.2},
      YEAR = {(2010), 163-168},
    NUMBER = {suppl. 2},
     PAGES = {S65--S70},}

@book{BK2,
    AUTHOR = {B. Kintu},
     TITLE = {Counting the number of $\mathcal{O}_{K}$-fixed points of a discrete dynamical system with applications from arithmetic statistics, II},
 PUBLISHER = {https://arxiv.org/abs/2503.11393},
      YEAR = {pp. 1-16},}

@book{BK3,
    AUTHOR = {Kintu, B.},
     TITLE = {Counting the number of $\mathbb{Z}_{p}$- and $\mathbb{F}_{p}[t]$-fixed points of a discrete dynamical system with applications from arithmetic statistics, III},
 PUBLISHER = {https://arxiv.org/pdf/2505.24565},
      YEAR = {pp. 1-25},}

@book{BK333,
    AUTHOR = {Kintu, B},
     TITLE = {Counting the number of $1_{n}$-preperiodic $\mathbb{Z}_{p}$- and $\mathbb{F}_{p}[t]$-points of a discrete dynamical system with applications from arithmetic statistics, VIII},
 PUBLISHER = {https://doi.org/10.5281/zenodo.22923779},
      YEAR = {pp. 1-46},}

@article{ho,
    AUTHOR = {Ho, W. and Shankar, A. and Varma, I.},
     TITLE = {Odd degree number fields with odd class number},
   JOURNAL = {Duke Math. Journal.},
    VOLUME = {Vol. 167(5)},
      YEAR = {(2018), pp. 1-53},}

@book{Sia,
    AUTHOR = {Siad, A.},
     TITLE = {Monogenic fields with odd class number Part {II}: even degree},
 PUBLISHER = {https://arxiv.org/pdf/2011.08842},
      YEAR = {pp. 1-49},}

@article{Kat,
    AUTHOR = {A. Katok and B. Hasselblatt},
     TITLE = {Introduction to the Modern Theory of Dynamical Systems},
   JOURNAL = {Cambridge University Press},
   VOLUME  = {Vol. 54},
      YEAR = {1995},}

@article{Nico,
    AUTHOR = {Shankar, A. and S\"{o}dergren, A. and Templier, N.},
     TITLE = {Sato-$\text{T}$ate equidistribution of certain families of $\text{A}$rtin $\textit{L}$-functions},
   JOURNAL = {Forum of Mathematics, Sigma (2019)},
    VOLUME = {Vol.7, e23, 62 pages},
}

@article{Sar,
    AUTHOR = {Sarnak, P. and Shin, S.W. and Templier, N.},
     TITLE = {Families of $\textit{L}$-functions and their symmetry},
   JOURNAL = {Proceedings of Simons Symposia, Families of Automorphic Forms and the Trace Formula},
    VOLUME = {(Springer Verlag, 2016), 531-578},
}

@article{gav1,
    AUTHOR = {Bhargava, M.},
     TITLE = {Galois groups of random integer polynomials and van der Waerden's Conjecture},
   JOURNAL = {Ann. of Math.},
    VOLUME = {201},
      YEAR = {(2025), 339–377},
}

@article{AM,
    AUTHOR = {Artin, M. and B. Mazur},
     TITLE = {On periodic points},
   JOURNAL = {Ann. of Math.},
    VOLUME = {81 (1)},
      YEAR = {(1965), pp. 82–99},
}

@book{Bou,
    AUTHOR = {P. Boudec and N. M. Mavraki},
     TITLE = {Arithmetic dynamics of random polynomials},
 PUBLISHER = {https://arxiv.org/abs/2112.12005},
      YEAR = {pp. 1-27},
}

@article{Ole,
    AUTHOR = {Olechnowicz, M.},
     TITLE = {Distribution of preperiodic points in one-parameter families of rational maps},
   JOURNAL = {Trans. Amer. Math. Soc},
   VOLUME  = {},
     YEAR  = {(2026), pp. 1-43},}

@article{Loo,
    AUTHOR = {Looper, N.R.},
     TITLE = {Dynamical uniform boundedness and the \text{abc}-conjecture},
   JOURNAL = {Invent. Math.},
    VOLUME = {Vol. 225},
      YEAR = {(2021), pp. 1-44},}

@article{DoyPo,
    AUTHOR = {Doyle, J. and B. Poonen},
     TITLE = {Gonality of dynatomic curves and strong uniform boundedness of preperiodic points},
   JOURNAL = {Compos. Math.},
    VOLUME = {156},
      YEAR = {(2020), 733-743},
}

@article{Gart,
    AUTHOR = {Andersen, A. and D. Garton},
     TITLE = {Preperiodic points of polynomial dynamical systems over finite fields},
   JOURNAL = {Internat. Journal of Number Theory.},
  FJOURNAL = {World Scientific},
    VOLUME = {20},
      YEAR = {2024},
    NUMBER = {09},
     PAGES = {2307-2316},
}

@article{Gart2,
    AUTHOR = {Garton, D},
     TITLE = {Periodic points of rational functions over finite fields},
   JOURNAL = {Trans. of London Math. Soc.},
  FJOURNAL = {},
    VOLUME = {12},
      YEAR = {2025},
    NUMBER = {e70007},
     PAGES = {pp. 1-14},
}

@article{Sad1,
    AUTHOR = {Sadek, M.},
     TITLE = {Families of polynomials of every degree with no rational preperiodic points},
   JOURNAL = {Comptets Rendus Math\'{e}matique},
    VOLUME = {Vol. 359, 2},
      YEAR = {(2021), p. 195-197},
}

@book{Sad2,
    AUTHOR = {Sadek, M.},
     TITLE = {On rational periodic points of $x^d + c$},
 PUBLISHER = {https://arxiv.org/pdf/1804.09839},
      YEAR = {pp. 1-14},
}

@book{BK3rm,
    AUTHOR = {B. Kintu},
     TITLE = {Counting the number of $2_{n}$-preperiodic integral points of a discrete dynamical system with applications from arithmetic statistics, IX},
 PUBLISHER = {https://doi.org/10.5281/zenodo.22946192},
      YEAR = {pp. 1-21},}

@book{BK3rm2,
    AUTHOR = {B. Kintu},
     TITLE = {Counting the number of $2_{n}$-preperiodic  $\mathcal{O}_{K}$-, $\mathbb{Z}_{p}$- and $\mathbb{F}_{p}[t]$-points of a discrete dynamical system with applications from arithmetic statistics, X and XI},
 PUBLISHER = {In preparation to appear on zenodo.org},
      YEAR = {},}
\bibliographystyle{plain}

\noindent University of Toronto, Ontario, Canada \newline
Brian Kintu, founder of \textbf{Determination in Action} and \textbf{Powered by Foundations}\newline
\textit{E-mail address:} \textbf{brian.kintu@mail.utoronto.ca}\newline 
\date{\small{\textit{October 9, 2026}}}

\end{document}